\documentclass[10.6pt, reqno]{amsart}
\usepackage{tabular}
\usepackage{caption}
\usepackage{graphicx}
\usepackage{cancel}
\usepackage{mathtools}
\usepackage{subfigure}
\usepackage{slashed}
\usepackage{amsfonts}
\usepackage{amssymb}
\usepackage{accents}
\usepackage{amsmath}
\usepackage{amsxtra}
\usepackage{epstopdf}
\usepackage{latexsym}
\usepackage{mathrsfs}
\usepackage{leftidx}
\usepackage{tensor}
\usepackage{enumitem}

\newtheorem{theorem}{Theorem}[section]
\newtheorem{lemma}[theorem]{Lemma}
\newtheorem{corollary}[theorem]{Corollary}
\newtheorem{proposition}[theorem]{Proposition}
\newtheorem{definition}[theorem]{Definition}

\theoremstyle{remark}
\newtheorem{remark}[theorem]{Remark}

\numberwithin{equation}{section}

\def\uda#1{\underaccent{#1}}
\def\ude#1{\underaccent{e}#1}
\def\nabe{\underaccent{e}{\nab}}
\def\ne{\underaccent{e}N}
\def\sf#1{\stackrel{\frown}{#1}}
\def\sJ{\mathscr{J}}
\def\sR{\mathscr{R}}
\def\rp#1{^{\!(#1)}}
\def \up#1{{}^{(#1)}\!}
\def\sPs{{}^\star \Psi}

\def\pih{\hat{\pi}}
\def\pt {{}^{(\bT)}\pi}

\def\bj{{\bf j}}
\def\bi{{\bf i}}
\def\dn{{\Delta_0}}
\def\fm{\mathfrak{m}^2}

\def\pr{{}^{(\sR)}\pi}
\def\px{\pi\rp{X}}

\def\phr{{}^{(\sR)}\pih}

\def\Hb{\underline{H}}
\def\Eb{\underline{E}}
\def\loc{{\mbox{\textit{loc}}}}
\def\Eeb{\underline{\E}}

\def\gb{\underline{g}}

\def\T{\mathcal{T}}

\def\bb{{\mathbf{b}}}

\def\sn{{\slashed{\nabla}}}
\def\Q{\mathcal{Q}}
\def\sQ{\mathscr{Q}}

\def\ab{{\underline{\a}}}

\def\zb{{\underline{\zeta}}}

\def\J{{\mathcal{J}}}

\def\Ab{{\underline{A}}}

\def\M{{\mathcal{M}}}
\def\bT{{\textbf{T}}}

\def\bR{{\textbf{R}}}

\def\bd{{\textbf{D}}}
\def\ti{\tilde}

\def\bg{\mathbf{g}}

\def\hk{{\hat{k}}}
\def\I{{\mathcal I}}

\def\beaa{\begin{eqnarray*}}
\def\eeaa{\end{eqnarray*}}

\def\ba{\begin{array}}
\def\ea{\end{array}}
\def\d{\delta}
\def\be#1{\begin{equation} \label{#1}}
\def \eeq{\end{equation}}

\newcommand{\nn}{\nonumber}

\def\l{\langle}
\def\r{\rangle}

\def\pih{\hat{\pi}}
\def\cir{\overset\circ}
\def\nn{\nonumber}
\def\S{{\mathcal S}}

\def\ud#1{\underline{#1}}
\def\zb{\ud{Z}}

\def\S2{{\mathbb S}^2}

\def\E{{\mathcal E}}

\def\K{{\mathcal{K}}}

\def\W{{\mathcal W}}

\def\Ephin{\E_\phi}
\def\Ephi#1{\E^{(#1)}_\phi}
\def\Edphin{\E_{d\phi}}
\def\Edphi#1{\E^{(#1)}_{d\phi}}
\def\ub{\underline{u}}
\def\Lb{\underline{L}}
\def\Lie{{\mathcal L}}
\def\hLie{\hat{\mathcal{L}}}
\def\tr{\mbox{tr}}

\def\H{{\mathcal H}}

\def\c{\cdot}
\def\hot{\widehat{\otimes}}
\def\sig{\sigma}
\def\s{\sigma}
\def\a{\alpha}
\def\b{\beta}

\def\ep{{\epsilon}}
\def\ve{{{\textbf{$\varepsilon$}}}}
\def\l{\langle}
\def\r{\rangle}
\def\ga{\gamma}
\def\Ga{\Gamma}

\def\O{\mathcal{O}}

\def\p{\partial}

\def\nab{\nabla}

\def\Kb{\underline{\K}}
\def\Wb{\underline{\W}}

\def\C{{\mathcal C}}

\def\Lb{{\underline{L}}}

\def\div{\mbox{\,div\,}}
\def\curl{\mbox{\,curl\,}}

\def\tr{\mbox{tr}}
\def\Tr{\mbox{Tr}}

\def\tir{{\tilde r}}

\def\itt{{\mbox{Int}}}

\def\f14{\frac{1}{4}}
\def\f12{{\frac{1}{2}}}

\def\t1a{t^{-\frac{1}{a}}}

\def\bm{{\bf m}}

\def\sl{\slashed}

\def\sn{{\slashed{\nabla}}}

\def\ab{{\underline{\a}}}

\def\zb{{\underline{\zeta}}}

\def\J{{\mathcal{J}}}

\def\Ab{{\underline{A}}}

\def\M{{\mathcal{M}}}
\def\bT{{\emph{\bf{T}}}}

\def\bR{{\emph{\bf{R}}}}

\def\bd{{\emph{\bf{D}}}}
\def\ti{\tilde}

\def\hk{{\hat{k}}}
\def\I{{\mathcal I}}

\def\beaa{\begin{eqnarray*}}
\def\eeaa{\end{eqnarray*}}

\def\ba{\begin{array}}
\def\ea{\end{array}}
\def\be#1{\begin{equation} \label{#1}}
\def \eeq{\end{equation}}
\def\nn{\nonumber}

\def\l{\langle}
\def\r{\rangle}

\def\pih{\hat{\pi}}
\def\cir{\overset\circ}
\def\nn{\nonumber}
\def\S{{\mathcal S}}

\def\S2{{\mathbb S}^2}

\def\bo{\mathbf{o}}

\def\E{{\mathcal E}}

\def\ub{\underline{u}}
\def\udb{\underline{\b}}
\def\Lb{\underline{L}}
\def\tr{\mbox{tr}}

\def\H{{\mathcal H}}

\def\Nb{{\underline{N}}}

\def\Nba{{\acute{\Nb}}}

\def\c{\cdot}
\def\hot{\widehat{\otimes}}
\def\sig{\sigma}

\def\s{\sigma}
\def\a{\alpha}
\def\b{\beta}

\def\l{\langle}
\def\r{\rangle}
\def\ga{\gamma}
\def\Ga{\Gamma}
\def\la{\lambda}

\def\p{\partial}

\def\nab{\nabla}
\def\nabb{\underline{\nabla}}

\def\Lb{{\underline{L}}}

\def\div{\mbox{\,div\,}}
\def\curl{\mbox{\,curl\,}}

\def\tr{\mbox{tr}}
\def\Tr{\mbox{Tr}}

\def\tir{{\tilde r}}

\def\itt{{\mbox{Int}}}

\def\f14{\frac{1}{4}}
\def\f12{{\frac{1}{2}}}

\def\t1a{t^{-\frac{1}{a}}}

\def\bm{{\bf m}}

\def\sl{\slashed}

\def\ckk{\check}

\def\dum{\mbox{ }}
\newcommand{\bea}{\begin{eqnarray}}
\newcommand{\eea}{\end{eqnarray}}

\def\nn{\nonumber}

\def\hLW#1{\hLie^{{(#1)}}_{\sR}W}

\newcommand{\chih}{\hat{\chi}}
\newcommand{\chib}{\underline{\chi}}

\newcommand{\chibh}{\underline{\hat{\chi}}\,}
\newcommand{\les}{\lesssim}

\def\fB{\mathfrak{B}}
\def\bN{{\mathbf{N}}}
\def\S{\mathcal{S}}

\def\bt{\l t\r}

\def\brho{\l \rho\r}
\def\cir#1{\stackrel{\circ}{#1}}

\def\gs{{\bg_s}}

\def\tf{{t_\flat}}

\def\rf{{r_\flat}}

\def\sE{\mathscr{E}}

\def\bI{{\mathbf{I}}}
\def\bJ{{\mathbf{J}}}

\def\sQ{{\mathscr{Q}}}

\def\spE{{}^{\!\sharp} \E}
\def\ud#1{\underline{#1}}

\def\dfa{{}^\dag\mathfrak{a}}
\def\pib{\pmb{\pi}}

\begin{document}
\title[]{An intrinsic hyperboloid approach for \\Einstein Klein-Gordon equations}
\author{Qian Wang}
\address{
Oxford PDE center, Mathematical Institute, University of Oxford, Oxford, OX2 6GG, UK}
  \email{qian.wang@maths.ox.ac.uk}
  \date{\today}
\begin{abstract}
In \cite{KKG} Klainerman introduced the hyperboloidal method to prove the global existence
results for nonlinear Klein-Gordon equations by using commuting vector fields. In this paper,
we extend the hyperboloidal method from Minkowski space to Lorentzian spacetimes. This approach
is developed in \cite{Wang15} for proving, under the maximal foliation gauge,  the global nonlinear stability of
Minkowski space for Einstein equations with  massive scalar fields, which  states that,
 the sufficiently small  data in a compact domain, surrounded by a Schwarzschild metric, leads to a unique,
 globally hyperbolic, smooth and geodesically complete solution to the Einstein Klein-Gordon system.

In this paper, we set up the geometric framework of the intrinsic hyperboloid approach in the curved spacetime.
By performing a thorough geometric comparison between the radial normal vector field induced by the intrinsic
hyperboloids and the canonical $\p_r$, we manage to control the hyperboloids when they are close to their
asymptote, which is a light cone in the Schwarzschild zone. By using such geometric information,  we not only
obtain the crucial boundary information for running the energy method in \cite{Wang15}, but also  prove that
the intrinsic geometric quantities including the Hawking mass all converge to their Schwarzschild values when
approaching the asymptote.
\end{abstract}
\maketitle

\section{\bf Introduction}

We introduce the intrinsic hyperboloid approach in the dynamic, Lorentzian spacetime. This approach is
developed in \cite{Wang15} to prove, under the maximal foliation gauge,  the global stability of Minkowski
space for Einstein equations with  massive scalar fields, which reads as
 \begin{equation*}
\bR_{\mu\nu}-\frac{1}{2}\bg_{\mu\nu} \bR=\T_{\mu\nu}
\end{equation*}
with the stress-energy tensor\begin{footnote}{ We fix the convention that, in the Einstein summation convention, a Greek letter
is used for index taking values $0, 1, 2, 3$ and a little Latin letter is used for index taking values $1,2,3$.}\end{footnote}
\begin{equation*}
\T_{\mu\nu}=\p_\mu \phi\p_\nu \phi-\f12 \bg_{\mu\nu} \left(\bg^{\a\b}\p_\a \phi\p_\b \phi+\fm\phi^2\right), \quad\quad \mathfrak{m}=1,
\end{equation*}
where $\bR_{\mu\nu}$ and $\bR$ denote the Ricci curvature tensor and the scalar curvature of the Lorenzian metric $\bg$
respectively. Applying the conservation law $\bd^\mu\T_{\mu\nu}=0$, which is due to the Bianchi identity, gives the
Einstein Klein-Gordon system
\begin{align}
\bR_{\a\b}&=\p_\a \phi \c \p_\b \phi+\f12 \, \bg_{\a\b} \phi^2, \label{ricci}\\
\Box_\bg \phi&=\fm\phi.\label{gkg}
\end{align}
It is obvious that $({\Bbb R}^{3+1},\bm, \phi\equiv0)$, with $\bm$ being Minkowski, trivially solves the system.
To construct nontrivial global solutions of (\ref{ricci})-(\ref{gkg}), it is natural to consider the Cauchy
problems with the initial data set being small perturbations of the trivial one.

We first briefly review the framework for studying the Cauchy problem of the Einstein equations.
Let $(\mathbf{M}, \bg)$  be globally hyperbolic which means that there is a Cauchy hypersurface,
which is a spacelike hypersurface with the property that any causal curve intersects it at precisely
one point. This allows  $\mathbf{M}$ to be foliated by the level surfaces $\Sigma_t$ of a time function $t$.
Let $\bT$ be the future directed unit normal to $\Sigma_t$. Let $\pib$ be the second fundamental form
of $\Sigma_t$ in $\mathbf{M}$ defined by
\begin{equation}\label{6.6.2.16}
\pib(X, Z): =-{\bf g} ({\bf D}_X {\bf T}, Z), \quad X, Z\in \T\Sigma_t,
\end{equation}
where ${\bf D}$ denotes the covariant differentiation of  $\bf{g}$ in $\mathbf{M}$.

Let $g$ be the induced metric of $\bg$ on $\Sigma_t$. We decompose
$$
\p_t =n \bT +Y,
$$
where $n$ is the lapse function and $Y\in \T \Sigma_t$ is the shift vector field.
Assuming $Y=0$, then the metric $\bg$ can be written as
\begin{equation}\label{gg}
\bg=-n^2 d t^2+g_{ij}d x^i d x^j,
\end{equation}
and the Einstein equations are equivalent to
the evolution equations
\begin{align}
& \partial_t g_{ij}=-2n {\pib}_{ij}, \label{bg}\\
& \partial_t \pib_{ij}=-\nabla_i \nabla_j n+n(-\bR_{ij}+R_{ij}+\mbox{Tr} \pib\,  \pib_{ij}-2
\pib_{ia}\pib^a_j) \label{intro04}
\end{align}
together with the constraint equations
\begin{align}
&R-|\pib|^2+(\mbox{Tr} \pib)^2=2 \bR_{\bT\bT}+\bR, 
\qquad \nabla^j \pib_{ji}-\nabla_i \mbox{Tr} \pib=\bR_{\bT i}, \label{intr.2}
\end{align}
where $\Tr \pib:=g^{ij} \pib_{ij}$ is the mean curvature of $\Sigma_t$ in ${\bf M}$, $\nabla$ denotes
the covariant differentiation of $g$,  $R_{ij}$ and $R$ are the Ricci curvature and the scalar curvature
of $g$ on $\Sigma_t$.

The maximal foliation gauge imposes
\begin{equation}\label{5.13.3.16}
Y=0 \quad \mbox{ and } \quad \Tr \pib=0 \quad \mbox{ on } \Sigma_t.
\end{equation}
This implies $n$ satisfies the elliptic equation
\begin{equation}\label{lapse}
\Delta_g n-|\pib|^2 n = n \bR_{\bT\bT}\qquad \mbox{on } \Sigma_t,
\end{equation}
and the second fundamental form $\pib$ satisfies the Codazzi equation
\begin{equation}\label{codazziT}
\div \pib=\bR_{\bT i}, \quad \curl \pib=H,
\end{equation}
where $H$ is  the magnetic part of the Weyl curvature, defined in (\ref{7.23.8}).

The first proof of the global stability of Minkowski spacetime for generic, asymptotically flat data
is provided in the monumental work \cite{CK}, where the Einstein vacuum Bianchi equation is thoroughly
and systematically treated. Heuristically, we regard the nonlinear wave equation verifying the standard
null condition as the vastly simplified model for the  Einsteinian Bianchi equation. Then (\ref{ricci})-(\ref{gkg})
is a coupled system  between such nonlinear wave equations and the Klein-Gordon equation in the Einsteinian
background. Due to the presence of the massive scalar field, the approach we introduce in this paper
is to twist the hyperboloidal energy method devised in the flat spacetime in \cite{KKG} for the
Klein-Gordon equations to the Lorentzian spacetime, in the sense of incorporating it to the intrinsic
energy scheme devised in \cite{CK}. Such generalization triggers fundamental changes to the geometry
of the intrinsic framework in \cite{CK} for the Einstein equations, which  by itself is very challenging
even merely for the vacuum case. Our approach is robust for treating both the scalar field and the
Einstein part of the equation system. This will be fully confirmed in \cite{Wang15}.

In what follows, we will use the linear Klein-Gordon equation to motivate the use of the intrinsic hyperboloids.
To begin with, let us recall some basics of the invariant vector fields  for the free
wave $\Box_\bm \phi=0$.\begin{footnote}{We assume the initial data for $\phi$  have compact support.}\end{footnote}
We denote by $Z$ a set of vector fields, which consists of the translation  $\p_\mu, $ the scaling vector
field $S=x^\mu\p_\mu$ and the generator of Lorentz group
\begin{equation}\label{3.30.6.16}
\Omega_{\mu\nu}=x_\mu\p_\nu-x_\nu\p_\mu, \quad \mu, \nu=0, 1,2,3 \mbox{ where }x_\mu=\bm_{\mu\nu}x^\nu.
\end{equation} This set of vector fields is named as commuting vector fields  due to
the fundamental property
\begin{equation}\label{3.30.3.16}
[\Box_\bm, Z]=0 \mbox{ or } 2\Box_\bm
\end{equation}
with the second identity   occurring only when  $Z=S$.

In order to get the decay estimate $(t+1)|\phi|\les 1$ by the energy approach, we rely on two ingredients:
one is the boundedness of the energy or the generalized energy; the other is the Klainerman-Sobolev inequality.

The standard Klainerman-Sobolev inequality
\begin{equation}\label{3.30.5.16}
\bt (1+|t-r|)^\f12 | f|\les \|Z^{(\le 2)}f(t,\cdot)\|_{L^2({\Bbb R}^3) }
\end{equation}
relies on the full set of $Z$ derivatives, where $r=(\sum_{i=1}^3 |x^i|^2)^{1/2}$, $\bt=t+1$ and
$Z^{(\le m)} f$ denotes the application of the differential operators in $Z$ to $f$ up to $m$ times.
For the free wave equation  $\Box_\bm \phi=0$, by using $\p_t$ as a multiplier, one can obtain
the conserved energy
\begin{equation}\label{6.4.2.16}
\|\p_t \phi(t,\cdot)\|^2_{L^2({\mathbb R}^3)}+\sum_{i=1}^3\|\p_{x^i}\phi(t,\cdot)\|^2_{L^2({\mathbb R}^3)}.
\end{equation}
By using the canonical Morawetz vector field, $K=2t S-(t^2-r^2) \p_t, $ as a multiplier, one can obtain the
conserved generalized energy, which is uniformly comparable for $t \ge 0$ to
\begin{equation}\label{3.30.4.16}
\|Z\phi(t,\cdot)\|^2_{L^2({\Bbb R}^3)}+\|\phi(t, \cdot)\|^2_{L^2({\Bbb R}^3)}.
\end{equation}
In view of (\ref{3.30.3.16}), one can see that (\ref{6.4.2.16}) and (\ref{3.30.4.16}) hold with $\phi$ replaced
by\begin{footnote}{For a differential operator $P$, we use $P\rp{n} f$ to mean the $m$-time application of $P$
to $f$.}\end{footnote} $Z\rp{n}\phi$. These estimates together with (\ref{3.30.5.16}) imply that
 \begin{equation}\label{6.4.4.16}
 \bt (1+|t-r|)^\f12 |\phi|\les \|Z\rp{\le 2}\phi(0,\cdot)\|_{L^2({\Bbb R}^3)}\les 1
\end{equation}
which gives more information for $\phi$ than desired.

To see the difference in the treatment for the Klein-Gordon equation, we consider the linear Klein-Gordon
equation
\begin{equation}\label{3.30.2.16}
\Box_\bm \phi= \phi
\end{equation}
in the Minkowski spacetime.  Due to (\ref{3.30.3.16}) there holds  $[(\Box_\bm-1), S]=2 \Box_{\bm} \neq 2 (\Box_{\bm}-1)$.
Thus the scaling vector field $S$ can not be used as a commuting vector field for (\ref{3.30.2.16}).
Similar to (\ref{6.4.2.16}), we can obtain the conserved energy
\begin{equation*}
\|\p_t \phi(t,\cdot)\|^2_{L^2({\mathbb R}^3)}+\sum_{i=1}^3\|\p_{x^i}\phi(t,\cdot)\|^2_{L^2({\mathbb R}^3)}
+\|\phi(t,\cdot)\|^2_{L^2({\mathbb R}^3)}
\end{equation*}
which stays conserved if $\phi$ is replaced by $Z\rp{n}\phi$ except $Z=S$. In contrast to the case of the free wave,
the boundedness of energy does not hold for the full set of the commuting vector fields in $Z$. The Klainerman-Sobolev
inequality (\ref{3.30.5.16}) can not be used directly. To get the decay estimates for the Klein-Gordon equations,
in \cite{KKG} the Klainerman-Sobolev inequality is applied on the canonical hyperboloids
${\Bbb H}_\rho=\{t^2-\sum_{i=1}^3 |x^i|^2=\rho^2\}$ which are the surfaces orthogonal to $S$. The Klainerman-Sobolev
inequality on hyperboloids merely relies on the Lorentz boosts $\{R_i=t \p_{x^i}+x^i \p_t, i=1,2,3\}$ which are commuting
vector fields of (\ref{3.30.2.16}) and tangent to the hyperboloids. By virtue of this tool, the standard sharp decay
estimate\begin{footnote}{The area element of $\mathbb{H}_\rho$ is $d \mu_{\mathbb{H}_\rho}=\frac{\rho}{t} d\mu_{\mathbb{R}^3}$.}\end{footnote}
\begin{equation}\label{3.30.7.16}
\bt^\frac{3}{2}| \phi|\les \left\|\left(\frac{t}{\rho}\right)^\f12 R^{(\le 2)}\phi(\rho,\cdot)\right\|_{L^2(\mathbb{H}_{\rho})}
\les \left\|\p\rp{\le 1}  R^{(\le 2)}\phi(0,\cdot)\right\|_{L^2(\mathbb{R}^3)}.
\end{equation}
can be derived from the boundedness of energies on hyperboloids. Thus, in order to get the sharp decay for the
solutions of (\ref{3.30.2.16}), the same order of commuting vector fields are applied and energies have to
be controlled up to one order higher compared with the free wave case. This coincides with the case when
we treat  Klein-Gordon equation (\ref{gkg}) coupled with the Einstein Bianchi equations, for which (\ref{3.30.2.16})
and the free wave are the simplest toy models for each part.

We also observe that the two weighted multipliers, $S$ and $K$, can not be used to obtain bounded generalized
energy for (\ref{3.30.2.16}). This fact together with the fact that the scaling $S$ is not a commuting vector
field for (\ref{3.30.2.16}), demonstrates that decomposing  $\p\phi$ in terms of the null frame
$\{L=\p_t+\p_r, \Lb=\p_t-\p_r\}$ does not improve the decay along the good direction $L$. This is another
difference compared with the free wave. Contributed by the commuting vector fields $R$, $\p\phi$ exhibits much
stronger decay along the tangential directions of ${\Bbb H}_\rho$;  however, $\p\phi$ has the weakest decay along  $B:=\frac{t}{\rho}\p_t+\frac{r}{\rho}\p_r=\rho^{-1}S$, the future directed unit normal of ${\Bbb H}_{\rho}$.
The weakest decay is much weaker than that a free wave exhibits along its only bad direction $\Lb$.

Figure \ref{fig1}(a) depicts the method in \cite{KKG}, where the data with compact support in $\{r\le R\}$
are given at $t=0$. The energy argument is divided into two steps. The first step is the local energy
propagation from $t=0$ to the initial hyperboloidal slice ${\Bbb H}_{\rho_0}$, with $\rho_0\approx 1$.
The second step is to propagate energy on hyperboloids ${\Bbb H}_{\rho}$ from ${\Bbb H}_{\rho_0}$ to
the last slice ${\Bbb H}_{\rho_*}$, in the region  enclosed by a Minkowskian light cone as the boundary,
along which the solution varnishes due to finite speed of propagation. This figure gives us the blue-print
of treating the Einstein-Klein-Gordon system.
\begin{figure}[h]
\subfigure[]{\includegraphics[width = 0.49\textwidth, height= 1.6in]{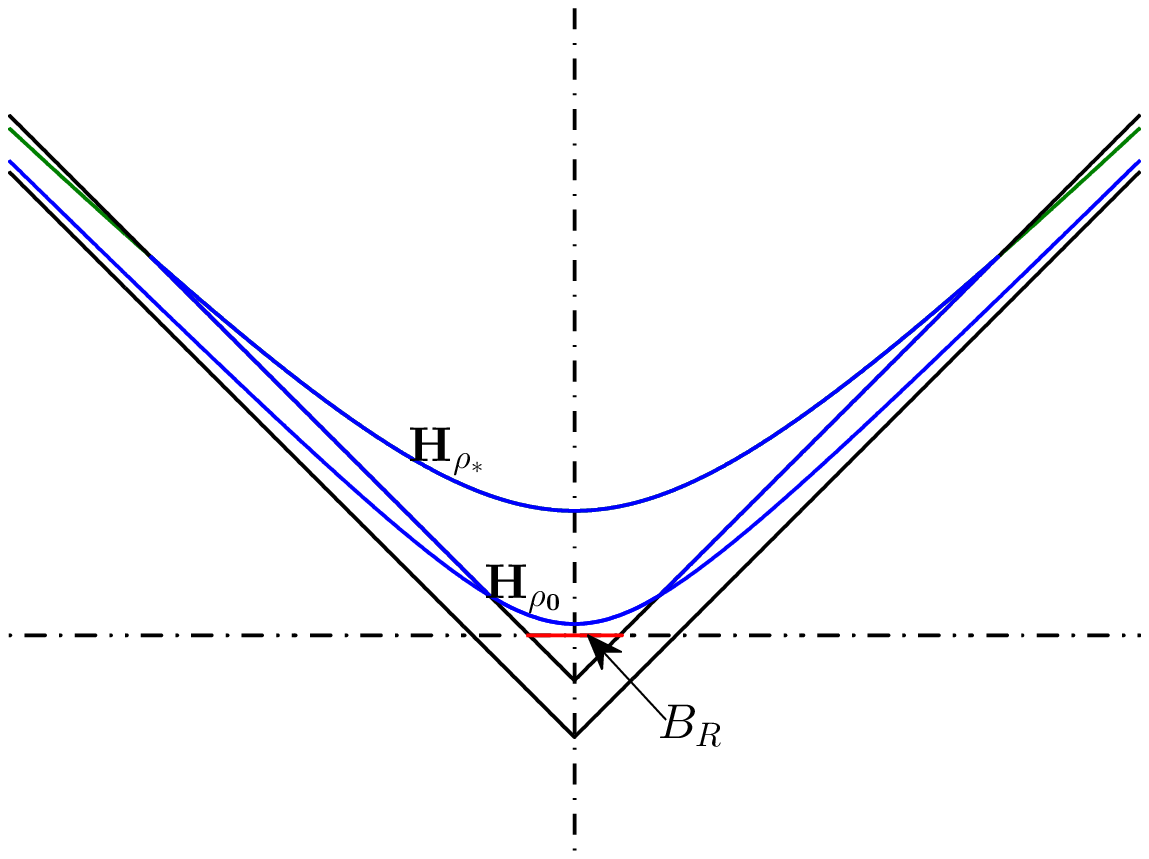}}
\subfigure[]{\includegraphics[width = 0.49\textwidth, height= 1.6in]{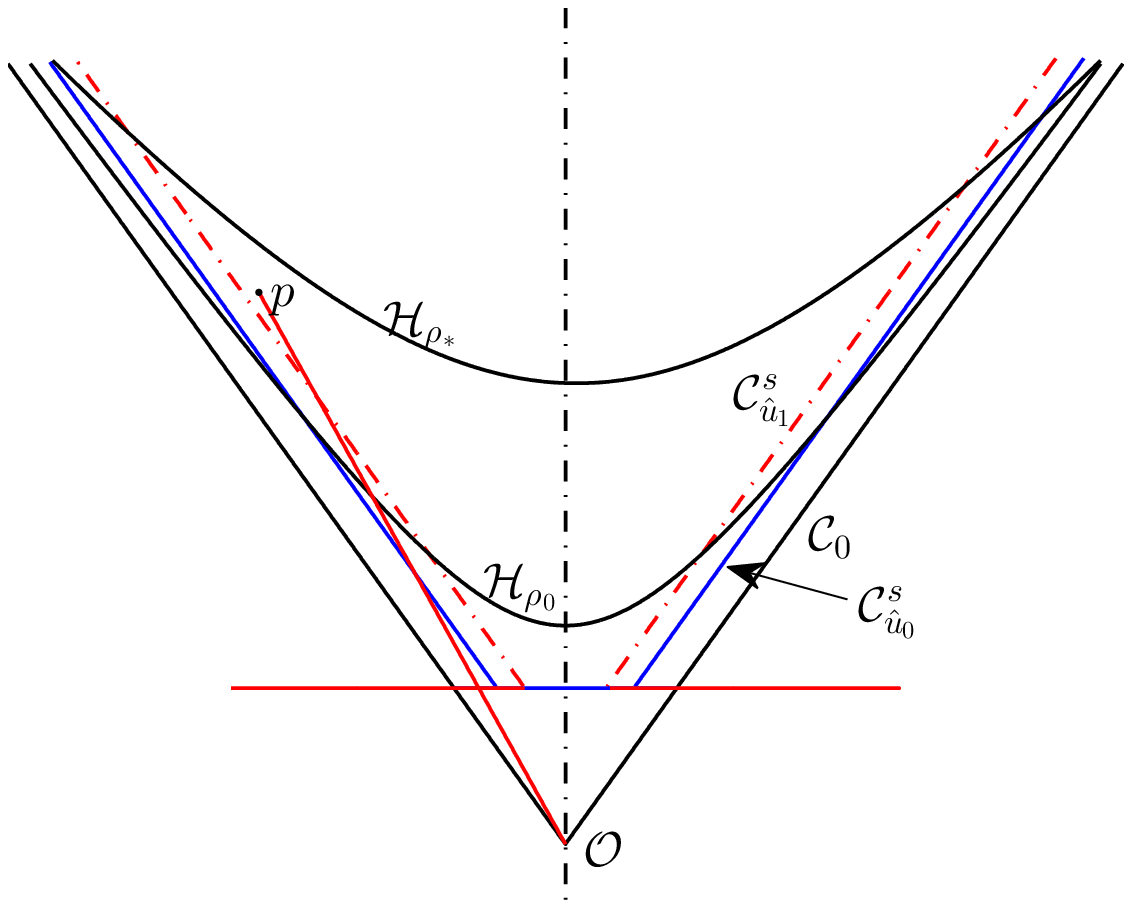}}
\vskip -0.5cm
\caption{}\label{fig1}
\end{figure}

In order to set up the Cauchy problem for the Einstein-Klein-Gordon system (\ref{ricci})-(\ref{gkg})
appropriately, to match with, in particular, the scenario that the data for Klein-Gordon equation have
compact support, we consider the initial data set $(g_0, \pib_0,\phi[0])$ for (\ref{ricci})-(\ref{gkg}),
which verify the Einstein constraint equations (\ref{intr.2}) and $\phi[0]$ is compactly supported
within $B_1$, the unit Euclidian ball. Outside of the co-centered Euclidean ball of radius $2$, there
glues a surrounding Schwarzschild metric specified at Theorem \ref{2.15.1}. See Figure \ref{fig1}(b).
We will call the region with $t\ge 0$, exterior to the Schwarzschild outgoing light cone $\C^s_{\hat u_1}$
as the Schwarzschild zone $Z^s$, where $\C^s_{\hat u_1}$ initiates from the Euclidean sphere $\{r=2\}$
with the value of $\hat u_1$ specified in Section \ref{8.07.1}. We still  need to determine the foliation
of hyperboloids in the curved spacetime.

There are two options at this point. One way is based
entirely on the symmetry and geometry in Minkowski space. This  method has been developed in \cite{LRCMP} and \cite{LRANA}
for the Einstein equations under the wave coordinates. The philosophy of the regime is to  close the
energy argument without aiming at achieving sharp decay for geometric quantities.  This allows  the stability result to be achieved within a much smaller framework compared with \cite{CK}.  However it is less
precise on the asymptotic behavior of the solution (see \cite[Page 47]{LRCMP}).

In this paper and \cite{Wang15}, we take the other option which constructs intrinsic hyperboloids adapted
to the curved spacetimes. We not only prove the global nonlinear stability, but also give a comprehensive,
analytic, global-in-time depiction of the solution.
The goal of this paper is to introduce the geometric framework, which equips the nonlinear analysis with
sets of tetrads, recovering the symmetry and playing the role of coordinates, all of which are adapted
to the dynamical spacetime.  The global existence of such tetrads will be justified simultaneously with
the quantitative depiction of the spacetime.

When setting up the geometric framework, it is necessary to discriminate, among all the symmetry in
the Minkowski space, the most crucial geometric information that needs precision from those allowing
error to be controlled analytically. For this purpose, we run a simple energy argument for
\begin{equation}\label{3.30.8.16}
\Box_\bg\phi=\phi
\end{equation}
by taking the approach  as in \cite{LRCMP}, that is to  consider
\begin{equation}\label{3.30.9.16}
\Box_\bg R\rp{n}\phi=R\rp{n}\phi+[\Box_\bg, R\rp{n}]\phi.
\end{equation}
The error integral contributed by one term contained in the commutator on the right hand side
of (\ref{3.30.9.16}) takes the form
\begin{align*}
&\int^\rho_1\int_{{\Bbb H}_{\rho'}} {}^{(R)}\pi^{\a\b}  \underbrace{R^{(n-1)}\bd^2_{\a\b}\phi}_{\in L^2}
\c\underbrace{ \bd_\bT R^{(n)} \phi }_{\in L^2}d\mu_{{\Bbb H}_{\rho'}} d\rho',
\end{align*}
where the deformation tensor is defined by
\begin{equation}\label{6.6.1.16}
{}^{(X)}\pi_{\a\b} =\bd_\a X_\b+\bd_\b X_\a
\end{equation}
with $\bd$ denoting the connection induced by the metric $\bg$. Here  $X=R$ is the Lorentz boost in Minkowski
space, ${}^{(R)}\pi^{\a\b}\neq 0$ since $\bg$ is not the Minkowski metric. With $\a,\b=B$ in ${}^{(R)}\pi^{\a\b}$,
the derivative of $\phi$ contracted with this term is evaluated at the bad direction $B$ along which it does
not decay strongly. Under local coordinates the expression of ${}^{(R)}\pi^{BB}$ contains $B$ derivatives of the metric $\bg$,
paired with large weights. The best decay for ${}^{(R)}\pi^{BB}$ expected by the approach from \cite{LRCMP} is below
the borderline  $C\ep/\rho$ for applying Gronwall inequality to control energies.  To salvage the energy
argument, we construct a set of approximate Lorentz boosts $\sR$ and the intrinsic hyperboloids $\H_\rho$,
adapted to the Einstein background, so that ${}^{(\sR)}\pi^{\fB\fB}=0$, where $\fB$ denotes the unit normal of
the constructed foliation of hyperboloids. These $\sR$ and $\fB$ can be viewed as the corresponding
replacements of $R$ and $B$ in the curved spacetime. The construction of these $\sR$ and $\fB$ needs to
preserve the following features:
\begin{enumerate}[leftmargin=0.7cm]
\item  $[\fB, \sR]=0$  and $\sR=\{\sR_i, i =1,2,3\}$ are tangent to $\H_\rho$.
\item $\cup_{\rho=0}^\infty\H_\rho$ exhausts the chronological  future of an origin $\O$, with the origin to be
appropriately chosen. All the hyperboloids are asymptotic to the outgoing light cone $\C_0$ emanating from $\O$.
\end{enumerate}

The origin $\O$ is chosen at $t=-T$, which can be done due to the local extension of the solution.
We may choose $T$ so that $\C_0$ intersects $\{t=0\}$ outside of $B_4$. The freedom of such choice is
fixed in Section 2, which is crucial for the proof of the main results of this paper.
We leave the details of the constructions to  Section 2-3. See Proposition \ref{2.20.18.16}
for using the first feature to prove  ${}^{(\sR)}\pi_{\fB\fB}=0$ and more results on ${}^{(\sR)}\pi$.

The geometric constructions equip us with the approximate Lorentz boost, scaling, and translation
vector fields. With them we can run the commuting vector field approach to the Einstein Bianchi equations, which
can be viewed as an extension of the approach in \cite{CK}, where the regime is based on the construction
of the rotation vector fields and the intrinsic null cones. The task, in our situation, is  much more
involved, due to the difficulties caused by the massive scalar field, the geometry of the hyperboloids,
as well as the complexity of the analytic control on the Lorentz boosts.  In what follows we focus on
addressing the following two basic issues.

\begin{enumerate}[leftmargin = 0.7cm]
\item For the Weyl part of curvature, we will run the regime of Bel-Robinson energies defined by the
weighted multipliers $S$ and $K$. For closing the top order energy, we encounter the issue of requiring
higher order $\sR$-energy for the massive scalar field. However, in order to close the energy estimates,
we have to control the energies of the Weyl tensors and the massive scalar field, up to the same order
in terms of the $\sR$-derivatives.

\item  The intrinsic hyperboloids, in principle, are defined from the Minkowskian counterparts
by a global diffeomorphism, which needs to be justified simultaneously with the proof of the global
existence of the solution. In Minkowski space, the density of the foliation of the canonical hyperboloids
approaches infinity near the causal boundary. The control on the intrinsic foliation is considerably
more delicate since, analytically, terms of $1/\rho$ type appear frequently, with $\rho\rightarrow 0$
when approaching the causal boundary $\C_0$.
\end{enumerate}
To solve the first issue, it is crucial to use the Einstein Bianchi equations, see Lemma \ref{7.23.12}, which allows
us to perform the integration by parts when treating the worst type of terms. We then take advantage of the null forms
in the Einstein Bianchi equations, together with the expected
strong decay from the scalar  field. This enables the top order energies to be closed at a sharp level.
We will sketch briefly the energy scheme in Section 4.

The second issue is connected to  the set-up of the wave zone, the region where we run the energy estimates.
We have to take account of the gravitational influence to the causal structure of the foliation of the
intrinsic hyperboloids. In this paper, we focus on controlling the intrinsic geometry of the chronological future
$\I^+(\O)$ for $t\les 1$ and for all $t$ in the Schwarzschild zone.  This geometric control is significant
for dealing with the problem of leakage, for demonstrating that a constructed function is almost optical,
and for justifying an excision procedure on the wave zone for the energy scheme. These aspects will be
explained in the sequel.

In the Minkowskian set-up (Figure \ref{fig1}(a)),  a light cone is  used to  cut the family of hyperboloids,
as the boundary of the wave zone. The cone needs to be uniformly away from the asymptote. The set-up of such
boundary in the curved spacetime is more subtle. First of all, this boundary should be chosen in the Schwarzschild
region, to guarantee the dynamical part of the solution is contained in its interior. It ought to be a canonical
Schwarzschild light cone $\C^s_{\hat u_0}$\begin{footnote}{The value of $\hat u_0$  can be found in
Section \ref{8.07.1}.}\end{footnote} for the ease of analysing energy flux therein. More importantly, we need
a function measuring the ``distance" from any point in the entire wave zone to $\C_0$, which is the
asymptote of the hyperboloids. This nonnegative function needs to be bounded uniformly away from zero in the wave zone, for the purpose of running the energy estimates. This task intuitively could be achieved if $\C^s_{\hat u_0}$  is spaced away uniformly  from $\C_0$ in terms of a canonical optical function in $Z^s$.

In Lorentzian spacetime the light cones are usually characterized by the level sets of an optical function $u$ (see \cite{CK})
which is defined as the solution of the eikonal equation
\begin{equation}\label{5.8.1.16}
\bg^{\a\b}\p_\a u\p_\b u=0
\end{equation}
with prescribed boundary or initial conditions. Then the optical function $u$  naturally measures the distance to
the causal boundary. To obtain the information of $u$ would require geometric controls on the foliations of light cones $C_u$.
However, since such light cones are not involved in our analysis, we do not use the actual optical function.
In our framework, $\H_\rho$  conceptually replaces the role usually played by $C_u$. The geometric control on
the hyperboloids $\H_\rho$ lies at the core of our analysis. To achieve the desired analytic feature,
we choose $\rho$ to be the proper time to $\O$, where $\rho$ verifies the eikonal equation
$$
\bg^{\a\b}\p_\a \rho \p_\b \rho=-1, \quad \rho(\O)=0.
$$
Throughout the chronological future $(\I^+(\O), \bg)$, we define an alterative function, still denoted by $u$,
which does not verify (\ref{5.8.1.16}), yet taking the role of measuring the distance to $\C_0$.
In particular, we can show that this function $u$, vanishing on $\C_0$,  is sufficiently close to the
canonical optical function near $\C_0$ in $Z^s$. To show such property, we perform in Section \ref{almost}
a full analytic comparison between the radial normal of the Schwarzschild frame and the normal vector
field on $\Sigma_t$, induced by the foliation of the intrinsic hyperboloid $\cup_{\rho}S_{t,\rho}$.
The main estimates are established in Theorem \ref{4.10.6} throughout $Z^s\cap \{\rho\le \rho_*\}$, which
is the major building block of this paper. These estimates and their higher order counterparts will be
used in the main energy scheme in \cite{Wang15}.

Next we address the issue of the leakage. Let $p$ be a point inside the wave zone, near the
boundary $\C^s_{\hat u_0}$. The distance maximizing timelike geodesic connecting $p$ and $\O$ is not
entirely contained in the wave zone. See Figure \ref{fig1}(b).  This phenomenon can be easily seen
in Minkowski space. In  Minkowskian case the deformation tensors of the boosts vanish and the deformation
tensor of the scaling vector field has a standard value. However, in the dynamical spacetime, deformation
tensors ${}^{(\sR)}\pi$ and ${}^{(S)}\pi$ need to be analyzed, which is done by integrating along the
aforementioned time-like  geodesics with the help of  the structure equations which contain both the
curvature components under the hyperboloidal frame and the second fundamental forms; see Section \ref{5.10.1.16}.
Whether the path of the integration is contained in the wave zone determines how to control the integrand.
The geometric information outside of the wave zone can not be provided by the energy estimates. Such
information is obtained  simultaneously with the main estimates in Section 5 by geometric comparisons
and bootstrap arguments.

Now we explain, as part of the energy scheme, the excision of a region which is related to the so-called
last slice of hyperboloids, denoted by $\H_{\rho_*}$. As a standard method for proving global results
of non-linear dynamical problem, one can suppose a set of bootstrap assumptions hold till certain maximal
life-span. Due to various concerns, we set the maximal life-span in terms of the proper time, labeled by
$\rho_*$. Once the bootstrap assumptions can be improved for all $\rho\le \rho_*$, by the principle
of continuation, the solution and the quantitative control can be extended beyond $\rho_*$.  The wave zone
is a region which is enclosed by the initial slice $\{t=0\}$, the last slice $\H_{\rho_*}$ as well as the
cone $\C^s_{\hat u_0}$. Consider the energy estimates on $\Sigma_t$, which are crucial for controlling $\pt$.
When $t\ge t_*$ where $t_*:=\min\{ t: S_{\rho_*, \hat u_0}\}$, we no longer expect a regular subset of $\Sigma_t$
within the wave zone to do the energy estimates (see Figure \ref{fig4} in Section \ref{engsch}). The subset of
wave zone with $t\ge t_*$ will be excised for obtaining the $\Sigma_t$-energy. This may lead to the loss of
control of $\pt$ in a region with $t$ large within the wave zone, which would fail the energy control
on $\H_\rho$. Our strategy is to show that the region of excision is fully contained in $Z^s$, where $\pt$
and other geometric quantities can be controlled by the main estimates. This proof has to be done merely
depending on local energy estimate, and the assumption that the foliation of $\H_\rho$ exists up
to $\rho\le \rho_*$, which is the case in this paper (see Section 6).

As the other application of the main estimates, we show that the Hawking mass is convergent to the ADM mass
of the surrounding Schwarzschild metric along every hyperboloid.

Finally, we comment on the analysis of the intrinsic geometry in $Z^s$. This analysis is independent of the
long-time energy estimates in the wave zone. The idea is to use the transport equations to perform the long-time
geometric comparison. We define a set of quantities which encode the deviation between the intrinsic and the extrinsic
tetrads, and derive the transport equations for them along well-chosen paths. In order to prove the
function $u$ is almost optical, we uncover a series of cancelations, contributed by the Schwarzschild metric
and the structure equations of the hyerboloidal foliation. It necessitates
delicate bootstrap arguments and weighted estimates\begin{footnote}{ The primitive version of such weighted
estimates can be seen in  \cite{Wangthesis}.}\end{footnote}. The obtained main estimates are crucial for
the applications in Section 6-7.

The paper is organized as follows. In Sections 2-3 we carefully set up the analytic framework of the
foliation of intrinsic hyperboloids, and provide the geometric construction of the intrinsic frame of the
Lorentz boosts, since such set-up and construction have never appeared in the literature. In Section \ref{engsch}
we sketch the energy scheme  in the proof  of global stability of Minkowski space for (\ref{ricci})-(\ref{gkg}).
In Section \ref{almost}, by assuming the foliation of the intrinsic hyperboloids and the maximal foliation
exist till the last slice of hyperboloid, we provide a thorough depiction of the intrinsic geometry
in the Schwarzschild zone, presented in Theorem \ref{4.10.6}, as the main estimates  of this paper. The
region considered there is the most sensitive region for having the geometric control on hyperboloids.
The set of main estimates depends merely on local-in-time energy estimates and the smallness of the given data
on the initial maximal slice. We then give applications of the main estimates. The one in Section \ref{wza}
is to control the region of the excision.  In Section \ref{mass}, we give the asymptotic behavior of
the Hawking mass along all hyperboloids.

\section{\bf Construction of the boost vector fields}\label{sec_c}

By standard energy and iteration argument, we first solve  the Cauchy problem of EKG back to
the past to certain fixed $t\le -T$. Let $\bo$ be the spacial origin of the given initial slice.
We denote by $\Ga(t)$ the geodesic through $\bo$ with velocity $-\bT$, where $\bT$ is the future-directed
time-like unit normal of the initial slice $\Sigma_0$.  The geodesic is extended (back-in-time)
within the radius of injectivity of $\bo$, intersecting $\{t=-T\}$   at $\O$. $T$ is chosen so that
the given Cauchy data at the initial slice is fully contained in $\sJ^+(\O)\cap \{t=0\}$, where $\sJ^+(\O)$
denotes the causal future of $\O$.   Hence $T$ depends on the size of the support of Cauchy data, and is
comparable to $1$. To be more precise, $T$ is chosen  such that $\C_0$ intersects at $t=0$ outside of $B_4$.
Now by the shift of  $t'= t+T$, as well as an abuse of notation, $t=0$ at $\O$ and the initial data is
prescribed at $t=T$, according to the time coordinate after the shift.

We use $\I^+(\O)$ to denote the chronological future of $\O$.
Let $i_*$ be the time-like radius of injectivity of $\O$, which is defined to be the supremum
over all the values $s_0>0$ for which the exponential map
\begin{equation}\label{12.29.4}
\exp_\O: (\rho, V)\to \Upsilon_V(\rho),  \quad V\in \mathbb{H}_1
\end{equation}
is a global diffeomorphism from $(0, s_0)\times \mathbb{H}_1$ to its
image in $\I^+(\O)$, where
$$
\mathbb{H}_1:= \left\{ V=(V^0, V^1, V^2, V^3): (V^0)^2 - \sum_{i=1}^3 (V^i)^2 =1\right\}
$$
is the canonical hyperboloid in ${\mathbb R}^{3+1}$ and $\Upsilon_V$ is the time-like geodesic with
$\Upsilon_V(0)=\O$ and $\Upsilon_V'(0) = V$. We use $\I_*^+(\O)$ to denote the part of $\I^+(\O)$ within the time-like
radius of injectivity. In \cite{Wang15} we will prove that the time-like radius of injectivity is $+\infty$
simultaneously when we prove the global well-posedness for EKG, provided the Cauchy data is sufficiently small.
Thus we will have $\I_*^+(\O) = \I^+(\O)$ once this result is established.

For a point $p$ in $\I^+(\O)$, we use $\rho$ to denote its geodesic distance to $\O$. Then $\rho$ is a smooth function
on $\I_*^+(\O)$ satisfying $\l \bd \rho, \bd \rho\r = -1$ with $\rho(\O)=0$.
We introduce the vector field
\begin{equation}\label{fb1}
\fB = -\bd \rho.
\end{equation}
Then $\fB$ is geodesic, i.e. $\bd_{\fB} \fB =0$ and $\l \fB, \fB\r = -1$. Using this $\fB$ we
define the lapse function $\bb$ by
\begin{equation}\label{eq_1}
\l \fB, \bT\r = -\bb^{-1}\frac{t}{\rho}
\end{equation}
Let
$$
\H_\rho:=\exp_{\O}(\rho {\mathbb H}_1).
$$
Clearly $\{\H_\rho\}$ are the level sets of $\rho$ which give a foliation of $\I_*^+(\O)$ in terms
of hyperboloids. Moreover, by the Gauss lemma we can see that $\fB$ is the future directed normal
to $\H_\rho$ and
\begin{equation}\label{2016.5.10}
\fB_p = \left(d \exp_\O\right)_{\rho V} (\p_\rho)
\end{equation}
for any $p\in \I_*^+(\O)$, where $(\rho, V)$ is the unique point in $(0, i_*) \times {\mathbb H}_1$
such that $p=\exp_{\O}(\rho V)$.

Using $\fB$ we may introduce the second fundament form $k$ of $\H_\rho$ defined by
\begin{equation*}
k(X, Y) =\l\bd_X \fB, Y\r
\end{equation*}
where $X, Y$ are vector fields tangent to $\H_\rho$. Clearly $k$ is an $\H_\rho$ tangent,
symmetric $(0,2)$ tensor. We will use $\tr k$ and $\hk$ to denote the trace and traceless part of $k$
respectively.
\begin{footnote}{ In general, for a $\H_\rho$-tangent symmetric 2-tensor $F$, with $\gb$ the induced metric on $\H_\rho$,
its trace and traceless part can be defined by $\tr F=\gb^{ij} F_{ij}$ and
$\hat F_{ij}=F_{ij}-\frac{1}{3}\tr F \gb_{ij}$ respectively.}\end{footnote}

According to the expression of $\bg$, we can derive that the future directed unit normal $\bT$ of $\Sigma_t$
takes the form
\begin{equation}\label{gt}
\bT = n^{-1} \p_t.
\end{equation}
This together with $\bd t = -n^{-2} \p_t$ and (\ref{eq_1}) implies that
\begin{equation}\label{bt}
\fB(t)=\l \fB, \bd t\r = \frac{\bb^{-1}n^{-1} t}{\rho}.
\end{equation}
For future reference, we set
\begin{equation}\label{4.5.1}
 \tf:=(\bb^{-1}t)(\Ga(t));\quad \rf=\sqrt{\tf^2-\rho^2}.
\end{equation}

According to the definition of $\I_*^+(\O)$, for any $p\in \I_*^+(\O)$ there corresponds a unique
$(\rho, V)\in (0, i_*) \times {\mathbb H}_1$ with $V = (V^0, V^1, V^2, V^3)$ such that
\begin{equation}\label{coord1}
p=\exp_{\O}(\rho V).
\end{equation}
We set
\begin{equation}\label{n_coord}
y^0=\tau:=\rho\sqrt{1+\sum_{i=1}^3 (V^i)^2} \quad \mbox{and} \quad y^i=\rho V^i \mbox{ for } i=1, 2, 3.
\end{equation}
Then $\{y^\mu, \mu=0,\cdots 3\}$ gives the geodesic normal coordinates for $\I_*^+(\O)$.


\begin{lemma}\label{inida}
For any $V\in {\mathbb H}_1$ there hold
\begin{equation}\label{init}
\lim_{\rho\rightarrow 0} \frac{\tau}{t}(\rho V)=n(\O),\quad
\lim_{\rho\rightarrow 0} \frac{\bb\tau}{t}(\rho V)=1,\quad
\lim_{\rho\rightarrow 0} \bb^{-1}(\rho V)=n(\O).
\end{equation}
\end{lemma}

\begin{proof}
By using (\ref{eq_1}) we can consider the local expansion of $\bb^{-1}t=\f12 \bd_\mu (\rho^2)\bT^\mu$ at $\O$
as follows
\begin{align*}
\bb^{-1}t&=\left.\bb^{-1}t\right|_\O+ \left.\f12\bd_\nu(\bd_\mu (\rho^2) \bT^\mu)\right|_\O \rho V^\nu +O(\tau^2)\\
& = \left.-(\bg_{\mu\nu}\bT^\mu)\right|_\O\rho V^\nu- \left.\pt_{\mu\nu}\right|_\O\rho^2 V^\mu V^\nu+O(\tau^2)\\
& = \tau+O(\tau^2),
\end{align*}
where we employed  \cite[Page 50]{Eric}  to get $\left.\f12 \bd_\nu \bd_\mu(\rho^2)\right|_\O=-\bg_{\mu\nu}(\O)$.
This implies the second identity in (\ref{init}). Similarly, for the function $n^{-1} t= \f12 \bT(t^2)$ we have
the local expansion
\begin{align*} 
n^{-1}t&= \left.\f12\bd_\nu (\bT(t^2))\right|_\O \rho V^\nu +O(\tau^2)
= \bd_\nu (t\bT^\a \bd_\a t)\big|_\O \rho V^\nu+O(\tau^2)
\end{align*}
Note that
\begin{align*}
\bd_\mu(t\bT^\a \bd_\a t) &=t(\bd_\mu \bT^\a \bd_\a t+\bT^\a \bd_\mu\bd_\a t)+\bT^\a \bd_\a t \bd_\mu t,
\end{align*}
which, in view of (\ref{gt}), implies that
\begin{equation*} 
\bd_\mu(t\bT^\a \bd_\a t)\big|_\O = n^{-1}\bd_\mu t\big|_\O=-(n^{-2}\bg_{\mu\b}\bT^\b)\big|_\O.
\end{equation*}
Therefore we can obtain $n^{-1} t=n^{-2}(\O) \tau+O(\tau^2)$ which gives the first identity in (\ref{init})
as $\rho\rightarrow 0$. The last identity follows as a consequence of the first two.
\end{proof}

\subsection{Construction of the boost vector fields}

Recall that in Minskowski space, in terms of the geodesic coordinates introduced by (\ref{n_coord}), the boost vector
fields are defined by
\begin{equation}\label{3.18.11}
\cir{\sR}_i=y^i\p_\tau+\tau\p_i, \qquad i=1,2,3.
\end{equation}
Note that $\rho= \sqrt{\tau^2 -\sum_{i=1}^3 (y^i)^2}$ and $\p_\rho = \frac{1}{\rho} (\tau \p_\tau + y^i \p_i)$.
It is straightforward to show that
\begin{align}\label{lse}
[\p_\rho, \cir{\sR}_i]=0, \qquad i =1, 2, 3.
\end{align}
By using the exponential map to lift vector fields, this leads to introduce boost vector fields
\begin{equation}\label{boost}
\sR_i := (d\exp_\O)_{\rho V}(\cir{\sR_i}), \qquad i =1, 2, 3.
\end{equation}
defined on $\I_*^+(\O)$.

\begin{lemma}\label{4.1.1.16}
The boost vector fields $\sR_i$, $i=1,2,3$ are tangent to $\H_\rho$ and
\begin{equation}\label{lse3}
[\fB,\sR_i]=0, \qquad  \fB(\tau)=\frac{\tau}{\rho}.
\end{equation}
\end{lemma}

\begin{proof}
Since $\cir{\sR}_i$ are tangent to ${\mathbb H}_\rho:=\rho {\mathbb H}_1$ in the Minkowski spacetime,
by the definition of $\H_\rho$ and $\sR_i$, we can conclude that $\sR_i$ are tangent to $\H_\rho$.
In view of (\ref{2016.5.10}), (\ref{boost}) and (\ref{lse}) we have
$$
[\fB, \sR_i] = \left[ (d\exp_\O)_{\rho V} (\p_\rho), (d\exp_\O)_{\rho V}(\cir{\sR}_i)\right]
= (d\exp_\O)_{\rho V} \left(\left[\p_\rho, \cir{\sR}_i\right]\right) = 0.
$$
From the definition of $\tau$ we can obtain $\fB(\tau)=\tau/\rho$ by direct calculation.
\end{proof}

\begin{proposition}\label{2.20.18.16}
Let $\sR$ denote one of the boost vector fields $\sR_i$, $i=1,2,3$. Then
\begin{equation}\label{2.20.17.16}
\Lie\rp{n}_\sR \pr(\fB, \fB)=0, \quad \Lie\rp{n}_\sR\pr(\fB, \sR_i)=0.
\end{equation}
\end{proposition}

\begin{proof}
We prove (\ref{2.20.17.16}) by induction.
 First, we consider $n=0$. By using the first identity in (\ref{lse3}), we can obtain
\begin{equation*}
\pr(\fB, \fB)=2\l \bd_\fB\sR, \fB\r=2\l \bd_\sR \fB, \fB\r=0
\end{equation*}
and
\begin{align*}
\pr(\fB, \sR_i)&=\l\bd_\fB \sR, \sR_i\r+\l\bd_{\sR_i}\sR,\fB\r = \l\bd_\sR \fB, \sR_i\r-\l \bd_{\sR_i}\fB, \sR\r\\
&=k(\sR, \sR_i)-k(\sR_i, \sR)=0.
\end{align*}
Now consider $n\ge 1$. For a symmetric $(0,2)$ tensor $F$, suppose
\begin{equation*} 
F(\fB, \fB)=0,  \qquad F(\fB, \sR_i)=0,
\end{equation*}
we can obtain from the first equality in (\ref{lse3}) that
\begin{equation}\label{2.20.20.16}
\Lie_\sR F(\fB, \fB)= \sR (F(\fB, \fB))-2F(\Lie_\sR\fB, \fB)=0
\end{equation}
and
\begin{equation}\label{2.20.21.16}
\Lie_\sR F(\fB, \sR_i)=\sR (F(\fB, \sR_i))-F(\Lie_\sR \fB, \sR_i)-F(\fB, \Lie_\sR  \sR_i)=0.
\end{equation}
Since each $\Lie\rp{n}_\sR \pr$ is still a symmetric $(0,2)$, $\H_\rho$-tangent tensor,
which can be regarded as $F$, then (\ref{2.20.20.16}) and (\ref{2.20.21.16}) imply that
(\ref{2.20.17.16}) holds for $n+1$. Thus the proof of Proposition \ref{2.20.18.16} is
complete by induction.
\end{proof}

\section{\bf Intrinsic  hyperboloids}


We will use $g$ to denote the induced metrics on $\Sigma_t$ and
let $\nabla$ be the covariant differentiation. It is known that
$$
\nabla^\mu = \Pi_{\nu \eta} \bg^{\mu\eta} \bd^\nu,
$$
where
$$
\Pi_{\nu\eta} = \bg_{\nu\eta} + \bT_\nu \bT_\eta
$$
denote the tensor of projection to $\Sigma_t$.

Let $S_{t, \rho}:= \Sigma_t \cap \H_\rho$. Then for fixed $t$, $\{S_{t, \rho}\}_\rho$ gives a foliation
of $\Sigma_t$. Let $\gamma$ be the induced metric on $S_{t, \rho}$ and let $\sn$ be the associated covariant
differentiation. Since $\fB$ is normal to $S_{t, \rho}$, we have
\begin{equation}\label{split}
g = a^2 d\rho^2+\ga_{AB} d\omega^A d\omega^B
\end{equation}
where $a$ is the lapse function given by $a^{-1} =|\nabla \rho|_g$. By using $\l \bd \rho, \bd \rho\r = -1$
and (\ref{eq_1}) we have
\begin{align*}
-1=\bg^{\mu\nu}\p_\mu \rho\p_\nu\rho &=-(\bT(\rho))^2 + |\nabla \rho|_g^2
=-\frac{\bb^{-2} t^2}{\rho^2}+ |\nabla \rho|_g^2. 
\end{align*}
This shows that $\rho\le \bb^{-1}t$ on $\Sigma_t$ and the lapse function $a$ is given by
\begin{equation*} 
a^{-2}=|\nabla \rho|_g^2 = \frac{(\bb^{-1} t)^2 -\rho^2}{\rho^2}.
\end{equation*}
Therefore $a^{-1} = \frac{\tir}{\rho}$, where
\begin{equation*} 
\tir =\sqrt{\bb^{-2}t^2-\rho^2}.
\end{equation*}
Let $\bN$ denote the outward unit normal of $S_{t,\rho}$ in $\Sigma_t$. Then, according to (\ref{split})
we have
\begin{equation}\label{4.1.2.16}
\bN=-\frac{\nab\rho}{|\nab \rho|_g}=-a^{-1} \p_\rho  \quad \mbox{ on } \Sigma_t.
\end{equation}

Similarly let $\gb$ be the induced metric on $\H_\rho$ and let $\nabb$ be the corresponding covariant
differentiation. Then
$$
\nabb^\mu = \bar \Pi_{\nu\eta} \bg^{\mu\eta} \bd^\nu,
$$
where $\bar \Pi$ denotes the tensor of projection to $\H_\rho$ given by
$$
\bar \Pi_{\nu\eta} = \bg_{\nu\eta} + \fB_\nu \fB_\eta.
$$
Note that for fixed $\rho$, $\{S_{t, \rho}\}_t$ gives the radial foliation of $\H_\rho$. Since $\bT$ is
normal to $S_{t, \rho}$, we have
\begin{equation}\label{split1}
\gb = |\nabb t|_{\gb}^{-2} dt^2 + \ga_{AB} d\omega_A d\omega_B = (an)^2 dt^2+\ga_{AB} d\omega_A d\omega_B
\end{equation}
where for the second equality we used
\begin{equation}\label{nabt}
|\nabb t|_{\gb} = (an)^{-1}.
\end{equation}
The equation (\ref{nabt}) follows from the fact
\begin{equation*}
-n^{-2} = \l \bd t, \bd t\r = \bg^{\mu\nu} \p_\mu t\p_\nu t = -(\fB(t))^2 + |\nabb t|_\bg^2
=-\left(\frac{\bb^{-1}n^{-1} t}{\rho}\right)^2 + |\nabb t|_\bg^2
\end{equation*}
which also shows that $t\ge \bb\rho$ on $\H_\rho$. Let $\Nb$ denote the outward normal vector field of
$S_{t, \rho}$ in $\H_\rho$. Then
\begin{equation}\label{nb}
\Nb = \frac{\nabb t}{|\nabb t|_{\gb}} = an\nabb t.
\end{equation}

According to (\ref{split}) and (\ref{split1}), the volume form $d\mu_g$ on $\Sigma_t$ and the volume form
$d\mu_{\bg}$ on $\H_\rho$ are given respectively by
\begin{equation*}
d \mu_g=a\sqrt{|\ga|}d\rho d\omega, \qquad
d \mu_{\gb}=an\sqrt{|\ga|} dt d\omega.
\end{equation*}


\subsection{Decomposition of frames}\label{sec_frm}
Using $\bT$ and $\bN$ we define
\begin{equation}\label{nf}
L=\bT+\bN, \quad \Lb=\bT-\bN.
\end{equation}
It is easy to see that
\begin{equation*}
\l L, L\r=\l \Lb, \Lb\r=0, \qquad \l L, \Lb\r=-2.
\end{equation*}
Thus if $\{e_A, A=1,2\}$ is an orthonormal frame on $S_{t, \rho}$, then $\{L, \Lb, e_A, A=1,2\}$ form
a null frame.

We define a pair of functions
\begin{equation}\label{opti}
u:=\bb^{-1}t-\tir, \quad \ub:=\bb^{-1}t+\tir.
\end{equation}
which can be regarded as the counterparts for``$t-r, t+r$" in the Minkowski spacetime.
Due to the construction, there hold the two fundamental facts:\begin{footnote}{From now on, for convenience, $\I^+(\O)$ is understood to be $\I^+_*(\O)$.}\end{footnote}
\begin{enumerate}
\item $u> 0 $ in $\I^+(\O)$.  $u=0$ if and only if $\rho=0$, which holds only on $\C_0$,  the causal boundary of $\sJ^+(\O)$.
\item  Assuming $\bb^{-1}\ge C$ for some fixed constant $C>0$, \begin{footnote}{
This property can be found in  Proposition \ref{prl_1}, which can be quickly proved. }
\end{footnote} any $\H_\rho$ is asymptotically  approaching $\C_0$ as $t\rightarrow \infty$. This can be seen by using
\begin{equation}\label{2.21.1.16}
\rho^2=u \ub
\end{equation}
and $\ub\ge \bb^{-1}t $ in $\I^+(\O)$.
\end{enumerate}

\begin{lemma}\label{frames}
There hold
\begin{eqnarray}
\fB=\frac{b^{-1}t}{\rho}\bT+\frac{\tir}{\rho}\bN, && \Nb=\frac{\tir}{\rho} \bT+\frac{\bb^{-1} t}{\rho}\bN, \label{fbtn}\\
 2\rho\fB=\ub L+u\Lb,  && 2\rho \Nb=\ub L-u\Lb, \label{dcp_2}\\
 \rho \bT=\bb^{-1} t \fB-\tir \Nb,  && \rho \bN=\bb^{-1} t\Nb-\tir \fB. \label{dcp_3}
\end{eqnarray}
\end{lemma}
\begin{proof}
Since $\fB$ is normal to $S_{t, \rho}$, it can be decomposed using $\bT$ and $\bN$.
The component along $\bT$ follows directly from (\ref{eq_1}).
By using (\ref{eq_1}) and (\ref{4.1.2.16}) we have
\begin{align*}
-\l \fB, \nab\rho\r&=-\fB^\nu {\Pi}_{\mu\nu}\bd^\mu \rho= \fB^\nu\fB^\mu(\bg_{\mu\nu}+\bT_\mu \bT_\nu)
= \l \fB, \fB\r + \l \fB, \bT\r^2\\
&=-1+\l \fB, \bT\r^2=\frac{\bb^{-2}t^2-\rho^2}{\rho^2}=a^{-2}.
\end{align*}
This shows that $\l \fB, \bN\r = a^{-1} = \tir/\rho$ and hence the component along $\bN$ is obtained.
We therefore obtain the decomposition of $\fB$ in (\ref{fbtn}).

In view of (\ref{gt}), (\ref{bt}) and the decomposition of $\fB$, we have
\begin{align*}
(\nabb t)_\nu  &=\bar \Pi^{\mu\nu}\bd^\mu t=\bd_\nu t+\fB_\nu\fB(t)
  =-n^{-1}\bT_\nu +\frac{\bb^{-1} n^{-1} t}{\rho} \fB_\nu\\
&=n^{-1}\left(\frac{(\bb^{-2} t^2-\rho^2)}{\rho^2}\bT_\nu+\frac{a^{-1}\bb^{-1}t}{\rho}\bN_\nu\right)\\
& =n^{-1}\left(a^{-2}\bT_\nu+\frac{a^{-1}\bb^{-1}t}{\rho}\bN_\nu \right).
\end{align*}
This together with (\ref{nb}) shows the decomposition for $\Nb$ in (\ref{fbtn}).

By using  (\ref{nf}) and (\ref{opti}) we obtain (\ref{dcp_2}) from (\ref{fbtn}) directly.
 (\ref{dcp_3}) follows from (\ref{fbtn}) by a simple algebra.
\end{proof}

Recall the definitions (\ref{6.6.2.16}) and (\ref{6.6.1.16}). With $e_\bi, \bi=1,2,3$ the orthonormal basis on $\T \Sigma_t$,  there holds
\begin{equation}\label{6.6.3.16}
 \pt(e_\bi, e_\bj)=-2\pib(e_\bi, e_\bj), \quad \pt(\bT, e_\bi)=\nab_{\bi} \log n.
\end{equation}
\begin{lemma}
There hold
\begin{align}
&\l\bd_\fB \bT, \fB\r=-a^{-2} \pib_{\bN\bN}+\frac{\bb^{-1}t}{\rho}a^{-1}\l\bd_\bT \bT, \bN\r, \label{3.15.1} \displaybreak[0]\\
&\l \bd_\fB \bT, \Nb\r=\frac{\bb^{-2}t^2}{\rho^2}\l \bd_\bT \bT, \bN\r-\frac{\bb^{-1}t}{\rho} a^{-1} \pib_{\bN\bN}, \label{3.15.2} \displaybreak[0]\\
&\l \bd_\fB\bT, e_A\r= \frac{\bb^{-1}t}{\rho}\l \bd_\bT \bT, e_A\r-a^{-1} \pib_{\bN A}, \label{3.15.3} \displaybreak[0]\\
&\fB(\tir)=\frac{\tir}{\rho}+\bb^{-1}t \left(a^{-1} \pib_{\bN\bN}-\frac{\bb^{-1}t}{\rho}\l \bd_\bT \bT, \bN\r\right). \label{3.15.4}
\end{align}
\end{lemma}
\begin{proof}
In view of (\ref{fbtn}), we have
\begin{equation*}
\l \bd_\fB \bT, \fB\r=a^{-1}\left\l a^{-1} \bd_\bN \bT+\frac{\bb^{-1}t}{\rho}\bd_\bT \bT, \bN \right\r
=-a^{-2}\pib_{\bN \bN}+\frac{\bb^{-1}t}{\rho} a^{-1} \l \bd_\bT \bT, \bN\r.
\end{equation*}
Similarly, by using (\ref{fbtn}) we can (\ref{3.15.2}) and (\ref{3.15.3}).

To obtain (\ref{3.15.4}), we may use $\tir^2 =(\bb^{-1}t)^2-\rho^2$, (\ref{eq_1}) and $\bd_\fB \fB=0$ to derive that
\begin{align*}
2 \tir \fB(\tir) &= \fB(\tir^2) = 2 \bb^{-1}t \fB(\bb^{-1}t)-2 \rho \fB(\rho)=-2\bb^{-1}t\fB(\rho\l \bT,\fB\r)-2\rho\\
&=- 2\rho+2 \bb^{-1}t \left(-\rho \l \bd_{\fB} \bT, \fB\r +\frac{\bb^{-1} t}{\rho}\right).
\end{align*}
This shows that
\begin{equation*}
\fB(\tir)=\frac{\bb^{-2} t^2 -\rho^2}{\rho\tir }-\frac{\bb^{-1}t \rho}{\tir} \l \bd_{\fB} \bT, \fB\r
= \frac{\tir}{\rho}-\frac{\bb^{-1}t \rho}{\tir} \l \bd_{\fB} \bT, \fB\r.
\end{equation*}
In view of (\ref{3.15.1}), we therefore obtain (\ref{3.15.4}).
\end{proof}

\begin{lemma}
There hold
\begin{align}
\Nb(\bb^{-1}t)&=\rho \left(a^{-1} k_{\Nb \Nb}+\frac{\bb^{-1} t\tir }{\rho^2}\pib_{\bN\bN}
 -\frac{\tir^2}{\rho^2}\l \bd_\bT \bT, \bN\r \right), \label{3.5.1} \displaybreak[0]\\
t\sn (\bb^{-1})&=\tir \left(k_{A\Nb}+\pib_{A\bN} \right),\label{3.5.3}\\
\sn(\tir)&=\bb^{-1} t \left(k_{A\Nb}+\pib_{A\bN}\right),\label{3.6.4} \displaybreak[0]\\
\Nb(\tir) &=\frac{\bb^{-1}t}{\rho}+\bb^{-1}t \left(\hk_{\Nb\Nb}+\frac{1}{3}\left(\emph{\tr}\, k-\frac{3}{\rho}\right)\right)
  +\frac{\bb^{-2} t^2}{\rho}\pib_{\bN \bN} \nn \\
& \quad \, -\frac{\bb^{-1} t\tir}{\rho}\l \bd_\bT \bT, \bN\r. \label{6.26.8}
\end{align}
\end{lemma}

\begin{proof}
By using (\ref{eq_1}), (\ref{fbtn}) and (\ref{dcp_3}) we have
\begin{align*}
\rho^{-1}\Nb(\bb^{-1}t)& = - \Nb(\l \fB, \bT\r) = -\l \bd_{\Nb} \fB, \bT\r - \l\bd_{\Nb} \bT, \fB\r\\
&=a^{-1} \left(k_{\Nb \Nb}+\frac{\bb^{-1}t}{\rho}\pib_{\bN\bN}-a^{-1}\l \bd_\bT \bT, \bN\r\right),\\
\rho^{-1} \sn_A(\bb^{-1}t)&=-\l \sn_A \fB, \bT\r -\l \bd_A\bT, \fB\r \\
& =-\l \bd_A \fB,-a^{-1} \Nb\r - \l \bd_A \bT, a^{-1} \bN\r\\
& = a^{-1}\left(k_{A\Nb}+\pib_{A\bN} \right).
\end{align*}
We therefore obtain (\ref{3.5.1}) and (\ref{3.5.3}).
By using $\tir^2 = \bb^{-2} t^2 -\rho^2$ we have
$$
\tir \sn\tir = \bb^{-1} t^2 \sn(\bb^{-1}) \quad \mbox{ and } \quad
\tir \Nb(\tir) = \bb^{-1} t \Nb(\bb^{-1} t).
$$
These two equations together with (\ref{3.5.3}) and (\ref{3.5.1})
show (\ref{3.6.4}) and (\ref{6.26.8}); for deriving (\ref{6.26.8}) we also used the fact
$k_{\Nb\Nb} = \hat k_{\Nb\Nb} + \frac{1}{3}\tr k$.
\end{proof}

\begin{lemma}\label{4.3.6}
Let $\ckk k=k-\frac{1}{\rho}\gb$. Then
\begin{align}
\bT(u)&=1+u \left(a^{-1} \ckk k_{\Nb\Nb}+\l \bd_\bT \bT, \bN\r\right), \label{4.3.4}\\
 \bN(u)&=-1+u \left(\l\bd_\bN\bT, \bN\r-\frac{\bb^{-1} t }{\rho}\ckk k_{\Nb \Nb}\right), \label{4.3.5}\\
\bN(\bb^{-1}) & =\frac{\tir}{t} \left(\frac{\bb^{-1}t}{\rho}\ckk k_{\Nb \Nb} +\pib_{\bN \bN}\right). \label{4.5.6}
\end{align}
\end{lemma}

\begin{proof}
By using $\tir^2 = (\bb^{-1} t)^2 -\rho^2$ and (\ref{eq_1}) we obtain
$
\bT(\tir) = \frac{\bb^{-1} t}{\tir} \left(\bT(\bb^{-1} t) -1\right).
$
Thus, for $u = \bb^{-1} t - \tir$ we have
\begin{align} \label{2016.5.13}
\bT(u) = \left(1-\frac{\bb^{-1} t}{\tir}\right) \bT(\bb^{-1} t) + \frac{\bb^{-1} t}{\tir}.
\end{align}
In view of (\ref{eq_1}) and Lemma \ref{frames} we can derive that
\begin{align*}
\bT(\bb^{-1} t)&=- \bT(\rho \l\fB, \bT\r)=-\bT(\rho) \l \fB, \bT\r - \rho \left(\l \bd_\bT \fB, \bT\r + \l \fB, \bd_\bT \bT\r \right) \\
&=\frac{(\bb^{-1} t)^2}{\rho^2}-\left(\frac{\tir^2}{\rho} k_{\Nb \Nb}+a^{-1}\rho\l\bd_\bT \bT, \bN\r\right)\\
&=1-\rho\left(a^{-2} \ckk k_{\Nb\Nb}+a^{-1} \l \bd_\bT \bT, \bN\r\right).
\end{align*}
Plugging this equation into (\ref{2016.5.13}) shows (\ref{4.3.4}).

To see (\ref{4.3.5}), from (\ref{fbtn}) we note that
\begin{equation}\label{2016.5.13.2}
\bN(u)= a \left(\fB(u)-\frac{\bb^{-1} t}{\rho} \bT(u)\right).
\end{equation}
In view of (\ref{eq_1}), (\ref{3.15.1}) and (\ref{3.15.4}) we have
\begin{align*}
\fB(u)&=\fB(\bb^{-1}t)- \fB(\tir)= -\l \fB, \bT\r - \rho \l \fB, \bd_\fB \bT\r - \fB(\tir) \\
& = \frac{u}{\rho}- u a^{-1} \pib_{\bN\bN} + \frac{\bb^{-1} t u}{\rho} \l \bd_\bT \bT, \bN\r.
\end{align*}
Combining this and (\ref{4.3.4}) with (\ref{2016.5.13.2}), we obtain (\ref{4.3.5}).

Finally, noting that $\bN(\rho) = -|\nabla \rho|=-a^{-1}$, it follows from (\ref{eq_1}) and
Lemma \ref{frames} that
\begin{align*}
 \bN(\bb^{-1} t) & = \bN(-\rho \l \fB, \bT\r)= a^{-1} \l \fB, \bT\r -\rho\l \bd_\bN \fB, \bT\r-\rho\l \fB, \bd_\bN \bT\r\\
& = -\frac{\bb^{-1} t}{\rho} a^{-1} + \frac{\bb^{-1} t}{\rho} \tir k_{\Nb\Nb} + \tir \pib_{\bN\bN}\\
& = \tir\left(\frac{\bb^{-1}t}{\rho}\ckk k_{\Nb\Nb}+\pib_{\bN \bN}\right)
\end{align*}
which together with the fact $\bN(\bb^{-1} t) = t \bN(\bb^{-1})$ shows (\ref{4.5.6}).
\end{proof}

For future reference, we define
\begin{equation*} 
\theta_{AC}:=\l\sn_A \bN, e_C\r, \quad \underline{\theta}_{AC}:=\l \sn_A \Nb, e_C\r
\end{equation*}
We use $\tr \underline{\theta}:= \ga^{AC} \underline{\theta}_{AC}$ and $ \underline{\hat \theta}
:= \underline{\theta}-\frac{1}{2} \tr \underline{\theta} \ga$ to denote the trace and traceless part of $\ud\theta$ respectively.
Similarly we use $\tr \theta$ and $\hat \theta$ to denote the trace and traceless part of $\theta$ respectively.

\begin{lemma}\label{nbb}
There hold
\begin{align}
&\l \nabb_\Nb \Nb, e_A\r = -\sn_A \log a-\sn_A\log n, \label{nabba}\\
&\l \nab_\bN \bN, e_A\r=-\sn_A \log a, \label{nabn}\\
&\sn \log a=-\frac{\bb^{-1}t}{\tir} \left(\pib_{A\bN}+k_{A\Nb}\right),\label{3.6.7}\\
&\underline{\theta}_{AC} = \frac{\bb^{-1}t}{\tir} k_{AC}+\frac{\rho}{\tir}\pib_{AC},\label{3.8.1} \\
& \tr \ud \theta  =-\frac{\bb^{-1}t}{\tir}\ud \delta+\frac{\rho}{\tir}\delta'+\frac{2}{3}\tr k\frac{\bb^{-1}t}{\tir}, \label{2.3.3.16}\\
&\hat {\ud\theta}_{AC}=\frac{\bb^{-1}t}{\tir}(\hk_{AC}+\f12\ud\delta \ga_{AC})+\frac{\rho}{\tir}(\pib_{AC}-\f12 \ga_{AC}\delta'),\label{2.3.4.16}
\end{align}
where $\ud\delta=\hk_{\Nb\Nb}$ and $\delta'=-\pib_{\bN\bN}$.
\end{lemma}

\begin{proof}
(\ref{3.6.7}) follows from (\ref{3.6.4}) and $a^{-1} = \tir/\rho$, (\ref{3.8.1}) can be derived by using the first
equation in (\ref{dcp_3}), and (\ref{2.3.3.16}), (\ref{2.3.4.16}) are direct consequences of (\ref{3.8.1}) and $\bg^{\bi\bj} \pib_{\bi\bj}=0$ in (\ref{5.13.3.16}).
In view of (\ref{nabt}), we have $\Nb(t)=(an)^{-1}$. Thus, by using $e_A(t)=0$ and (\ref{nb}) we have
\begin{align*}
\sn_A((an)^{-1})&=[e_A, \Nb](t)= \l [e_A, \Nb], \nabb t\r =(an)^{-1}\l \sn_A \Nb-\nabb_{\Nb} e_A, \Nb\r\\
& = -(an)^{-1} \l \nabb_{\Nb} e_A, \Nb\r = (an)^{-1}\l \nabb_{\Nb} \Nb, e_A\r
\end{align*}
which gives (\ref{nabba}). (\ref{nabn}) can be similarly proved.
\end{proof}

\subsection{Structure equations}\label{5.10.1.16}

For $\H_\rho$-tangent symmetric  2-tensors $F_{ij}$ and $G_{ij}$, we set
\begin{align}
(F*G)_{ij} &=F_i^l G_{lj}+F_{jl}G^l_i, \label{3.14.2}\\
(F\hot F)_{ij}& =\f12(F*F)_{ij}-\frac{1}{3} |F|^2_{\gb} \gb_{ij} \label{hotd}
\end{align}
which define two $\H_\rho$-tangent symmetric 2-tensors.
%
We now derive the following structure equations.

\begin{proposition}
\begin{align}
&\left(\fB+\frac{n^{-1}\bb^{-1}}{\rho}\right)(n-\bb^{-1})=-\frac{\tir}{t} a^{-1}\pib_{\bN\bN}
  +\frac{\tir\bb^{-1}}{\rho}\l \bd_\bT \bT, N\r+\p_\rho n,  \label{Bb1}\\
&\fB \left(\log\frac{t}{\tau}\right)=\frac{\bb^{-1}n^{-1}-1}{\rho}, \label{ctt}\\
&\fB ( \emph{\tr}\, k)+\frac{1}{3}(\emph{\tr}\, k)^2=-\bR_{\fB\fB}- \hat k\c \hat k, \label{s1}\displaybreak[0]\\
&\fB\left(\emph{\tr}\, k-\frac{3}{\rho}\right)+\frac{2}{\rho} \left(\emph{\tr}\, k-\frac{3}{\rho}\right)
  =-\frac{1}{3} \left(\emph{\tr}\, k-\frac{3}{\rho}\right)^2-\bR_{\fB\fB}-|\hk|^2, \label{3.14.1}\\
&\bd_\fB {\hat k}_{ij}+\frac{2}{3} \emph{\tr}\, k \, {\hat k}_{ij}=-\widehat\bR_{\fB i \fB j}-\hat k\hot \hat k,\label{s1.1}\\
&\fB (\emph{\tr} \pr) =2\sR \left( \emph{\tr}\, k-\frac{3}{\rho}\right), \label{btrpr}\\
&\bd_\fB \phr_{ij}+({\hk}*\phr)_{ij}=2\Lie_\sR\hk_{ij}-\frac{2}{3}\emph{\tr}\,\pr \hk_{ij}, \label{bpr0}
\end{align}
where $\hat \bR_{\fB i\fB j}=\bR_{\fB i \fB j}-\frac{1}{3} \bR_{\fB \fB} \gb_{ij}$.
\end{proposition}

\begin{proof}
By using (\ref{bt}) we have
\begin{equation*} 
\fB \left(\frac{\rho}{t}\right)=\frac{1}{t}-\frac{\rho}{t^2}\fB(t)=\frac{1}{t} \left(1-\bb^{-1}n^{-1}\right).
\end{equation*}
In view of (\ref{eq_1}) we then obtain
\begin{align*}
\fB \left(-\bb^{-1}\right)&= \fB\left(\frac{\rho}{t} \l \fB, \bT\r\right)
  = \frac{\rho}{t}\l \fB, \bd_\fB \bT\r+\fB\left(\frac{\rho}{t}\right)\l \fB, \bT\r\nn\\
&=\frac{\rho}{t}\l \fB, \bd_\fB \bT\r-\frac{\bb^{-1}}{\rho}(1-\bb^{-1}n^{-1}) 
\end{align*}
which together with (\ref{3.15.1}) then implies (\ref{Bb1}).
(\ref{ctt}) is an immediate consequence of (\ref{bt}) and  $\fB(\tau)=\frac{\tau}{\rho}$ in (\ref{lse3}).

Now we consider (\ref{btrpr}) and (\ref{bpr0}).  Note that
\begin{equation*}
\Lie_\fB \Lie_\sR \gb_{ij}=\Lie_\sR \Lie_\fB \gb_{ij}, \quad \Lie_\sR \gb_{ij}=\pr_{ij}, \quad \Lie_\fB \gb_{ij}=2k_{ij}
\end{equation*}
we can obtain $\Lie_\fB \pr_{ij}=2\Lie_\sR k_{ij}$. Hence
\begin{equation}\label{3.14.8}
\bd_\fB \pr_{ij}+k_i^c \pr_{cj}+k_j^c \pr_{ic}=2\Lie_\sR k_{ij}.
\end{equation}
By taking trace of (\ref{3.14.8}) and using $\Lie_\sR \gb^{ij}=-\left(\pr\right)^{ij}$,
we get
\begin{equation*}
\bd_\fB \left(\tr\pr\right)+2k^{ic} \pr_{ic}=2\Lie_\sR \tr k+2\pr_{ic} k^{ic}.
\end{equation*}
This shows $\bd_\fB \left(\tr\pr\right) =2\Lie_\sR \tr k$ which gives (\ref{btrpr}).
Using (\ref{3.14.8}) and (\ref{btrpr}), we have
\begin{align}
\bd_\fB \phr_{ij}+k_i^c \pr_{cj}+k_j^c \pr_{ic}
&=2\Lie_\sR \left(\hk_{ij}+\frac{1}{3}\tr k \gb_{ij}\right)-\frac{2}{3}\Lie_\sR\tr k\gb_{ij}\nn\\
&=2 \Lie_\sR \hk_{ij}+\frac{2}{3}\tr k \c \pr_{ij}\nn
\end{align}
which implies (\ref{bpr0}).

To obtain (\ref{s1})--(\ref{s1.1}), we use the identity
\begin{equation}\label{4.4.1.16}
\bd_\fB k_{ij}=-\bR_{\fB i\fB j} -k_i^l k_{lj}.
\end{equation}
By taking the trace and traceless part of this identity we obtain (\ref{s1}) and (\ref{s1.1}).
(\ref{3.14.1}) is a direct consequence of (\ref{s1}).

Finally we show (\ref{4.4.1.16}). By using the boost vector field $\{\sR_i\}_{i=1}^3$ defined
in (\ref{boost}), we first note that,  in view of  (\ref{lse3})
\begin{align}
\bd_\fB k(\sR_i, \sR_j)&=\fB k(\sR_i, \sR_j)-k(\bd_\fB \sR_i, \sR_j)-k(\sR_i, \bd_\fB \sR_j)\nn\\
&=\fB k(\sR_i, \sR_j)-k(\bd_{\sR_i} \fB, \sR_j)-k(\sR_i, \bd_{\sR_j} \fB)\nn\\
&=\fB k(\sR_i, \sR_j)-2 k_{\sR_i}^l k_{l \sR_j}. \label{4.4.2.16}
\end{align}
For the first term, by using $\bd_\fB \fB=0$ and (\ref{lse3}) again, we have
\begin{align}
\fB \left(k(\sR_i, \sR_j)\right)
&= \fB\left(\l \bd_{\sR_i} \fB, \sR_i\r\right) =\l \bd_\fB \bd_{\sR_i}\fB, \sR_j\r+\l \bd_{\sR_i} \fB, \bd_\fB \sR_j\r\nn\\
&=\l\bd_{\sR_i}( \bd_\fB \fB)+\bR(\fB, \sR_i) \fB+\bd_{[\fB, \sR_i]}\fB, \sR_j\r+\l \bd_{\sR_i}\fB, \bd_\fB \sR_j\r\nn\\
&=\l\bR(\fB, \sR_i)\fB, \sR_j\r +\l \bd_{\sR_i}\fB, \bd_{\sR_j}\fB\r\nn\\
&=\bR_{\sR_j\fB \fB \sR_i} +k_{\sR_i}^l k_{l \sR_j}\label{4.4.3.16}
\end{align}
By combining (\ref{4.4.2.16}) with (\ref{4.4.3.16}), and using the fact that $\{\sR_i\}_{i=1}^3$ forms a frame
on $\T \H_\rho$, we can obtain (\ref{4.4.1.16}).
\end{proof}

\subsection{Radial decompositions on hyperboloids}\label{radiald}

Let $\{e_A\}$ be an orthonormal frame on $S_{t, \rho}$. We define $\zb_A=\l \bd_\fB\Nb, e_A\r$.
In view of (\ref{dcp_3}), we have
\begin{align*}
\l \bd_\fB \Nb, e_A\r = -\frac{\rho}{\tir}\l \bd_\fB \bT, e_A\r
\end{align*}
which together with (\ref{3.15.3}) shows that
\begin{equation}\label{zba}
\zb_A=-\left(\frac{\bb^{-1}t}{\tir}\l\bd_\bT \bT, e_A\r+\l \bd_\bN \bT, e_A\r\right).
\end{equation}
Note that $\l\bd_\fB \Nb, \fB\r=\l\bd_\fB \Nb, \Nb\r=0$,  we then have
\begin{equation}\label{3.14.10}
\bd_\fB \Nb^\mu=\zb^A e_A^\mu.
\end{equation}

Let us define the projection tensor from spacetime to $S_{t,\rho}$,
$$\sl{\Pi}^{\mu\nu}=\bg^{\mu\nu}+\fB^\mu\fB^\nu-\Nb^\mu \Nb^\nu. $$
By the definition of $\gb$, we have
$\sl{\Pi}^{ij}=\gb^{ij}-\Nb^i\Nb^j$.
For any spacetime one form $F^\mu$, we set
\begin{equation*}
\sn_\fB F_A:=e_A^\nu \bd_\fB (F^{\mu} \sl{\Pi}_{\mu\nu}).
\end{equation*}
This definition can be similarly extended to  $\H_\rho$-tangent symmetric two tensor $F_{ij}$.
For $F_{ij}$, we further define  a 1-form and a scalar function by
$$
F_{\Nb j}=F_{ij}\Nb^i,\quad F_{\Nb \Nb}=F_{ij}\Nb^i \Nb^j.
$$

\begin{lemma}\label{rad.1}
\begin{align}
&\sn_\fB F_{AC}=(\bd_\fB F)_{AC}-F_{\Nb C} \zb_A-F_{\Nb A } \zb_C\label{3.14.11}\\
&\sn_\fB F_{\Nb C}=(\bd_\fB F)_{\Nb C}+F_C^A \zb_A-F_{\Nb \Nb} \zb_C\label{3.14.12}\\
&\p_\rho (F_{\Nb\Nb})=(\bd_\fB F)_{\Nb\Nb}+2 F^A_\Nb\zb_A\label{3.14.13}
\end{align}
\end{lemma}

\begin{proof}
We have
\begin{align}
\sn_\fB F_{AC}&=e_A^i e_C^j \bd_\fB(F_{i'j'} \sl{\Pi}^{i'}_i \sl{\Pi}^{j'}_j)\nn\\
&=e_A^i e_C^j \left(\bd_\fB F_{i'j'}\sl{\Pi}^{i'}_i \sl{\Pi}^{j'}_j-F_{i'j'}\bd_\fB\Nb_i \Nb^{i'} \sl{\Pi}_j^{j'}- F_{i'j'}\bd_\fB \Nb_j \Nb^{j'}\sl{\Pi}_i^{i'}\right). \nn
\end{align}
This implies (\ref{3.14.11}) in view of (\ref{3.14.10}). Moreover
\begin{align*}
\sn_\fB F_{\Nb C}&=e_C^j \bd_\fB \left(F_{i'j'} \Nb^{i'}\sl{\Pi}_j^{j'}\right)\nn\\
&=e_C^j \left(\bd_\fB F_{i'j'} \Nb^{i'} \sl{\Pi}_j^{j'}+F_{i'j'} \bd_\fB \Nb^{i'} \sl{\Pi}^{j'}_j
+F_{i'j'}\Nb^{i'} \bd_\fB \sl{\Pi}_j^{j'} \right)\nn\\
&=\bd_\fB F_{\Nb C}+F_{AC}\zb^A-F_{\Nb j'} \left(\bd_\fB \Nb_j \Nb^{j'}+\Nb_j \bd_\fB \Nb^{j'}\right) e_C^j
\end{align*}
which  gives (\ref{3.14.12}) by using (\ref{3.14.10}).  (\ref{3.14.13}) can be obtained similarly.
\end{proof}

As a consequence of  (\ref{s1.1}) and Lemma \ref{rad.1}, We can obtain the structure equations for components of $\hk$ under radial decomposition on $\H_\rho$.
\begin{align}
&\p_\rho \hk_{\Nb\Nb}+\frac{2}{3}\tr k \hk_{\Nb\Nb}=G_{\Nb\Nb}=2 \hk_{\Nb A}\zb_A-\widehat{\bR}_{\fB \Nb\fB \Nb}-(\hk\hot\hk)_{\Nb\Nb}\label{2.20.3.16}\\
&\sn_\fB \hk_{\Nba C}+\frac{2}{3}\tr k \hk_{\Nb C}=G_{\Nba C}=\hk_C^A \zb_A-\hk_{\Nb\Nb}\zb_C-\widehat{\bR}_{\fB \Nb \fB C}-(\hk\hot \hk)_{\Nb C}\label{2.20.4.16}\\
&\sn_\fB \hk_{AC}+\frac{2}{3}\tr k \hk_{AC}=G_{AC}=-\hk_{\Nb C}\zb_A-\hk_{\Nb A}\zb_C-\widehat{\bR}_{\fB A \fB C}-(\hk\hot \hk)_{AC}.\label{2.20.5.16}
\end{align}
For convenience, we fix the convention that $\Ab$ denotes any elements in $\{\hk_{ij}, \tr k-\frac{3}{\rho}\}$, and that  $\Ab^\sharp$ denotes the 1-form $k_{\Nb A}$.
Symbolically,  the last terms of (\ref{2.20.3.16})-(\ref{2.20.5.16}) and (\ref{3.14.1}) can be recast below
\begin{equation}\label{4.28.9.16}
(\hk\hot \hk)_{\Nb C}=\Ab \c \Ab^\sharp,\quad (\hk\hot\hk)_{AC},\, (\hk\hot \hk)_{\Nb\Nb},\,  |\hk|^2=\Ab \c \Ab.
\end{equation}

\section{\bf Energy scheme and preliminary estimates on hyperboloids}\label{engsch}

In this section, we outline the main steps of the energy scheme in \cite{Wang15}
and give a rough statement of the main theorem therein.

\subsection{Bianchi equation of the spacetime}

Let us start with deriving the Einstein Bianchi equations for the EKG system,  the equation
system of weyl curvature tensor that our energy scheme is based on.

We decompose Riemannian curvature in the spacetime $(\M, \bg)$ into the Weyl curvature $W$
and the part of Schouten tensor
\begin{equation}\label{4.4.4.16}
S_{\mu\nu}=\bR_{\mu\nu}-\frac{1}{6} \bR \bg_{\mu\nu},
\end{equation}
with $\bR$ the scalar curvature in $(\M, \bg)$,
\begin{align}\label{wy}
W_{\a\b\ga\delta}=\bR_{\a\b\ga\delta}-\f12(\bg_{\a\ga}S_{\b\delta}+\bg_{\b\delta} S_{\a\ga}-\bg_{\b\ga}S_{\a\delta}-\bg_{\a\delta} S_{\b \ga}).
\end{align}
We define the left and the right dual of a Weyl tensor $\Psi$ to be
\begin{equation}\label{4.4.5.16}
{}^\star \Psi_{\a\b\ga\delta}=\f12 \ep_{\a\b\mu\nu} \Psi^{\mu\nu}_{\dum\dum\ga \delta}, \quad \quad { \Psi^{\star}}_{\a\b\ga\delta}=\f12 \Psi_{\a\b}^{\dum\dum \mu\nu} \ep_{\mu\nu\ga\delta}.
\end{equation}
It is a fact that the left and the right dual are equal since $\Psi$ is a Weyl tensor.

\begin{lemma}[The Bianchi equations]\label{7.23.12}
For $\Psi =W $ and ${}^\star W$, there hold the Bianchi equations,
\begin{equation} \label{bianchi1}
 \bd^\a W_{\a\b\ga\delta}=\J_{\b\ga\delta}  \quad \mbox{and} \quad
  \bd^\a {}^\star W_{\a\b\ga\delta}={}^\star \J_{\b\ga\delta}
\end{equation}
where the Weyl currents $\J$ and ${}^\star\J$  are 3-tensor fields, verifying
\begin{equation}\label{3.6.8}
\J_{\b\ga\delta}=\f12(\bd_\ga S_{\b\delta}-\bd_\delta S_{\b\ga}), \quad
{}^\star\J_{\b\ga\delta} =\frac{1}{4} \ep^{\mu\nu}_{\dum\dum\ga\delta}(\bd_\mu S_{\b\nu}-\bd_\nu S_{\b\mu})
\end{equation}
with
\begin{equation}\label{schouten}
S_{\a\b}=\bd_\a\phi \bd_\b \phi-\frac{1}{6}\bg_{\a\b}(\bd^\mu \phi \bd_\mu \phi-\fm \phi^2).
\end{equation}
\end{lemma}

\begin{proof}
(\ref{schouten}) is an immediate consequence of (\ref{4.4.4.16}) and (\ref{ricci}).
By using ${}^\star W=W^\star$,  $[\bd, \ep\c]=0$ and (\ref{4.4.5.16}), we can obtain
the second identity in (\ref{3.6.8}) from the first identity. It remains only to prove
the first identity in (\ref{3.6.8}). In view of (\ref{wy}) we have
\begin{align}\label{4.4.6.16}
\J_{\b\ga\delta} = \bd^\a \bR_{\a\b\ga\delta}
- \f12 \left(\bd_\ga S_{\b\delta} + \bg_{\b\delta} \bd^\a S_{\a\ga}-\bg_{\b\ga} \bd^\a S_{\a\delta} - \bd_\delta S_{\b\ga}\right).
\end{align}
By virtue of the contracted Bianchi identities
$$
\bd^\a \bR_{\a\b\ga\d} = \bd_\ga \bR_{\b\delta} -\bd_\delta \bR_{\b\ga}, \quad
\bd^\delta \bR_{\b \delta} = \f12 \bd_\b \bR,
$$
we can obtain from (\ref{4.4.4.16}) that $\bd^\b S_{\b\delta} = \frac{1}{3} \bd_\delta \bR$ and
\begin{equation*}
\bd^\a \bR_{\a\b\ga\d}
=\bd_\ga S_{\b\delta}-\bd_\delta S_{\b\ga}+\frac{1}{6}(\bd_\ga\bR \c \bg_{\b\delta}-\bd_\delta\bR \c\bg_{\b\ga}).
\end{equation*}
Substituting this identities into (\ref{4.4.6.16}) shows the first identity in (\ref{3.6.8}).
\end{proof}

We fix the convention that  $\{e_\bi, \bi=1,2,3\}$  denotes an orthonormal frame on $\Sigma_t$
and $\{e_i, i=1,2,3\}$ denotes an orthonormal frame on $\H_\rho$. With the tetrad $\{\fB, e_i\}$
of the hyperboloidal foliation and the tetrad $\{\bT, e_\bi\}$ of the maximal foliation,
we define for the weyl tensor $W$ the two sets of electric and magnetic decompositions
\begin{align} \label{7.23.8}
E_{\bi \bj}=W_{\bT \bi \bT \bj}, \quad H_{\bi \bj}={}^\star W_{\bT \bi \bT \bj}; \quad
\Eb_{ij}=W_{\fB i \fB j}, \quad \Hb_{ij}={}^\star W_{\fB i\fB j}.
\end{align}

\begin{lemma}\label{dcp_sc}
With respect to the tetrad $\{\fB, e_i\}$ there hold
\begin{align}
\widehat \bR_{ij} & =\bd_i \phi \bd_j \phi-\frac{1}{3} \gb_{ij} \bd^l \phi \bd_l \phi, \label{3.6.10}\\
\widehat \bR_{i\fB j \fB} & =\Eb_{ij}-\f12 \widehat\bR_{ij}, \quad {}^\star \bR_{i \fB j\fB}=\Hb_{ij}.
\label{3.6.11}
\end{align}
The same decomposition holds for $E, H$ with respect to the tetrad $\{\bT, e_\bi\}$.
\end{lemma}

\begin{proof}
Recall that ${\widehat\bR}_{ij} = \bR_{ij}-\frac{1}{3} \gb_{ij} \gb^{mn} \bR_{mn}$, we can obtain
(\ref{3.6.10}) from (\ref{ricci}) directly.
In view of (\ref{4.4.4.16}), (\ref{wy}), $\l e_i, \fB\r =0$ and $\fB, \fB\r =-1$, we can derive that
\begin{align*}
\Eb_{ij}&=\bR_{i\fB j\fB}-\f12 (\gb_{ij} \bR_{\fB \fB}+\bg_{\fB \fB} \bR_{ij})-\frac{1}{6} \bR\gb_{ij}\\
&=\bR_{i\fB j\fB}-\frac{1}{3}\gb_{ij}\bR_{\fB \fB}-\frac{1}{6}\gb_{ij} (\bR_{\fB \fB}+\bR)+\f12 \bR_{ij}\\
&=\bR_{i\fB j \fB}-\frac{1}{3}\gb_{ij} \bR_{\fB \fB}+\f12 (\bR_{ij}-\frac{1}{3} \gb_{ij} \gb^{mn} \bR_{mn}).
\end{align*}
Recall that ${\widehat\bR}_{i\fB j\fB} = \bR_{i\fB j \fB}-\frac{1}{3}\gb_{ij} \bR_{\fB \fB}$ and
we thus obtain the first identity in (\ref{3.6.11}).

To show the second identity in (\ref{3.6.11}), we note that
\begin{align*}
&\ep_{\mu\nu}^{\dum\dum\dum \a\b} (\bg_{\a\ga}S_{\b\delta}+\bg_{\b\delta} S_{\a\ga}-\bg_{\b\ga} S_{\a\delta}-\bg_{\a\delta} S_{\b\ga})
=2\ep_{\mu\nu \ga}^{\dum\dum\dum\dum\b} S_{\b\delta}-2 \ep_{\mu\nu \delta}^{\dum\dum\dum\dum \b} S_{\b \ga}.
\end{align*}
Thus we can obtain from (\ref{wy}) that
\begin{equation*} 
{}^\star W_{\mu\nu \ga\delta}={}^\star \bR_{\mu\nu \ga\delta}-(\ep_{\mu\nu \ga}^{\dum\dum\dum\dum \b} S_{\b\delta}-\ep_{\mu\nu \delta}^{\dum\dum\dum\dum \b} S_{\b\ga}).
\end{equation*}
This gives ${}^\star W_{\mu\fB\ga \fB}={}^\star \bR_{\mu \fB \ga\fB}-\tensor{\ep}{_{\mu\fB \ga}^ \b}S_{\b\fB}$
which shows ${}^\star \bR_{i \fB j\fB}=\Hb_{ij}$ by symmetrization.
\end{proof}

\begin{corollary}
\begin{equation}\label{codazzib}
\emph{\div} k=\nabb \emph{\tr} k-\bR_{\fB i}, \quad \emph{\curl} k=-\Hb.
\end{equation}
\end{corollary}

\begin{proof}
We only need to check the second identity. By using \cite[Page 10]{CK} we have
\begin{equation*}
\nabb_i k_{jm}-\nabb_j k_{im}=-\bR_{m\fB ij}.
\end{equation*}
Thus $\ep_n^{\dum\dum i j}\nabb_i k_{j m}=-\bR_{m \fB ij} \ep_{n\fB}^{\dum\dum\dum ij}$.
By symmetrizing $n,m$ we obtain $\curl k_{nm}=-{}^\star\bR_{n \fB m \fB}$ which together with
the second identity in (\ref{3.6.11}) shows $\curl k=-\Hb$.
\end{proof}

Connected to the explicit formula for the Schouten tensor, we give a lemma for future reference.
\begin{lemma}
Let $S\subset\M$ be a $2$-D compact manifold, diffeomorphic to ${\mathbb S}^2$. Let
$\{\mathbb{L},\mathbb{\Lb}, e_A, A=1,2\}$ be a canonical null tetrad on $S$ in the sense that
$\mathbb{L}$ and $\mathbb{\Lb}$ are null vectors orthogonal to $S$ satisfying $\l\mathbb{ L},\mathbb{\Lb} \r=-2$
and $e_A, A=1,2$ are orthonormal frame on $S$. Then the Gauss curvature $K$ on $S$ satisfies the equation
\begin{equation}\label{gauss}
K+\frac{1}{4}\tr\chi\tr\chib-\f12 \chih_{AC}\chibh_{AC}=-\frac{1}{4}W(\mathbb{L}, \mathbb{\Lb}, \mathbb{L}, \mathbb{\Lb})+\f12\ga^{AC} S_{AC}
\end{equation}
where $S_{AC}$ denotes the angular component of the Schouten tensor, see (\ref{schouten}).
Here $\chi$ and $\chib$ are the null second fundamental forms defined by $\mathbb{L}$ and $\mathbb{\Lb}$ respectively as follows,
\begin{equation*}
\chi_{AB}=\l \bd_A \mathbb{L}, e_B\r;\quad \chib_{AB}=\l \bd_A \mathbb{\Lb}, e_B\r, \quad A, B=1,2.
\end{equation*}
\end{lemma}

\begin{proof} Let $\ga$ be the induced metric on $S$ and let $\sl{R}_{ADCB}$ be the curvature tensor on $S$.
By the Gauss equation we have
\begin{equation*}
\bR_{ADCB}=\sl{R}_{ADCB}+\f12(\chi_{AC} \chib_{BD}-\chi_{AB}\chib_{CD})+\f12(\chib_{AC}\chi_{BD}-\chib_{AB}\chi_{CD}).
\end{equation*}
Note that $\sl{R}_{ADCB}=(\ga_{AC}\ga_{DB}-\ga_{AB}\ga_{CD})K$, we obtain
\begin{equation*} 
\f12 \ga^{AC}\ga^{BD} \bR_{ADCB}=K+\frac{1}{2}(\f12\tr\chi\tr\chib-\chih_{AC}\chibh_{AC}).
\end{equation*}
Note that, due to (\ref{wy}),
\begin{equation*} 
W_{ADCC'}=\bR_{ADCC'}-\f12(\ga_{AC}S_{DC'}+\ga_{DC'}S_{AC}-\ga_{DC}S_{AC'}-\ga_{AC'}S_{DC})
\end{equation*}
which gives,
\begin{equation*} 
\ga^{AC}\ga^{DC'}W_{ADCC'}=\ga^{AC}\ga^{DC'}\bR_{ADCC'}-\ga^{DC'}S_{DC'}.
\end{equation*}
Since \cite[(7.3.3c)]{CK}  gives $\ga^{AC}\ga^{DC'}W_{ADCC'}=-\f12 W(\mathbb{L}, \mathbb{\Lb}, \mathbb{L}, \mathbb{\Lb})$,
we can use the above equation to obtain (\ref{gauss}).
\end{proof}

\begin{definition}\label{3.18.19}
Let $\{\mathbb{L},\mathbb{\Lb}, e_A, A=1,2\}$ be a canonical null tetrad.
We define the null decomposition  of a Weyl tensor $\Psi$ by a set of canonical null tetrad,
\begin{eqnarray*}
\ab(\Psi)(e_A, e_B)=\Psi_{\mu\ga \nu \delta}e_A^\mu e_B^\nu e_3^\ga e_3^\delta;
&&\udb(\Psi)(e_A)=\f12\Psi_{\mu \sig \ga\d}e_A^\mu e_3^\sig e_3^\ga e_4^\d; \\
\varrho(\Psi)=\frac{1}{4}\Psi_{\a\b\ga\d} e_3^\a e_4^\b e_3^\ga e_4^\d;
&&\sigma(\Psi)=\frac{1}{4}\sPs_{\a\b\ga\d}e_3^\a e_4^\b e_3^\ga e_4^\d; \\
\b(\Psi)(e_A)=\f12\Psi_{\mu\sigma\ga\delta}e_A^\mu e_4^\s e_3^\ga e_4^\d;
&& \a(\Psi)(e_A, e_B)=\Psi_{\mu\ga\nu\d}e_A^\mu e_B^\nu e_4^\ga e_4^\d;
\end{eqnarray*}
where $\mathbb{L}=e_4$ and $\mathbb{\Lb}=e_3$.
\end{definition}

\subsection{Energy for Klein-Gordon equation}\label{Eng_KG_d}

We consider the geometric Klein-Gordon equation (\ref{gkg}), i.e.
$\Box_\bg \phi= \fm\phi$. Recall the energy momentum tensor  $Q_{\mu\nu}[f]$ for (\ref{gkg}) defined by
\begin{equation*} 
Q_{\mu\nu}[f]=\p_\mu f \p_\nu f-\f12 \bg_{\mu\nu}(\bd^\a f \bd_\a f+\fm f^2).
\end{equation*}
This definition can be extended to a covariant 1-form $F$ as
\begin{equation*} 
Q_{\mu\nu|\ga\d}[F]=\bd_\mu F_\ga \bd_\nu F_\d-\f12 \bg_{\mu\nu}(\bd^\a F_\ga \bd_\a F_\d+\fm F_\ga F_\d)
\end{equation*}
which is a covariant 4-tensor, symmetric pairwise with respect to the indices $(\mu, \nu)$ and $(\ga, \d)$.
By virtue of the Riemannian metric $h_{\ga\d}:=\bg_{\ga\d}+2\bT_\ga \bT_\d$ we set
\begin{equation*} 
Q_{\mu\nu}[F]=h^{\ga\d}Q_{\mu\nu|\ga\d}[F].
\end{equation*}
The energy momentum tensors for higher order Lie derivatives of $f$ and $F$, with $\ell\in {\mathbb N}$,
can be defined as
\begin{align*}
Q\rp{\ell}_{\mu\nu}[f]&=\Lie\rp{\ell}_\sR \bd_\mu f \Lie\rp{\ell}_\sR \bd_\nu f\nn\\
&-\frac{1}{2} \bg_{\mu\nu}\left(\bg^{\rho \sigma}\Lie\rp{\ell}_\sR \bd_\rho f \Lie\rp{\ell}_\sR \bd_\sigma f
+\fm \Lie\rp{\ell}_\sR f \Lie\rp{\ell}_\sR f\right),
\end{align*}
\begin{align*}
Q\rp{\ell}_{\mu\nu}[F]&=h^{\a\b}(\Lie\rp{\ell}_\sR \bd_\mu F_\a \Lie\rp{\ell}_\sR \bd_\nu F_\b\nn\\
&-\frac{1}{2} \bg_{\mu\nu} \left(\bg^{\rho \sigma}\Lie\rp{\ell}_\sR \bd_\rho F_\a \Lie\rp{\ell}_\sR \bd_\sigma F_\b
+\fm \Lie\rp{\ell}_\sR F_\a \Lie\rp{\ell}_\sR F_\b\right). 
\end{align*}

For a smooth scalar function $f$ and a covariant 1-form $F$, with $\Omega_\rho\subset \H_\rho$ to be specified later,
we  set the energy current to be
\begin{equation*} 
P_f^\a=Q_{\a\b}[f] \bT^\b, \quad P_F^\a=Q_{\a\b}[F] \bT^\b.
\end{equation*}
\begin{align*}
&\E_f(\rho)^2=\int_{\Omega_\rho} P_f^\a \fB_\a  d\mu_{\gb},\quad
\E_F(\rho)^2=\int_{\Omega_\rho} P_F^\a \fB_\a  d\mu_{\gb},
\end{align*}
Then for a solution $\phi$ of (\ref{gkg}), we have
\begin{eqnarray*}
\Ephin(\rho)^2=\int_{\Omega_\rho} Q(\bT, \fB)[\phi] d\mu_{\gb} &&\Ephi{m}(\rho)^2:=\int_{\Omega_\rho} Q\rp{m}(\bT, \fB)[\phi] d\mu_{\gb}\\
\Edphin(\rho)^2=\int_{\Omega_\rho}Q(\bT, \fB)[\bd\phi] d\mu_{\gb} &&\Edphi{m}(\rho)^2:=\int_{\Omega_\rho} Q\rp{m}(\bT, \fB)[\bd\phi] d\mu_{\gb}.
\end{eqnarray*}
Energies $\Ephin(t), \Ephi{m}(t), \Edphin(t), \Edphi{m}(t) $ on $U_t\subset\Sigma_t$ can be similarly defined, with the surface normal $\fB$ replaced by  $\bT$.

By using (\ref{dcp_2}) and (\ref{nf}) we have
\begin{align*}
&Q(\bT, \fB)[f]= \frac{1}{4\rho} \left(u(\Lb f)^2+\ub (Lf)^2\right)+\frac{\bb^{-1} t}{2\rho}(|\sn f|^2+\fm f^2), \\ 
&Q(\bT, \fB)[F]=\frac{1}{4\rho}(u|\bd_\Lb F |_h^2+\ub|\bd_L F|_h^2)+\frac{\bb^{-1}t}{2\rho}(h_{\mu\nu}\ga^{AC}\bd_A F^\mu \bd_C F^\nu+\fm |F|_h^2). 
\end{align*}
With the help of (\ref{dcp_2}), we can derive that
\begin{align*}
\frac{1}{2\rho}\left(u(\Lb f)^2+\ub (Lf)^2\right)
&=\frac{1}{2}\left[\frac{\ub+u}{\rho}((\fB f)^2+(\Nb f)^2)+2\frac{\ub-u}{\rho}\fB f\c \Nb f\right]\\
&=\frac{u}{\rho}\left((\fB f)^2+(\Nb f)^2\right)+\frac{\tir}{\rho}\left(\frac{u}{\rho} \Lb f\right)^2
\end{align*}
which, in view of $\tir\ge 0$ and $\rho^2 = u \ub$, implies
\begin{align*}
&Q(\bT, \fB)[f]\ge\frac{\rho}{2\ub}(|\fB f|^2+|\Nb f|^2)+\frac{\bb^{-1}t}{2\rho}(|\sn f|^2+\fm f^2), \\ 
&Q(\bT, \fB)[F]\ge \frac{\rho}{2\ub}(|\bd_\fB F |_h^2+|\bd_\Nb F|_h^2)+\frac{\bb^{-1}t}{2\rho}(h_{\mu\nu}\ga^{AC}\bd_A F^\mu \bd_C F^\nu+\fm |F|_h^2). 
\end{align*}

\subsection{Bel-Robinson Energy }\label{BR-E}

Let $\Psi$ be a Weyl tensor and let $\Q(\Psi)$ be the associated Bel-Robinson tensor defined by
\begin{equation*} 
\Q(\Psi)_{\a\b\ga\delta}=\Psi_{\a\rho \ga\sigma} \Psi_\b^{\dum\dum \rho}{}_\delta^{\dum\dum \sig}
+\sPs_{\a\rho \ga\sigma} \sPs_\b^{\dum\dum \rho}{}_\delta^{\dum\dum \sig}.
\end{equation*}
Because $\bg^{\a\ga} \Lie_X \Psi_{\a\b\ga\delta}=\up{X}\pi^{\a\ga} \Psi_{\a\b\ga\delta}$,
the Lie derivative $\Lie_X \Psi$ is not necessarily a Weyl tensor. However, the normalized Lie derivative
$\hat{\Lie}_X \Psi_{\a\b\ga\delta}$ defined by
\begin{align*}
\hat{\Lie}_X \Psi_{\a\b\ga\delta}=\Lie_X \Psi_{\a\b\ga\delta}&+\frac{3}{8}\tr\px\Psi_{\a\b\ga\delta}
-\frac{1}{2} \left(\up{X}\pi^\mu_\a \Psi_{\mu\b\ga\delta}+\up{X}\pi^\mu_\b \Psi_{\a\mu \ga\delta} \right. \nn\\
&\left. +\up{X}\pi^\mu_\ga \Psi_{\a\b\mu\delta}+\up{X}\pi^\mu_\delta \Psi_{\a\b\ga\mu}\right) 
\end{align*}
is a Weyl tensor ( see \cite[Page 139]{CK}).

Let $U_t\subset \Sigma_t$ and $\Omega_\rho\subset\H_\rho$, which will be further specified shortly.
For the Weyl part $W$ of the Riemann curvature tensor $\bR_{\a\b\ga\d}$ we now introduce for each integers $m\ge 0$
the following  set of energies
\begin{equation}\label{5.13.2.16}
\begin{split}
&\Wb^{(m)}(\rho)^2=\int_{\Omega_\rho} \Q(\hLW{m})(S,S,\bT, \fB) d\mu_{\gb},  \quad m\le 3\\
&\Kb^{(m)}(\rho)^2=\int_{\Omega_\rho}\Q(\hLW{m})(K_1, \bT, \bT, \fB) d\mu_{\gb},\quad m \le 1\\
&\Eeb^{(m)}(\rho)^2=\int_{\Omega_\rho} \Q(\hLW{m})(S,S,S, \fB) d\mu_{\gb}, \quad m\le 2\\
&\W^{(m)}(t)^2=\int_{U_t} \Q(\hLW{m})(S,S,\bT, \bT) d\mu_g, \quad m\le 2\\
&\K^{(m)}(t)^2=\int_{U_t} \Q(\hLW{m})(K_1,\bT,\bT, \bT) d\mu_g, \quad m\le 1\\
&\spE\rp{m}(t)^2=\int_{U_t}\Q(\hLW{m})(S, S, S, \bT) d\mu_g, \quad m\le 2
\end{split}
\end{equation}
where $S=\rho \fB$ and $K_1=\vartheta(\rho/t) \tir^2 L$ with $\vartheta$ being a smooth function on
$[0, \infty)$ taking values in $[0,1]$ and
\begin{equation*} 
\vartheta(s)=\left\{
\begin{array}{lll}
1 & \mbox{ if } s\le 1/3, \\
0 & \mbox{ if }  s\ge 2/3.
\end{array}
\right.
\end{equation*}
For tensor fields $F$ we define the energy
\begin{equation*} 
\sE[F](t)^2=\int_{\Sigma_t\cap\I^+(\O)} \sQ[F](\bT, \bT) d\mu_g,
\end{equation*}
where
\begin{equation*}
\sQ[F](\bT, \bT)=\frac{1}{2}\left(|\bd_\bT F|_{h}^2+h^{\bI \bJ}g^{ij}\bd_i F_\bI \bd_j F_\bJ \right).
\end{equation*}
We also record here the canonical Bel-Robson energy
\begin{equation*} 
\sE[W](t)^2=\int_{\Sigma_t\cap {\I^+(\O)}} \Q[W](\bT, \bT,\bT, \bT)d \mu_g\approx\int_{\Sigma_t\cap \I^+(\O)} \{|E|^2+|H|^2\} d\mu_g.
\end{equation*}

\subsection{A rough statement of the main theorem in \cite{Wang15} and the sketch of the proof}\label{8.07.1}

We will give a brief statement of the result in \cite{Wang15}. We emphasize that the main result of
this paper is included in Theorem \ref{4.10.6} and Proposition \ref{1.2.1.16}. These results do not
depend on the global result and the long-time estimates stated below and in Theorem \ref{2.15.2}.
Instead they rely on Theorem \ref{5.13.1.16}, which can be proved in a rather standard way (see \cite{CK}, \cite{KR2}, \cite{Wang10}), together
with a natural assumption that the foliation of the hyperboloids exists till certain proper time $\rho_*$.

\begin{theorem}[The first statement of main theorem in \cite{Wang15}]\label{2.15.1}
Consider the Einstein Klein-Gordon system (\ref{ricci})-(\ref{gkg}) under the maximal foliation
gauge (\ref{5.13.3.16}). Let  $(g_0, \pib_0, \phi[T])$ be a maximal data set, which is smooth
and satisfying (\ref{intr.2}).
Suppose that\begin{footnote}{The existence of such data can be justified by \cite{Corvino} and \cite{Chrusciel}.}\end{footnote}
the data $\phi[T]$ of the Klein-Gordon equation (\ref{gkg}) are compactly supported within $B_1$,
the Euclidean ball of radius $1$ centered at origin on $\{t=T\}$. Suppose also that on $\{t=T, r>2\}$
the metric $\bg$ coincides with the Schwarzschild metric which in terms of the polar coordinates $(t, r, \theta, \varphi)$ given by
\begin{equation}\label{3.20.lap}
\bg=-\frac{r-2M}{r+2M} dt^2+ \frac{r+2M}{r-2M} dr^2 +(r+2M)^2 (d\theta^2+\sin^2 \theta d\varphi^2)
\end{equation}
and $\pib=0$. The data $(g_0, \pib_0, \phi[T])$ is assumed to satisfy the smallness condition
\begin{equation*}
\|\p_x\rp{\le 7}({g_0}_{ij}-\delta_{ij})\|_{L^2(B_2)}+\|\nab_{g_0}\rp{\le 6} \pib_0\||_{L^2(B_2)}+\|\phi[T]\|_{H^7\times H^6}+M<\ve
\end{equation*}
where $H^s$ denotes the Sobolev space $W^{s,2}(\mathbb{R}^3)$ and $\phi[T]=(\phi(T, \cdot), \p_t \phi(T, \cdot))$.

If $\ve>0$ is sufficiently small, then there exists a unique, globally hyperbolic, smooth and geodesic complete solution  $(\M, \bg, \phi)$ foliated with level sets of a maximal time function $t$ and level sets of a proper time $\rho$, on which various sets of  energy are controlled in terms of $\ve$ as specified in Theorem \ref{2.15.2}.
\end{theorem}
\begin{remark}
To obtain the results in Theorem \ref{2.15.2} in wave zone, we only need to  propagate energies from $\H_{\rho_0}$ to the last slice $\H_{\rho_*}$ of  two-order less than the given data at the initial maximal slice.
\end{remark}

\subsubsection{Sketch of the main steps of the proof of Theorem \ref{2.15.1}}\label{12.10.02}
We define
\begin{equation}\label{geosch}
\hat u(t, r) = t - \gamma(r),
\end{equation}
where
$$
\gamma(r) = r + 4M \ln (r-2M), \quad \mbox{ for } r>2M.
$$
Consider $t\ge T$, for each $R\ge 2$ let $\C_{\hat u}^s$ denote the level set of $\hat u$ with $\hat u = T- \gamma(R)$.
  This $\C_{\hat u}^s$, which is called the schwarzschild cone,  is a ruled surface  by the outgoing null geodesics initiating
from $\{r=R, t= T\}$. We use $\itt(\C^s_{\hat u})$ to denote the interior of the region enclosed by $\C^s_{\hat u}$.
For the following exposition, we choose $R_0= 5/2$, $R_1 =2$, set $\hat u_0=T-\ga(R_0), \hat u_1 =T-\gamma(R_1)$,
and consider $\C^s_{\hat u_i}$ and $\itt(\C^s_{\hat u_i})$ for $i=0,1$.
Let $\rho_*$ be a large number. We set $S_{\rho, \hat u}=\C^s_{\hat u}\cap \H_\rho$ and define
\begin{equation}\label{rhob}
t^*=\max\{t: S_{\rho_*,\hat u_0}\},\quad t_*=\min\{t: S_{\rho_*, \hat u_0}\}.
\end{equation}


\begin{definition}\label{2.28.2.16}
\begin{enumerate}[leftmargin = 0.7cm]
\item Given a set of points $E$, we use $\sf{E}$ to denote the collection of time-like distance maximizing
geodesics connecting $\O$ and every point in $E$. For convenience, we write $\sf{q}:=\sf{\{q\}}$.

\item We define the wave zone by
$$
\hat{I}_0^+ = \overline{\itt (\C^s_{\hat u_0})}\cap\{\rho\le \rho_*, t\ge T\}.
$$

\item We define $I_0^+$ to be a truncated communication zone, where
$$
I_0^+=\overline{\itt(\sf{S_{\rho_*, \hat u_0}})}\cap \{\rho\le \rho_*, t\ge T\}.
$$

\item  We set $Z^\diamond = I_0^+\backslash \overline{\itt(\C^s_{\hat u_0})}$ which is called the zone of leakage.

\item The set $Z^s=\left({\{t\ge T\}\cap \I^+(\O)}\right)\backslash\itt(\C^s_{\hat u_1})$ is called the Schwarzschild zone.

\item We denote by $\C_0$  the outgoing light cone emanating from $\O$.
\end{enumerate}
\end{definition}

\begin{remark}\label{4.9.2.16}
{It is important to point out that for $p\in Z^s$, $\sf{p}\cap \{t\ge T\}$ is fully contained in $Z^s$.
Indeed, because $\sf{p}$ is a timelike geodesic reaching $p$, which has to be in the interior of
the backward light cone initiated at $p$. Such light cone is completely outside of $\itt (\C^s_{\hat u_1})$
due to the finite speed of propagation.}
\end{remark}

\begin{definition}
For $\rho_0\ge 2 T$ we define the region\begin{footnote}{For convenience, we consider $\rho_0=2T$, such that $\H_{\rho_0}$
can be proved to be  fully above $t=T$ due to Proposition \ref{prl_1}. }\end{footnote}
\begin{equation}\label{wz}
Z^+=\hat{I}^+_0\cap \{\rho\ge \rho_0\}
\end{equation}
We can split $Z^+$ as $Z^+=Z^\flat\cup Z^\sharp$, where
\begin{align*}
&Z^\flat:=\left(\bigcup_{\rho_0\le\rho\le \rho_* } \H_\rho\right) \cap\{ t\le  t_*\} \cap\overline{\itt (\C^s_{\hat u_0})},\\
&Z^\sharp :=  \left(\bigcup_{\rho_0\le\rho\le \rho_* } \H_\rho\right)\cap\{t_* < t\le t^*\} \cap\overline{\itt (\C^s_{\hat u_0})}.
\end{align*}
\end{definition}

In order to set up the energy scheme appropriately, we will rely on the following property,
which will be proved in Proposition \ref{12.30.1}.

\begin{proposition}\label{12.29.2}
There holds $Z^\sharp \subset\overline{\itt(\C^s_{\hat u_0})}\backslash\overline{\itt(\C^s_{\hat u_1})}$.
Consequently $Z^\sharp\subset Z^s$.
\end{proposition}

\begin{figure}[ht]
\centering
\includegraphics[width = 0.48\textwidth]{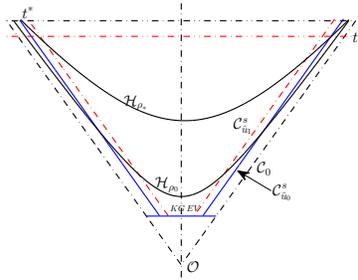}
  \vskip -0.7cm
  \caption{Illustration of wave zone}\label{fig4}
\end{figure}

Next, we sketch the main steps of the proof of Theorem \ref{2.15.1} the details will be given in \cite{Wang15}.
We consider the initial data set given in Theorem \ref{2.15.1}, at the maximal level set $t=T$,
due to the trivial shift of time stated at the beginning of  Section \ref{sec_c}.

{\bf Step 1: Local extension.} Initiated from $\{t=T,r=r_0>\frac{5}{2}\}$, we solve the Einstein-Klein-Gordon
system backward to a certain time $t<0$ and forward to $t=C^* T$ with
\begin{equation}\label{5.13.5.16}
C^*T \ge t_{\max}(S_{{\rho_0},{\hat u_0}} ).
\end{equation}
Our choice of $T$ is sufficiently large to guarantee\begin{footnote}{ Here $r(p)$ denotes the Euclidean distance
of the point $p$ to the center $(t_p,0)$.}\end{footnote} $r(\Sigma_T \cap \C_0)\ge 4$.
We establish in Theorem \ref{5.13.1.16} a set of energy estimates on $\Sigma_t$ with $0<t\le C^* T$ for the Weyl
tensor fields and the scalar fields. This allows us to control the geometry of the hyperboloids $\H_\rho$ in
$Z^{\loc}=\{0<\rho\le\rho_0, 0<t\le C^*T\}\cap \I^+(\O)$, see Proposition \ref{4.27.1.26}. The set of energy estimates
on $\H_{\rho_0}\cap \hat I^+_0$ is used  as the initial energies for the energy scheme in $Z^+$.

{\bf Step 2: Bootstrap assumptions.}   The goal of the energy scheme is to control the Bel-Robinson energies for
the Weyl part of the curvature and the energies on the scalar field $\phi$. These are achieved by a delicate
bootstrap argument. For a fixed but arbitrary number $\rho_*>0$, we make a set of bootstrap assumptions
on various sets of energies in the wave zone $\hat I^+_0$ up to the last slice $\H_{\rho_*}$. The deformation
tensors of $\bT, \fB$ and the boost vector fields $ \sR_a$, $a=1,2,3$, as the most crucial geometric quantities
that influence the propagation of energies, are also included in the bootstrap assumptions and need to be proved
simultaneously with the energy estimates. We also assume the radius of injectivity verifies  $i_*\ge \rho_*$.

{\bf Step 3: Boundedness theorem and the energy hierarchy.}
Establishing the boundedness theorem for various types of energies, undoubtedly, is the core part of the proof.
The analysis is based on the system (\ref{bianchi1}) and (\ref{gkg}). In the sequel, we only explain our strategy
in controlling the Weyl components, which already mirrors our treatment for the massive scalar fields. As stated in
Theorem \ref{2.15.2}, we will establish the boundedness theorem for three types of energies on the hyperboloids $\H_\rho$
and the maximal slice $\Sigma_t$ contained in $Z^+$. The three types of energies are called the standard energies, the
Morawetz energies and the CMC energies respectively, which form an energy hierarchy. There are two factors which need
to be balanced when constructing the hierarchy. One factor is the control in terms of weights, namely, the scalar
factors of $\rho$ or $t$ paired to the Weyl components. These weights, in particular, form the main factor that
determines the rate of decay for the Weyl components. The other factor is the control of the order of derivatives.
Among the three types of the energies, the standard energies give the control up to the third order derivatives
for the Weyl components. However, for certain components of the Weyl curvature and of the deformation tensor $\pt$,
the weights paired with do not provide sufficiently fast decay. To compensate such weakness, the other two types
of energies are created and bounded simultaneously with the standard energies.

One difficulty we encounter quite often is due to the incompatibility between the maximal frame and the hyperboloidal
frame, which constantly causes a loss of a weight of $\frac{t}{\rho}$, identical to a growth of $t^\f12$ in $\hat I^+_0$.
The other difficulty, not surprisingly, comes from the fact that the $\fB$ derivative does not take weight.
In both scenarios,  certain weights, being functions of $t,\rho$,  can not be bounded together with derivatives
of either the Weyl components or the scalar field.  To close the top order standard energies, by using the Einstein
Bianchi equations and the null condition exhibited in the Weyl currents, we perform the integration by part.
Such procedure, technically very involved though, is also used for closing other energy estimates. It
  gets the energy estimates closed at a very sharp level, which maximizes the benefit due to the null
forms that the Einstein Bianchi equations exhibit under the intrinsic tetrad.

In the boundedness theorem, we control energies on all hyperboloidal slices in $Z^+$ as well as on the part
of maximal slices  in $Z^\flat$. In particular, when $t_{\min}(\H_{\rho_*})< t<t_*$, the set where we consider
the $\Sigma_t$-energies is a family of annuli, with the inner boundary $S_{t, \rho_*}$ and the outer boundary
$S_{t, \hat u_0}$. When $t\ge t_*$, there is no such annuli region to obtain the $\Sigma_t$ energies. Hence
the region $Z^\sharp$ is excised from the wave zone, when consider the energy control on maximal slices.
However the energy argument in $Z^+$ relies  on the control of $\pt$ throughout $Z^+$. In the region $Z^\flat$,
we will control $\pt$  by combining energy estimates with the elliptic estimates provided by the Codazzi equations
(\ref{codazziT}). The Morawetz type energies are particularly important for obtaining sufficient control on $\pt$
in $Z^\flat$. Such energies are supposed to provide stronger control in terms of the weights for the Weyl components,
with a compromise on the order of derivatives. In $Z^\sharp$ the deformation tensor $\pt$ will be controlled in a
different way. In Proposition \ref{12.29.2} we show $Z^\sharp$ is contained in the Schwarzschild zone. Then $\pt$
can be analysed by using the information provided by the Schwarzschild metric and the geometric comparison established
in Theorem \ref{4.10.6}.

{\bf Step 4: Control of deformation tensors.} The proof of the boundedness theorem for the three types of energies
relies crucially on the control of the second fundamental form\begin{footnote}{$\up{S}\pi$ is fully represented
by $k$.}\end{footnote} $k$, $\pt$, $\up{\sR}\pi$ and their derivatives. The control of all these deformation
tensors are established simultaneous with all types of energy estimates, via a rather delicate bootstrap argument.
\begin{enumerate}[leftmargin = 0.7cm]
\item  To control $\pt_{\bi\bj}$, we use the Morawetz energies, the Codazzi equation (\ref{codazziT}) and the
Sobolev embedding on maximal slices. The control on the lapse function, which gives the control of $\pt_{0\bi}$,
is obtained by a set of elliptic estimates due to (\ref{lapse}).

\item To control $k_{ij}$ and $\up{\sR}\pi$, we use the transport equations for $k$ and $\up{\sR}\pi$;
see (\ref{s1})-(\ref{bpr0}). To obtain stronger control on $k_{i\Nb}$ and $\up{\sR}\pi_{i\Nb}$, we use
the Codazzi equation (\ref{codazzib}) on $\H_\rho$  with the boundary $S_{\rho, \hat u_0}$.
\end{enumerate}

{\bf Step 5: Boundary value and control of the leakage.}
In order to establish the long-time energy estimates inside the wave zone $\hat I_0^+$ enclosed by
$\C^s_{\hat u_0}$, we first need to derive two types of estimates on Weyl curvature components in $Z^s$.
One is  the bound of the curvature fluxes of various types of energy momentums along the Schwarzschild
cone $\C^s_{\hat u_0}$, which are as important as the bound on the initial data.  The other type is to
control the  Weyl components in the zone of leakage $Z^\diamond$. Such estimates are crucially used for
controlling the geometric quantities $k$ and $\up{\sR}\pi$ in the entire truncated communication zone
$I_0^+$ via the transport equations. Both types of estimates are derived by brutal force. This means that
they are based on a comprehensive comparison between the intrinsic hyperboloidal foliation with the
canonical Schwarzschild geometry in $Z^s$. In Proposition \ref{3.18.1.16}, we  provide the control of
the $0$-order Weyl curvature components in $Z^s$, which actually gives a set of very precise asymptotic
behavior of  the Weyl components when they approach the null infinity of the light cone $\C_0$ along all
the hyperboloids $\H_\rho$. More and higher order estimates of these types  are provided in \cite{Wang15}.

{\bf Step 6: Completion of the geometric argument.}
Finally, we extend the radius of injectivity beyond $\rho_*$. This is based on the control of curvature and
$k$ in $\I^+(\O)$ for $\rho\le \rho_*$. The control of curvature is obtained by the energy estimates in
the wave zone and the geometric comparison in $Z^s$ as explained in Step 5. The control of $k$ relies on the
transport equations and the control on curvature. The local-in-time estimates  and long-time estimates
in $Z^s$ for $k$ are proved in Proposition \ref{4.27.1.26} and Theorem \ref{4.10.6}.

Below we list the boundedness theorem and its consequence.

\begin{theorem}[main theorem of \cite{Wang15}: results in wave zone]\label{2.15.2}
Let the conditions in Theorem \ref{2.15.1} hold. Consider the energies defined in (\ref{5.13.2.16}) with
$U_t=Z^\flat\cap \Sigma_t$ and $\Omega_\rho=\hat I^+_0\cap \H_\rho$. Then for $\rho_0<\rho\le \rho_*$
and $T<t\le t_*$ with $\rho_*>0$ a fixed arbitrary large number and $t_*$ defined in (\ref{rhob}), there hold
\begin{enumerate}[leftmargin=0.7cm]
\item  the standard energy estimates:
\begin{equation*}
\Wb^{(\le 3)}(\rho)\les  \ve \rho^{C\ve},\quad \W^{(\le 2)}(t)\les \ve,
\end{equation*}
\item  the CMC energy estimates:
\begin{equation*}
\Eeb^{(\le 2)}(\rho)\les \ve \rho^{C\ve},\quad \spE^{(\le 2)}(t)\les \ve t^{C\ve},
\end{equation*}
\item  the Morawetz energy estimates:
\begin{equation*}
\Kb^{(\le 1)}(\rho)\les \ve \rho^{C\ve}, \quad \K^{(\le 1)}(t)\les \ve t^{C\ve},
\end{equation*}
\item the energy estimates for $\phi$:
\begin{equation*}
\Ephi{\le 3}(\rho)+\Ephi{\le 3}(t)\les \ve,\quad  \Edphi{\le 3}(\rho) \les \ve \rho^{C\ve},
\quad \Edphi{\le 3}(t)\les \ve t^{C\ve},
\end{equation*}
where  $C>0$ is a universal constant\begin{footnote}{For universal constant, we mean the constant
depends on the initial data in Theorem \ref{2.15.1} }\end{footnote}.
\item  In the sequel, we give results on the asymptotic behavior of the Weyl components and the deformation
tensors\begin{footnote}{Norms are taken by the appropriate induced metrics, i.e. $\gb$ for $\H_\rho$-tangent
tensor fields, $g$ for $\Sigma_t$-tangent tensor fields, the induced metric $\ga$ on $S_{t,\rho}$ for
$S_{t,\rho}$-tangent tensor.}\end{footnote}.  There is a universal constant  $c>0$ such that for $\delta=c\ve$
the following results hold\begin{footnote}{We may assume $\delta<\frac{1}{6}$ which can be achieved because
$\ve$ can be sufficiently small.}\end{footnote}
\begin{enumerate}
 \item For the Weyl components in Definition  \ref{3.18.19} with $\Psi=W$, we list two sets of asymptotic behavior
 in the following table
\begin{table}[ht] 
\begin{center}
\begin{tabular}{ | c | c| c | c | c |  }\hline
		 $\ab $& $\udb$ &$(\varrho, \sigma)$&$\b$&$\a$ \\[0.9ex]
 \hline
 $\ve\rho^\d t^{-1}u^{-2}$ & $\ve \rho^\d t^{-\frac{3}{2}}u^{-\frac{3}{2}}$&$\ve \rho^\d t^{-2}u^{-1}$& $\ve\rho^\d t^{-\frac{5}{2}} u^{-\f12}$& $\ve\rho^\d t^{-3}$ \\[0.8ex]
 $\ve t^{-1} u^{-\frac{3}{2}}$ & $\ve t^{-\frac{3}{2}} u^{-1}$&$ \ve t^{-2}u^{-\f12} $& $\ve t^{-\frac{5}{2}}$&$ \ve t^{-\frac{5}{2}}$ \\[0.8ex]
 \hline
\end{tabular}
\end{center}
 \end{table}

\item For the scalar field $\phi$ there hold
\begin{equation}\label{4.30.3.16}
\bt^\frac{3}{2}\left|\phi, \sn \phi, L\phi, \frac{\rho}{t} \Lb \phi\right|+\bt^{\frac{3}{2}-\d} |\Lb \phi|\les \ve.
\end{equation}
\end{enumerate}
\item For $\pt$ and $k$ there hold the estimates
\begin{align}
&\sup_{\Omega_\rho} \left|t\rho (\hk, \tr k -\frac{3}{\rho}) \right|\les \ve\l\rho\r^\d, \label{hk1}\\
&\sup_{\Omega_\rho} \left|t^\frac{3}{2}\hk_{j\Nb}\right|\les \ve, \label{hk2}\\
&\sup_{U_t}\left|t^{\frac{3}{2}}\bd n \right|\les \ve t^{3\d}, \label{lap}\\
&\sup_{U_t}\left|t u^\f12 \pt_{\bi\bj}\right|\les \ve t^{3\d}, \label{pt1}\\
& \sup_{U_t}\left|t^\frac{3}{2}\pt_{\bj\bN}\right|\les  \ve t^{3\d}. \label{pt2}
\end{align}
\end{enumerate}

\end{theorem}
Below we list the results concerning the local energy estimates.

\begin{theorem}[Local-in-time estimates]\label{5.13.1.16}
For $0<t\le C^* T$ with the fixed constant $C^*$ specified  in (\ref{5.13.5.16}),  there hold
\begin{equation}\label{lpt1}
\sup_{\Sigma_t\cap \I^+(\O)} \left| \pt_{\bi\bj}, \bd n, n-1 \right|\les \ve
\end{equation}
and
\begin{equation*}
\sE[W](t)+\sE[\bd\rp{\le 4}\bR](t)+\sE[\bd\rp{\le 6} \phi](t)\les \ve,
\end{equation*}
which, together with the Sobolev embedding, implies give the following estimates
\begin{equation}\label{baloceng}
\sup_{\Sigma_t\cap \I^+(\O)} \left|\bd \phi, \phi, \bR_{\a\b\ga\delta}\right|\les \ve,
\end{equation}
where for the norm $|\cdot|$ of the Riemann curvature, we mean $|\cdot|_h$.
\end{theorem}

As a consequence of (\ref{baloceng}), we will prove the following result at the end of this section.

\begin{proposition}\label{4.27.1.26}
In  $\I^+(\O)\cap \{t\le C^* T\}$, there hold
\begin{align}
&\left|\frac{\rho}{t}(\tr k-\frac{3}{\rho}, \hk )\right|\les \ve, \label{lhk1}\\
&\left|\hk_{\Nb A}\right|\les \ve.\label{lkn}
\end{align}
\end{proposition}

As a complement of the corresponding set of estimates in Theorem \ref{2.15.2}, we will prove the following
result in Section \ref{almost}. It is worthy to point out that the intrinsic null tetrad $\{L, \Lb, e_A, A=1,2\}$
is not spherically symmetric in $Z^s$.

\begin{proposition}\label{3.18.1.16}
In $Z^s\cap \{\rho\le \rho_*\}$ we have $r \approx t$,  $\a=\ab$, $\b=\udb$ and $\sigma =0$. Moreover
\begin{equation}\label{3.18.0}
|\a, \ab|\les \ve \bt^{-7},\quad \left|\varrho+\frac{4M}{(r+2M)^3}\right|\les\ve \bt^{-7},  \quad |\beta, \udb|\les \ve \bt^{-5}.
\end{equation}
All these null components are convergent to their schwarzschild value relative to the standard null frame
in $Z^s$.\begin{footnote}{For the standard frame and the schwarzschild value, we refer the reader to
Lemma \ref{4.28.17.16}.}\end{footnote}
\end{proposition}

\subsection{Preliminary estimates on Hyperboloids}

We first recall the following simple transport lemma.

\begin{lemma}\label{tsp1}
\begin{enumerate}[leftmargin = 0.6cm]
\item Suppose $F$ is an $\H_\rho$-tangent tensor field verifying the transport equation
\begin{equation}\label{3.15.13}
\bd_\fB F+\frac{m}{\rho} F=H\c F+G,
\end{equation}
where $m\in \mathbb Z$,  $G$ is a tensor field of the same type as $F$, and $H$ is a tensor field satisfying
\begin{equation}\label{3.20.5}
\sup_{V\in {\mathbb H}_1}\int_{\rho_1}^\rho |H|(\rho', V) d{\rho'}\les 1,\quad  \forall \rho>\rho_1\ge 0.
\end{equation}
Then, for the weights $\upsilon = \left(\frac{\tau}{\rho}\right)^{\la}\left(\frac{(\tau^2-\rho^2)^{1/2}}{\rho}\right)^{\la'}$
with constants $\la, \la'\in {\mathbb R}$, there holds
\begin{align}\label{3.15.14}
|\rho^m (\upsilon F)|(\rho,V)&\les\lim_{\rho\rightarrow \rho_1} \upsilon \rho^m|F|(\rho, V)
+ \int_{\rho_1}^\rho|{\rho'}^m \upsilon G(\rho', V)| d\rho'.
\end{align}

\item The same result holds when $\bd_\fB F$ is replaced by $\sn_\fB F$ if $F$ is tangent to $S_{t,\rho}$; as well as
when $\frac{m}{\rho}$ in (\ref{3.15.13}) is  replaced by $\tr k$ if (\ref{3.20.5}) holds for $H=\tr k-\frac{3}{\rho}$.
\end{enumerate}
\end{lemma}

\begin{proof}
It follows by a standard ODE argument. See \cite[Lemma 13.1.1]{CK} and \cite[Section 5]{roughgeneral}.
\end{proof}

Observe that along the geodesic $\Upsilon_V$ parametrized by $\rho$ we have $d t=\frac{\bb^{-1}n^{-1}t}{\rho}d\rho$
as well as $d\tau=\frac{\tau}{\rho}d\rho$.  This implies that
\begin{equation}\label{3.20.1}
\begin{split}
&\int_{\rho_1}^{\rho_2} f (\Upsilon_V(\rho')) d\rho'
 =\int_{t(\rho_1, V)}^{t(\rho_2, V)} f(\Upsilon_V(\rho')) \frac{\rho'}{\bb^{-1}n^{-1}t'}d t',\\
& \int_{\rho_1}^{\rho_2} f (\Upsilon_V(\rho')) d\rho' = \int_{\tau(\rho_1, V)}^{\tau(\rho_2, V)} f(\Upsilon_V(\rho')) \frac{\rho'}{\tau'}d\tau'.
\end{split}
\end{equation}

\begin{proposition} \label{prl_1}
\begin{enumerate}[leftmargin = 0.7cm]
 \item In $\I^+(\O) \cap \{\rho\le \rho_*\}$ there hold
\begin{equation}\label{3.21.16}
\bb^{-1}\approx 1,\quad t\approx \tau, \quad \frac{t}{\rho}(V)\approx V^0,\quad |n-1|\les \ve.
\end{equation}
\item In $\I^+(\O)\cap \{\rho\le \rho_*\}$ there hold
\begin{align}
& \bt^{\frac{1}{2}-3\d}|\bb^{-1}-n|\les \ve, \label{b1}\\
& \left|\log \frac{t}{\tau}-\log n^{-1}(\O)\right|\les \ve,\label{4.15.12.16}
\end{align}
and
\begin{equation}\label{bb2}
 \bt|\bb^{-1}-n|\les \ve\ln\bt \quad  \mbox{ in } Z^s\cap \{\rho\le \rho_*\}.
\end{equation}

\end{enumerate}
\end{proposition}
\begin{remark}
To prove Proposition \ref{4.27.1.26} and to establish the results in Section \ref{almost} and \ref{wza}, we only
employ the results  in Proposition \ref{prl_1} for $t\le C^* T$ or in $Z^s$, which depends merely on Theorem \ref{5.13.1.16}.
\end{remark}

\begin{proof}
We first note that $|n-1|\les \ve$ can be obtained by integrating along the integral curve of $\bT$
and using (\ref{lpt1}), (\ref{lap}) and $\bT n=0$ in $Z^s$. \begin{footnote}{ The way presented here is
for the purpose of completion in the framework of this paper. The actual control on $n$ are based on
elliptic estimates coupled with the control on $\pib_{\bi\bj}$. }\end{footnote}

We next prove (\ref{3.21.16})-(\ref{4.15.12.16}) in the region $\I^+(\O)\cap \{t\le C^* T\}$
by a bootstrap argument. Because of (\ref{init}), we may make the bootstrap assumptions
that in $\I^+(\O)\cap \{t\le C^*T\}$ there hold
\begin{align}
&\left|\log \frac{t}{\tau}-\log n^{-1}(\O)\right|\le \dn, \label{4.7.3}\\
&\int_0^\rho\frac{|\bb^{-1}n^{-1}-1|}{\rho'}d\rho'\le \dn, \label{3.21.13}\\
&|\bb^{-1} n^{-1}-1|\le  \dn, \label{3.21.17}
\end{align}
where $\dn=2c_0\ep$ with a universal constant $c_0>0$ to be specified. We will improve these estimates with
$\dn$ replaced by $\f12 \dn$. By  (\ref{4.7.3}), (\ref{3.21.17}) and $n\approx 1$ we have in
$\I^+(\O) \cap \{t\le C^*T\}$ that
\begin{equation}\label{ctb.1}
  t\approx \tau, \quad \tir\le \bb^{-1}t\approx t.
\end{equation}
We claim that in $\I^+(\O)\cap \{t\le C^*T\}$ there holds
\begin{equation}\label{3.21.bi}
|\tau^{-1}(\bb^{-1} n^{-1}-1)|\les\ve.
\end{equation}
To see this, for the function $f = n - \bb^{-1}$ we may use (\ref{Bb1}) to obtain
\begin{equation}\label{3.21.12}
\p_\rho f+\frac{f}{\rho}=H f+G,
\end{equation}
where $H= -\frac{n^{-1}\bb^{-1}-1}{\rho}$ and
\begin{equation}\label{3.22.1}
G=-\frac{\tir}{t} a^{-1}\pib_{\bN\bN}+\frac{\tir\bb^{-1}}{\rho}\l \bd_\bT \bT, \bN\r+\p_\rho n.
\end{equation}
Noting that (\ref{init}) implies $f(\O)=0$ and (\ref{3.21.13}) implies $\int_0^\rho |H| d\rho'\le \Delta_0$,
we may apply Lemma \ref{tsp1} to obtain (\ref{3.21.bi}) if we can show
\begin{equation}\label{3.21.11}
\rho^{-1}\int_0^\rho \rho' |G|(\rho', V) d\rho'\les\ve\tau.
\end{equation}

Now we prove (\ref{3.21.11}). In view of (\ref{fbtn}) and (\ref{ctb.1}) we have
\begin{equation}\label{3.21.10}
\left|\frac{\rho \p_\rho n }{t}\right|= \left|\frac{\bb^{-1} t\bT n+\tir \bN(n)}{t}\right| \les |\bd n|.
\end{equation}
By virtue of (\ref{3.22.1}), (\ref{3.20.1}) and (\ref{ctb.1}), we then obtain
\begin{align*}
\rho^{-1} \int_0^\rho \rho' |G|(\rho',V) d\rho'
&\les \int_0^\rho\left|\frac{\tir'}{\rho'} \left(-\frac{\tir'}{t'}\pib_{\bN\bN}+ \bb^{-1}\l \bd_\bT \bT, \bN\r\right)
+\p_\rho n\right| d\rho'\\
&\les \int_0^{t(\rho,V)} \left(\frac{\tir'}{t'}\left(\frac{\tir'}{t'}|\pib_{\bN\bN}|+\bb^{-1}|\nab_\bN\log n|\right)
+\left|\frac{\rho'\p_\rho n }{t'}\right|\right) d t',
\end{align*}
where we used the fact that $\l \bd_\bT \bT, e_\bi\r=\nab_{e_\bi} \log n$.
With the help of (\ref{ctb.1}), (\ref{3.21.10}) and (\ref{lpt1}), we thus obtain (\ref{3.21.11}).
Consequently (\ref{3.21.bi}) is proved in the region  $\I^+(\O)\cap \{0<t\le C^*T\}$.

By (\ref{3.20.1}), (\ref{3.21.bi}) and (\ref{ctb.1}),
\begin{equation}\label{4.15.2}
\int_0^\rho \frac{|\bb^{-1}n^{-1}-1|}{\rho'} d\rho'=\int_0^{\tau(\rho, V)} \frac{|\bb^{-1}n^{-1}-1|}{\tau'} d\tau'
\les \tau \ve\le C_1 T \ve.
\end{equation}
Combining this estimate with (\ref{ctt}) and (\ref{init}),  we can obtain
\begin{equation}\label{4.7.4}
\left|\log \frac{t}{\tau}-\log n^{-1}(\O)\right|\le C_1 T\ve.
\end{equation}
Moreover,  by (\ref{3.21.bi}) and (\ref{ctb.1}) we have
\begin{equation}\label{4.15.1}
|\bb^{-1}n^{-1}-1|\les\ve\tau\le C_1 T\ve
\end{equation}
Thus, if we take $\Delta_0 = 2 c_0 \ve$ with $c_0 = C_1 T$, then (\ref{4.7.4}),  (\ref{4.15.2}), and  (\ref{4.15.1})
improve (\ref{4.7.3}), (\ref{3.21.13}) and (\ref{3.21.17}) respectively with $\dn$ replaced by $\f12 \dn$.

Combining the above estimate  with (\ref{ctb.1}) and $n\approx 1$,  (\ref{3.21.16})-(\ref{4.15.12.16}) are proved for $t\le C^*T$.

Next we prove the estimates (\ref{3.21.16})-(\ref{4.15.12.16}) in the region with $t\ge C^* T$.
Due to (\ref{4.7.4}) and (\ref{4.15.1}), we may make the bootstrap assumptions
\begin{align}
&t^{\frac{1}{2}-3\d}|\bb^{-1}n^{-1}-1|\le \dn \label{4.7.5}\\
&\left|\log \frac{t}{\tau}-\log n^{-1}(\O)\right|\le \dn.\label{4.7.6}
\end{align}
for $ t>C^* T$, where $\dn =2 c_1\ve$ with a universal constant $c_1>0$ to be specified.
We will improve these two estimates by showing that $\dn$ on the right hand side can be replaced by $\f12 \dn$.
From (\ref{4.7.5}), (\ref{4.7.6}) and the last estimate in (\ref{3.21.16}), it follows that
\begin{equation}\label{4.7.8}
n\approx 1,\quad  \bb^{-1}\approx 1,\quad t\approx \tau.
\end{equation}
For any $p=\exp_\O(\rho V)$, with $q=\exp_\O(\rho_0 V)\subset \{t=T\}$,  by integrating along the geodesic
from $q$ to $p$,  we may use (\ref{4.7.5}), (\ref{3.20.1}) and (\ref{4.7.8}) to derive that
\begin{equation}\label{4.7.7}
\int_{\rho_0}^\rho \frac{|\bb^{-1}n^{-1}-1|}{\rho'} d\rho'\les \dn\int_{T}^{t(p)}\frac{\l t'\r^{-\frac{1}{2}+3\d}} {\bb^{-1}n^{-1}t'} dt'\les \dn
\end{equation}
Thus, we may apply Lemma \ref{tsp1} to (\ref{3.21.12}) with the weight $\upsilon=\frac{\tau}{\rho}$ to deduce that
\begin{align}
|\tau f(\rho, V)|&\les\ve+\int_{\rho_0}^\rho \tau' |G|(\rho',V) d\rho'\nn\\
&\les\ve+\int_T^{t(p)} \tau'\frac{\rho'}{\bb^{-1}n^{-1} t'}
\left|-\frac{\tir'}{t'} a^{-1}\pib_{\bN\bN}+\frac{\tir'\bb^{-1}}{\rho'}\l \bd_\bT \bT, \bN\r+\p_\rho n \right| d t'\nn\\
&\les \ve+\int_T^{t(p)}\left(\tir'(|\pib_{\bN\bN}|+|\l \bd_\bT \bT, \bN\r|)+|\rho \p_\rho n| \right)d t', \label{6.4.1.16}
\end{align}
where we employed the first identity in (\ref{3.20.1}), (\ref{4.7.8}) and $\tir'\les t'$. In the above derivation, we
used the fact
\begin{equation*}
|\tau f(\rho_0, V)|\les \ve,
\end{equation*}
which is a consequence of (\ref{3.21.bi}) and (\ref{ctb.1}). Now consider the right hand side of (\ref{6.4.1.16})
with the help of (\ref{3.21.10}). Note that in $\hat{I}^+_0$ we can use (\ref{lap}) and (\ref{pt2}) and
in $Z^s$ we have $\pib=0$ and $|\bd n|\les \ve r^{-2}\les \ve \bt^{-2}$.\begin{footnote}{The fact that $r^{-1}\les \bt^{-1}$
is explained in (\ref{4.9.6})}\end{footnote}
Hence we have
\begin{equation*}
|\tau f(\rho,V)|\les \ve+\ve \bt^{\f12+3\d}
\end{equation*}
which, in view of (\ref{4.7.8}), shows that
\begin{equation}\label{4.7.10}
\bt^{\frac{1}{2}-3\d}|\bb^{-1}-n|\le C \ve.
\end{equation}
In particular in $Z^s$, by repeating the bootstrap argument and using Remark \ref{4.9.2.16} we can obtain
\begin{equation*}
\bt|\bb^{-1}-n|\les \ve\ln \bt
\end{equation*}
which is (\ref{bb2}).

By using (\ref{ctt}), (\ref{3.20.1}), (\ref{4.7.8}), (\ref{4.7.4}),  (\ref{4.7.10}) and (\ref{bb2}), we also have
\begin{equation}\label{bb3}
\left|\log (\frac{t}{\tau})-\log n^{-1}(\O)\right|
\les \ve+\int_{T}^{t(\rho, V)}\left|\frac{\bb^{-1}n^{-1}-1}{\rho'}\right|\frac{\rho'}{\bb^{-1}n^{-1}t'} d t'\le C \ve.
\end{equation}
Therefore, if we take $\dn = 2 C \ve$, then (\ref{4.7.10}) and (\ref{bb3}) improve (\ref{4.7.5}) and
(\ref{4.7.6}) respectively in $\hat I^+_0$. Similarly, we can obtain the same estimate in $Z^s$ by using (\ref{bb2}).

Finally, the first three estimates in (\ref{3.21.16}) hold   since (\ref{4.7.5}) and (\ref{4.7.6}) have been
proved in $\I^+(\O)\cap\{\rho\le \rho_*\}$.
\end{proof}

\begin{proof}[Proof of Proposition \ref{4.27.1.26}]
In view of (\ref{baloceng}), (\ref{fbtn}) and $\bb^{-1}\approx 1$ we have for $0< t\le C^*T$ that
\begin{equation}\label{2.20.8.16}
\left|\frac{\rho}{t} \Nb \phi,\frac{\rho}{t} \fB \phi,  \sn \phi, \phi\right|\les \ve.
\end{equation}
Due to the fact that $\widehat \bR_{i\fB j\fB}=\bR_{i \fB j\fB}-\frac{1}{3} \bR_{\fB\fB}\gb_{ij}$,
by virtue of  (\ref{ricci}), (\ref{2.20.8.16}), Lemma \ref{frames} and the curvature estimate in
(\ref{baloceng}), we can obtain  for $0<t\le C^* T$ that
\begin{equation}\label{2.20.9.16}
\left|\left(\frac{\rho}{t}\right)^2 \widehat{\bR}_{A\fB C\fB},\, \left(\frac{\rho}{t}\right)^2 \widehat{\bR}_{\Nb\fB\Nb\fB},
\,\left(\frac{\rho}{t}\right)^2 \bR_{\fB\fB},\,  \frac{\rho}{t} \widehat{\bR}_{\Nb \fB C\fB}\right|\les \ve.
\end{equation}
Note that by local expansion, along $\sf{p}=\exp_{\O}(\rho V)$ for $V\in {\mathbb H}_1$ we have (see \cite[Section 6.2]{Eric})
\begin{equation*} 
\bd_\nu S_\mu(p)=\bg_{\mu\nu}(p)+O(\rho^2)
\end{equation*}
which gives
\begin{equation}\label{4.28.1.16}
\tr \, k-\frac{3}{\rho},\   \hk \rightarrow 0 \quad \mbox{ as } \rho\rightarrow 0.
\end{equation}
We will prove (\ref{lhk1}) and (\ref{lkn}) by a bootstrap argument. According to (\ref{4.28.1.16}),
we may make bootstrap assumptions that
\begin{align}
&\left|\frac{\rho}{t}\left(\tr k-\frac{3}{\rho}, \hk\right)\right|\le \dn, \label{4.28.3.16}  \\
&\left|\hk_{\Nb A}\right|(\exp_{\O} (\rho V))\le \dn ,\mbox{ if } V^0>3 \label{4.28.11.16}
\end{align}
for $t\le C^* T$, where $\dn>0$ is a small number to be chosen. We then show that
\begin{align}
&\left|\frac{\rho}{t}\left(\tr k-\frac{3}{\rho}, \hk\right)\right|\le \ti C(\Delta_0^2+\ve),\label{4.28.8.16}\\
&\left|\hk_{\Nb A}\right|\le \ti C(\Delta_0^2+\ve), \mbox{ if } V^0>3\label{4.28.12.16}
\end{align}
which improves (\ref{4.28.3.16}) and (\ref{4.28.11.16})  as long as we choose $\dn=4\ti C \ve<\frac{1}{2\ti C}$.
Then (\ref{lhk1}) is proved due to (\ref{4.28.8.16}), which implies (\ref{lkn}) for the case $V^0\le 3$.
In the region where $V^0=\frac{\tau}{\rho}>3$, (\ref{lkn}) holds true due to (\ref{4.28.12.16}).

Now we prove (\ref{4.28.8.16}). Due to Proposition \ref{prl_1}, we can obtain for $t\le C^*T$ that
\begin{equation}\label{4.28.5.16}
\bb^{-1}\approx 1, \quad t\approx \tau.
\end{equation}
As a direct consequence of (\ref{4.28.3.16}), for any $p=\exp_\O(\rho V) $  with $t(p)\le C^*T$ we have
\begin{equation}\label{4.28.4.16}
\int_0^\rho \left|\tr k-\frac{3}{\rho'}\right|d \rho'\les \dn
\end{equation}
since $\rho\le \bb^{-1}t\les T$ due to (\ref{4.28.5.16}). In view of (\ref{4.28.1.16}) and (\ref{4.28.4.16}),
we may apply Lemma \ref{tsp1} to (\ref{3.14.1}) with $H=-\frac{1}{3}(\tr k-\frac{3}{\rho})$,
$\upsilon=\frac{\tau}{\rho}$ and $m=2$ to obtain
\begin{equation}\label{4.28.13.16}
\begin{split}
\left|\rho\tau \left(\tr k-\frac{3}{\rho}\right)\right| &\les \int_0^\rho (\bR_{\fB \fB}+|\hk|^2) \rho'\tau' d\rho'
\les \int_0^t (\bR_{\fB \fB}+|\hk|^2) \rho' \tau'\frac{\rho'}{\bb^{-1}t'} d t'\\
&\les \int_0^t {t'}^2 (\ve+\Delta_0^2)dt'\les (\ve+\Delta_0^2) t^3,
\end{split}
\end{equation}
where we employed (\ref{4.28.5.16}), (\ref{2.20.9.16}) and (\ref{4.28.3.16}). This implies that
\begin{equation}\label{4.28.6.16}
\left|\frac{\rho}{t}\left(\tr k-\frac{3}{\rho}\right)\right|\les T(\ve+\Delta_0^2)\les \ve+\Delta_0^2.
\end{equation}
Similarly, by repeating the above argument, and using the transport equation  (\ref{s1.1}),
the estimates in  (\ref{2.20.9.16}) and the initial condition (\ref{4.28.1.16}) we can derive that
\begin{equation*} 
\left|\frac{\rho}{t} \hk\right|\les \ve+\Delta_0^2
\end{equation*}
which gives the control on $\hk_{ij}$ in (\ref{4.28.8.16}). Thus (\ref{4.28.8.16}) is proved.

Next we prove (\ref{4.28.12.16}). By using (\ref{b1}), (\ref{4.15.12.16}) and the last estimate in (\ref{3.21.16}), we have $|\bb^{-1}\frac{t}{\tau}-1|\les \ep$ in $\I^+(\O)\cap \{t\le C^*T\}$.  Thus  for sufficiently small $\ve$ we have
$\bb^{-1}t>2\rho$ along $\exp_\O (\rho V)$ if $t\le C^*T$ and $V^0>3$. This implies $\tir \approx t$ therein.
Hence, in view of  (\ref{zba}) and (\ref{lpt1}), we can obtain
\begin{equation}\label{4.28.10.16}
|\zb|\les \ve.
\end{equation}
We now apply Lemma \ref{tsp1} to (\ref{2.20.5.16}), with the help of (\ref{lhk1}) and (\ref{2.20.9.16}).
Similar to  (\ref{4.28.13.16}), by integrating (\ref{2.20.4.16}) along $\exp_\O (\rho V)$ with $V^0>3$
and using the initial condition in (\ref{4.28.1.16}), it follows that
\begin{equation*} 
\tau^2| \hk_{\Nb A}|\les \int_0^t \rho\tau' \left((|\zb|+|\hk_{\Nb A}|)\c |\Ab|+|\widehat{\bR}_{\fB \Nb \fB A}|\right)d t'.
\end{equation*}
By using (\ref{4.28.8.16}), (\ref{2.20.9.16}), (\ref{4.28.11.16}) and  (\ref{4.28.10.16}), we can obtain
$ 
\tau^2 |\hk_{\Nb A}|\les(\ve+\Delta_0^2) t^3
$ 
which implies (\ref{4.28.12.16}). Thus the proof of Proposition \ref{4.27.1.26} is completed.
\end{proof}

\section{\bf Radial comparison in $Z^s$}\label{almost}

In this section, we compare the canonical schwarzschild wave front with the wave fronts formed by
intersection in $\Sigma_t$ by hyperboloids.

The core analysis will be radial comparisons which  takes place in  the initial slice $\{t=T\}$ and in the
Schwarzchild zone $Z^s$. We will use the parametrization $(t, r, \omega)$, where $(r, \omega)$ denotes the
polar coordinates with  $\omega= (\theta, \varphi)\in {\mathbb S}^2$.
The Schwarzschild metric $\bg$ in (\ref{3.20.lap}) can be written as
\begin{align}
\bg &=-n^2 d t^2+n^{-2} dr^2+(r+2M)^2 ( d\theta^2+\sin^2\theta d\varphi^2), \label{scm}
\end{align}
where $n^2 = \frac{r-2M}{r+2M}$. Let $\Ga$ denote the Christoffel symbol of $\bg$. Direct calculation shows that
\begin{equation*}
\begin{array}{lll}
\displaystyle{\Ga_{tt}^r= n^2 \frac{2M}{(r+2M)^2}}, &\quad \displaystyle{\Ga_{rr}^r = -n^{-2}\frac{2M}{(r+2M)^2}},\\[1.8ex]
\displaystyle{\Ga_{\theta\theta}^r = -(r-2M)}, &\quad \displaystyle{\Ga_{\varphi\varphi}^r = -(r-2M)\sin^2\theta}.
\end{array}
\end{equation*}

The lapse function
\begin{equation}\label{3.25.2.16}
\varpi :=\bN(r),
\end{equation}
is a crucial quantity for comparing the intrinsic radial normal $\bN$ on $\Sigma_t$ with the Euclidean radial normal $\p_r$.
We will carry out the comparison along $\Sigma_T$ from the center $\bo$. Then we will control
the evolution of $\varpi$ in $Z^s$ along all time-like geodesics initiating at $\{t=T, r\ge 2\}$.

Recall the optical function $\hat u$ defined in (\ref{geosch}) in Section \ref{12.10.02} whose level sets
$\C_{\hat u}^s$ are called the Schwarzchild cones. We also introduced $\C_{\hat u_0}^s$ and $\C_{\hat u_1}^s$.
Let $\hat L$ be the null geodesic generator of $\C^s_{\hat u}$, normalized by $\hat L(t)=1$. Then
\begin{equation}\label{4.3.2}
\l \bd \hat u,  \bd \hat u\r =0, \quad  \hat L=\p_t+n^2 \p_r.
\end{equation}
For $(t,x)\in Z^s$, it follows from (\ref{geosch}) that
\begin{equation*}
t\le {\hat u_1}+ 2 r, \quad \hat u_1=T-\ga(2).
\end{equation*}
Since $0<M\ll 1$, by combining the above equation with (\ref{geosch}) we have
\begin{equation}\label{4.9.6}
 r^{-1} \le \frac{2}{t-\hat u_1},  \quad \mbox{ if  }  (x,t) \in Z^s.
\end{equation}

We first derive important equations for $\varpi$. For a $\Sigma_t$ tangent vectorfield $F$, let $F^i$
be the component relative to the cartesian frame $\p_{x^i}, i=1,2, 3$. We will lift and lower the
index of a vector field by Minkowsi metric, unless specified otherwise.
Let
$$
\ne:=\p_r =\frac{x^i}{r}\p_i.
$$
We denote by $\l\cdot, \cdot\r_e$ and $\nabe$ the Euclidean metric and its connection respectively. Inspired by \cite[Page 416]{CK} and \cite[Page 162]{shock},
for $\bN$ we have the radial decomposition
\begin{equation}\label{neuc}
\bN=\Sigma N+\varpi\ne
\end{equation}
where the vector field $\Sigma N$ is tangent to the level set of $r$ and is given by $\Sigma N^i=\ude{\Pi}^i_j\bN^j $ with ${\ude{\Pi}}^{ij}=\delta^{ij}-\frac{x^i}{r}\frac{x^j}{r}$.
By using (\ref{neuc}) and (\ref{scm}) we have in $Z^s$ that
\begin{equation}\label{12.10.26}
|\Sigma N|_{\bg}^2=1-n^{-2}\varpi^2
\end{equation}
and
\begin{equation}\label{eun}
\nab r=n^2 \p_r=\sn r+\varpi \bN.
\end{equation}
Let $h_{ij}=g_{ij}-\delta_{ij}$. In view of (\ref{neuc}), (\ref{scm}) and (\ref{eun}), it is
straightforward to see that
\begin{equation}\label{3.31.4}
|\Sigma N|_e^2=1-\varpi^2-h_{ij} \bN^i \bN^j, \quad |\nab r|^2_{\gs}=n^2=\varpi^2+|\sn r|^2_{\gs}.
\end{equation}


Similarly for the orthonormal frame $\{e_A\}$ in $S_{t, \rho}\subset Z^s$, we can decompose as
\begin{equation}\label{4.16.1}
e_A=\Sigma e_A+\sn_A r \ne.
\end{equation}
Using this equation we can obtain
\begin{equation}\label{4.16.2}
e_A \log n =n^{-1}\p_r n \sn_A r=\frac{n^{-2}2M}{(r+2M)^2}\sn_A r.
\end{equation}

\subsection{Structure equations for radial comparison}

\begin{lemma}\label{scgeo}
\begin{enumerate}
\item[(i)] In $(\M, \bg)$ there holds
\begin{align}
\bN(\varpi)&=-\Ga_{mn}^i \bN^m \bN^n \ne_i- \l \sn \log a, \ne\r_e +\frac{1}{r} |\Sigma N|_e^2. \label{3.31.2}
\end{align}

\item[(ii)] In $Z^s$ there holds
\begin{equation}\label{3.31.9}
|\Sigma N|^2_e=(1-n^{-2} \varpi^2)\frac{r^2}{(r+2M)^2},
\end{equation}
and for any $\Sigma_t$-tangent vector fields $F$ satisfying $\l F, \bN\r =0$ there holds
\begin{equation}\label{3.31.5}
\l F,  \ne\r_e =\l F, \sn r\r .
\end{equation}

\item[(iii)] In $Z^s$ there hold
\begin{align}
&\bN(\varpi)=\frac{2M}{(r+2M)^2} n^{-2} \varpi^2+\frac{2(r-M)}{(r+2M)^2}(1-n^{-2} \varpi^2)-\sn_A \log a\sn_A r,\label{3.31.7}\\
&\bN(n-\varpi)=\frac{2M}{(r+2M)^2} n^{-1} \varpi (1-n^{-1} \varpi)-\frac{2(r-M)}{(r+2M)^2}(1-n^{-2} \varpi^2) \nn\\
&\qquad \qquad\quad \, +\sn_A \log a \sn_A r. \label{3.31.12}
\end{align}
\end{enumerate}
\end{lemma}


\begin{proof} (i) Noting that $\varpi = \l \bN, \ne\r_e$ and using (\ref{nabn}), we have
\begin{align*}
\bN(\varpi) & = \bN(\l \bN,  \ne\r_e )=\nabe_\bN(\l \bN,  \ne\r_e)= \l\nabe_\bN \bN,  \ne\r_e +\l \bN,  \nabe_\bN\ne\r_e\\
&=\l \nabe_\bN \bN-\nab_\bN \bN, \ne\r_e+\l\nab_\bN \bN, \ne \r_e+ \l \bN,  \nabe_\bN \ne\r_e \\
&=-\Ga_{mn}^i\bN^m \bN^n \ne_i-\l \sn\log a, \ne\r_e + \l \bN, \nabe_\bN \ne\r_e.
\end{align*}
For the last term, we further employ (\ref{neuc}) to derive
\begin{align*}
\l \bN,  \nabe_\bN \ne\r_e &= \l \bN, \nabe_{\Sigma N+\varpi \ne} \ne\r_e  = \l \bN,  \nabe_{\Sigma N}\ne\r_e
=\l \Sigma N +\varpi \ne,  \nabe_{\Sigma N} \ne\r_e\\
&=\l \Sigma N, \nabe_{\Sigma N} \ne\r_e =r^{-1}\l\Sigma N, \Sigma N\r_e. 
\end{align*}
Combining the above two equations we therefore obtain (\ref{3.31.2}).

(ii) From (\ref{scm}) we can see that
\begin{equation*} 
|\Sigma N|^2_{\gs}=|\Sigma N|_e^2 \frac{(r+2M)^2}{r^2}.
\end{equation*}
This together with (\ref{12.10.26}) shows (\ref{3.31.9}). With the Schwarzschild metric (\ref{scm}) in $Z^s$
and $\l F, \bN\r =0$ we have
\begin{align*}
\l F,  \ne\r_e =n^2 \l F,  \ne\r =\l F, n^2 \p_r\r =\l F, \varpi \bN+\sn r\r =\l F, \sn r\r
\end{align*}
which shows (\ref{3.31.5}).

(iii) We first prove (\ref{3.31.7}) using (\ref{3.31.2}). In view of (\ref{neuc}) we have
\begin{align*}
\Ga_{mn}^i \bN^m \bN^n \ne_i &= \Ga_{mn}^r(\varpi\ne^m+\Sigma N^m)(\varpi \ne^n+\Sigma N^n) \\
&= \Ga_{mn}^r \left(\varpi^2 \ne^m \ne^n+\varpi \Sigma N^m \ne^n+\varpi\ne^m \Sigma N^n+\Sigma N^m \Sigma N^n\right).
\end{align*}
For Schwarzchild metric in $Z^s$ we have $\Gamma_{mn}^r=0$ when $m\neq n$. Since $\ne=\p_r$ and
$\Sigma N$ is tangent to the level set of $r$, we have $\Ga_{mn}^r \Sigma N^m \ne^n =0$. Therefore
\begin{align*}
\Ga_{mn}^i \bN^m \bN^n \ne_i &=\Ga_{mn}^r\left(\varpi^2 \ne^m \ne^n +\Sigma N^m \Sigma N^n\right).
\end{align*}
Note that
\begin{align*}
\Ga_{mn}^r \ne^m\ne^n = \Ga_{rr}^r= -\frac{2M n^{-2}}{(r+2M)^2}, \quad
\Ga_{mn}^r \Sigma N^m \Sigma N^n= -\frac{r-2M}{r^2} |\Sigma N|^2_e,
\end{align*}
we have
\begin{equation*} 
\Ga_{mn}^i \bN^m \bN^n \ne_i=-\frac{2M}{(r+2M)^2} n^{-2} \varpi^2-\frac{r-2M}{r^2}|\Sigma N|_e^2.
\end{equation*}
Combining this equation with (\ref{3.31.2}) and using (\ref{3.31.5}), (\ref{3.31.9}) we thus obtain (\ref{3.31.7}).
Finally, by using (\ref{12.10.26}) we have
\begin{equation}\label{3.31.13}
\bN(n)= \varpi \ne(n)=\frac{\varpi}{2n} \p_r (n^2)=\frac{ 2 n^{-1} \varpi M}{(r+2M)^2}.
\end{equation}
This together with (\ref{3.31.7}) shows (\ref{3.31.12}).
\end{proof}

\begin{lemma}In the Schwarzschild zone $Z^s$ there hold
\begin{align}
&\hat L(u)=(n-\varpi)-u \sn_A r k_{A\Nb}+u \ckk k_{\Nb \Nb}\frac{(n-\varpi)\tir -u \varpi}{\rho}+u n \l\bd_\bT \bT, \bN\r, \label{4.3.3s}\\
&\fB(n-\varpi)+2\frac{\tir}{\rho}\frac{r-M}{(r+2M)^2}(1-n^{-2} \varpi^2)\nn\\
&\quad \quad \quad \quad \quad\quad
=\frac{2M}{n^2(r+2M)^2} \left(\frac{\tir}{\rho}\varpi+\frac{\bb^{-2}t^2}{\rho \tir} (n+\varpi)\right)(n-\varpi).\label{bvarpi}
\end{align}
\end{lemma}

\begin{proof}
We first prove (\ref{4.3.3s}). By using (\ref{eun}) we have in $Z^s$ that
\begin{equation}\label{4.3.12}
\hat L=\p_t+n^2 \p_r=n \bT+\varpi \bN +\sn r.
\end{equation}
In view of (\ref{4.3.4}), (\ref{4.3.5}), (\ref{3.5.3}) and (\ref{3.6.4}), we can derive that
\begin{align*}
\hat L(u)&=n-\varpi+n u \left(a^{-1} {\ckk k}_{\Nb\Nb}+\l \bd_\bT \bT, \bN\r\right)
-\varpi u\left(\pib_{\bN \bN}+\frac{\bb^{-1}t}{\rho} \ckk k_{\Nb\Nb}\right)\\
&\quad \,  - u\left(k_{A \Nb}+\pib_{A\bN}\right)\sn_A r.
\end{align*}
By using $a^{-1} = \tir/\rho$ and $u = \bb^{-1} t - \tir$, this shows (\ref{4.3.3s}) since
$\pib_{\bi\bj}$  vanishes  in the schwarzschild zone $Z^s$.

We next prove (\ref{bvarpi}). By definition of $\varpi$ we have
\begin{equation}\label{4.3.14}
\bT(\varpi)=\bT(\bN(r))=[\bT, \bN]r=(\bd_\bT \bN^\mu-\bd_\bN \bT^\mu) \p_\mu r.
\end{equation}
Note that $\l\bd_\bT \bN, \bN\r=0$ and $\l \bd_\bT \bN, \bT\r=-\l \bN, \bd_\bT\bT\r$.
By using (\ref{fbtn}) and (\ref{dcp_3}) we also have
\begin{align*}
\l \bd_\bT \bN, e_A\r&=\l \bd_\bT(a\c a^{-1}\bN ), e_A\r=\l a \bd_\bT (\fB-\frac{\bb^{-1}t}{\rho}\bT), e_A\r\\
&=a\l\bd_\bT \fB, e_A\r-\frac{a\bb^{-1}t}{\rho}\l \bd_\bT \bT, e_A\r\\
&=-\l \bd_\Nb \fB, e_A\r-\frac{a\bb^{-1}t}{\rho}\l \bd_\bT \bT, e_A\r \\
& = -k_{\Nb A}-\frac{\bb^{-1}t}{\tir} \l \bd_\bT \bT, e_A\r.
\end{align*}
Thus, it follows from (\ref{4.3.14}) that
\begin{align*}
\bT\varpi=-k_{A\Nb} \sn_A r-\frac{\bb^{-1}t}{\tir}\sn_A \log n \sn_A r
+\left(\pib_{\bN\bN}\varpi+\pib_{\bN A}\sn_A r\right). 
\end{align*}
This together with (\ref{3.31.12}) gives
\begin{align*}
&\fB(n-\varpi)+2\frac{\tir}{\rho}\frac{r-M}{(r+2M)^2}(1-n^{-2} \varpi^2)\\
& \quad \quad =\frac{\bb^{-1}t}{\rho}\left(k_{A\Nb}\sn_A r +\frac{\bb^{-1}t}{\tir} \sn_A \log n \sn_A r
- \left(\pib_{\bN\bN} \varpi+\pib_{\bN A}\sn_A r\right)\right)\\
& \quad \qquad +\frac{\tir}{\rho}\left(\frac{2M}{(r+2M)^2} n^{-1}\varpi(1-n^{-1}\varpi)
+\sn_A \log a \sn_A r\right). 
\end{align*}
By using (\ref{3.6.7}), we can get the cancelation between the first and the last term on the right hand
side. Therefore
\begin{align*}
&\fB(n-\varpi)+2\frac{\tir}{\rho}\frac{r-M}{(r+2M)^2}(1-n^{-2} \varpi^2)=\frac{\bb^{-1}t}{\rho}(\frac{\bb^{-1}t}{\tir} \sn \log n \sn r\nn\\
&\quad \quad \quad-(\pib_{\bN\bN} \varpi+2\pib_{\bN A}\sn r))
+\frac{\tir}{\rho}\frac{2M}{(r+2M)^2} n^{-1}\varpi(1-n^{-1}\varpi).
\end{align*}
Since the the metric in $Z^s$ is Schwarzschild,  we may use (\ref{4.16.2}), (\ref{3.31.4}) and the vanishing of
$\pib_{\bi\bj}$ to obtain (\ref{bvarpi}).
\end{proof}

\begin{lemma}
In the Schwarzchild zone $Z^s$ there holds
\begin{align}
\fB\left( \frac{\tir}{r}-n^{-1}\right)+\frac{1}{\rho}\left(\frac{\tir}{r}-n^{-1}\right)
&=-\left(\frac{n}{\rho}\left(\frac{\tir}{r}-n^{-1}\right)+\frac{\tir}{\rho}\bN (\log n)\right) \left(\frac{\tir}{r}-n^{-1}\right) \nn \\
&+\frac{\bb^{-1}t \tir}{\rho r}\pib_{\bN\bN}+\frac{\tir^2}{r^2\rho}(n-\varpi)-\frac{\rho}{r}\bN (\log n). \label{cmr_1}
\end{align}
\end{lemma}

\begin{proof}
From Lemma \ref{frames} we have $\fB(r)=\frac{\tir}{\rho}\varpi$. Combining this with (\ref{3.15.4}) gives
\begin{equation*} 
\fB \left(\frac{\tir}{r}\right) =\frac{\tir}{r \rho}\left(1-\frac{\tir}{r}\varpi\right)
+\frac{\bb^{-1}t}{r}\left(a^{-1}\pib_{\bN\bN}-\frac{\bb^{-1}t}{\rho}\l \bd_\bT \bT, \bN\r\right).
\end{equation*}
Hence, by regrouping the terms we obtain
\begin{align*} 
\fB\left(\frac{\tir}{r}-n^{-1}\right)
&=\frac{1}{\rho}\left(n^{-1}-\frac{\tir}{r}\right)-\frac{n}{\rho}\left(n^{-1}-\frac{\tir}{r}\right)^2
-\frac{\tir^2}{r^2\rho}(\varpi-n)\\
&-\fB(n^{-1})+\frac{\bb^{-1}t}{r}\left(a^{-1}\pib_{\bN\bN}-\frac{\bb^{-1}t}{\rho}\l \bd_\bT \bT, \bN\r\right).
\end{align*}
Noting that, in view of Lemma \ref{frames} and $\tir^2 = (\bb^{-1} t)^2 - \rho^2$,
\begin{align*}
\frac{(\bb^{-1}t)^2}{r\rho}\l \bd_\bT \bT, \bN\r+\fB(n^{-1})&=\frac{(\bb^{-1}t)^2}{r\rho}\bN \log n+\frac{\tir}{\rho}\bN(n^{-1})\\
&=\left(\frac{\rho}{r}+\frac{\tir}{\rho}\left(\frac{\tir}{r}-n^{-1}\right)\right)\bN (\log n) 
\end{align*}
which, combining with the above equation, shows (\ref{cmr_1}).
\end{proof}

\subsection{\bf Radial comparison on the initial slice}

We will control $n-\varpi$ and $\sn r$ by using (\ref{3.31.12}) and (\ref{bvarpi}).
In order to use (\ref{bvarpi}), we need to consider the initial data for these geometric
quantities in  $r\ge 2$, which will be understood by propagating $|\sn r|^2_g$ from $\bo:=(T, 0)$
on the initial slice $\{t=T\}$ along the radial normal $\bN$. In view of (\ref{3.25.2.16}),
we can easily obtain the transport equation
\begin{equation*} 
\sn_{\bN}\sn_A r+\f12\tr \, \theta \sn_A r= \sn_A \varpi -\hat \theta_{AC} \sn_C r.
\end{equation*}
By deriving an equation for $\sn_A \varpi$ on the initial slice $\{t=T\}$ and using the above equation
we have the following result.

\begin{lemma}\label{angome}
In $Z^s \cap \{t = T\}$ there holds
\begin{align*}
\sn_{A} \varpi&= \f12\left(\emph{\tr}\,\theta-\frac{2}{r}\right)\sn_A r+\hat\theta_{AC}\sn_C r
+g_{j' i'} e_A^{i'} \left(-{\sl\Pi}^{j'j} \Ga_{j l}^i \bN^l \ne_i \right)\\
&+\frac{1}{r}g_{j'i'} e_A^{i'} \left[h^{j'j}\bN_j+(1-\varpi)\left(h^{j'j} \ne_j +\ne^{j'}(1+\varpi)+{\Sigma N}^{j'}\right)\right]
\end{align*}
and consequently
\begin{align}
\sn_{\bN}\sn_A r+\frac{\varpi^2}{r}\sn_A r
& = \frac{1}{r}g_{j'i'} e_A^{i'} \left[h^{j'j}\bN_j +(1-\varpi)\left(h^{j'j} \ne_j +{\Sigma N}^{j'}\right) \right]\nn\\
& \quad \, + g_{j' i'} e_A^{i'} \left(-{\sl\Pi}^{j'j} \Ga_{j l}^i \bN^l \ne_i \right).\label{4.4.9}
\end{align}
\end{lemma}

\begin{proof}
Note that $\varpi = \bN^i \ne_i$, we have
\begin{align}
\sn_{A} \varpi & = \sn_A(\bN^i \ne_i)= g_{j'i'}e_A^{i'} {\sl\Pi}^{j' j} \left(\p_j \bN^i \ne_i+ \bN^i \p_j \ne_i\right)\nn\\
&=g_{j'i'}e_A^{i'}{\sl\Pi}^{j' j} \left( \sn_j \bN^i \ne_i- \Ga_{jl}^i \bN^l \ne_i\right) \nn\\
&\quad \, +g_{j'i'}e_A^{i'} \left(\bN^i \left({\sl\Pi}^{j'j}-{\ude{\Pi}}^{j'j}\right)\p_j\ne_i +\bN^i{\ude\Pi}^{j'j }\p_j \ne_i )\right)\nn\\
&=g_{j'i'}e_A^{i'} {\sl\Pi}^{j' j} \left(\sn_j \bN^i \ne_i-\Ga_{jl}^i \bN^l \ne_i\right)\nn\displaybreak[0]\\
&\quad \, +g_{j'i'}e_A^{i'} \left(\bN^i(h^{j'j}-\bN^{j'} \bN^j+\ne^{j'} \ne^j) \p_j \ne_i +\bN^i{\ude\Pi}^{j'j} \p_j \ne_i \right)\nn\\
&=g_{j'i'}e_A^{i'}\left(\f12\sl{\Pi}^{j'j}\ne_j \left(\tr\theta-\frac{2}{r}\right)
 +\hat\theta_j^i {\sl\Pi}^{j'j} \ne_i-{\sl\Pi}^{j' j} \Ga_{jl}^i \bN^l \ne_i\right)\nn\\
& \quad \, +I^{(1)} + I^{(2)},\label{4.1.1}
\end{align}
where
\begin{align*}
I^{(1)} & = g_{j'i'}e_A^{i'}\bN^i \left(h^{j'j}-\bN^{j'} \bN^j +\ne^{j'} \ne^j\right) \p_j \ne_i,\\
I^{(2)} & = \frac{1}{r} g_{j'i'}e_A^{i'} \left(\sl{\Pi}^{j'j} \ne_j +{\ude\Pi}^{j'j} \bN_j\right).
\end{align*}
We will employ (\ref{neuc}) to consider these two terms. For $I^{(1)}$ we note that
$\p_j \ne_i = \p_j (\frac{x^i}{r})=\frac{1}{r} {\uda{e}\Pi}_{ij}$, $\ude\Pi_{ij} \ne^j =0$ and
$\ude\Pi_{ij} \bN^j = {\Sigma N}_i$. Therefore
\begin{align*}
I^{(1)}&= \frac{1}{r} g_{j'i'}e_A^{i'}\bN^i \left(h^{j'j}-\bN^{j'} \bN^j +\ne^{j'} \ne^j\right) {\ude\Pi}_{ij} \\
&=\frac{1}{r}g_{j'i'}e_A^{i'} h^{j'j}{\Sigma N}_j -\frac{1}{r} \l e_A, \bN\r \bN^i \Sigma N_i\\
&=\frac{1}{r}g_{j'i'}e_A^{i'}h^{j'j} {\Sigma N}_j.
\end{align*}
For $I^{(2)}$ we may use $\varpi = \bN^j \ne_j$ and (\ref{neuc}) to obtain
\begin{align*}
I^{(2)}&=\frac{1}{r}g_{j'i'} e_A^{i'} \left[(g^{j'j}-\bN^{j'}\bN^j )\ne_j+{\ude\Pi}^{j'j}\bN_j \right]\\
&=\frac{1}{r}g_{j'i'} e_A^{i'} \left[(g^{j'j}-\delta^{j'j})\ne_j+\ne^{j'}-\bN^{j'}\varpi+{\ude\Pi}^{j'j} \bN_j \right]\\
&=\frac{1}{r}g_{j'i'} e_A^{i'} \left[h^{j'j}\ne_j+\ne^{j'}-(\ne^{j'}\varpi+\Sigma N^{j'})\varpi+{\ude \Pi}^{j'j} \bN_j\right]\\
&=\frac{1}{r}g_{j'i'} e_A^{i'} \left[h^{j'j} \ne_j+\ne^{j'}(1-\varpi^2)+\Sigma N^{j'}(1-\varpi)\right].
\end{align*}
By substituting the above two formulaes into (\ref{4.1.1}), we have
\begin{align*}
\sn_{A} \varpi&= g_{j' i'} e_A^{i'} \left[\f12 \left(\tr\theta-\frac{2}{r}\right){\sl\Pi}^{j'i} \ne_i
+\hat \theta_{j}^i {\sl\Pi}^{j'j} \ne_i-{\sl\Pi}^{j'j} \Ga_{j l}^i \bN^l \ne_i \right]\\
&+\frac{1}{r}g_{j'i'} e_A^{i'}\left[h^{j'j}  \Sigma N_j+h^{j'j} \ne_j +\ne^{j'}(1-\varpi^2)+{\Sigma N}^{j'}(1-\varpi)\right].
\end{align*}
Using (\ref{4.16.1}) and the fact that $\theta$ is a $S_{t,\rho}$-tangent tensor, we have
\begin{align*}
\ne_i\sl{\Pi}^{j' i} e_A^{i'} g_{j'i'}=\l e_A, \ne\r_e= \sn_A r, \qquad
g_{j'i'}e_A^{i'}\sl{\Pi}^{j j'} \hat \theta_j^i \ne_i = \hat \theta_{AC} \sn_C r.
\end{align*}
Lemma \ref{angome} then follows by using (\ref{neuc}) again.
\end{proof}

On the initial slice $\{t = T\}$,  we use $\rf$ to replace $\rho$, where the radial parameter $\rf$
is defined in (\ref{4.5.1}). Then the induced metric on $\I^+(\O)\cap \{t=T\}$ can be written as
\begin{equation}\label{4.14.3.16}
\cir{a}^2 d{\rf}^2+\ga_{AB} d\omega^A d\omega^B,
\end{equation}
where $\cir{a} \, =\frac{\rf}{\tir}$. One has
\begin{equation}\label{4.14.25.16}
\cir{a} \,\approx 1
\end{equation}
 which follows from  (\ref{4.5.17}) in the following result.

\begin{lemma}\label{cmpr}
In $\I^+(\O)\cap \{t=T\}$ there holds
\begin{equation}
\left|1-\frac{\tir}{\rf} \right|\les \ve.\label{4.5.17}
\end{equation}
\end{lemma}

\begin{proof}
We first check that
\begin{equation}\label{4.15.5.16}
\limsup_{\rho\rightarrow \rho(\bo)} \left|1-\frac{\tir}{\rf}\right|\approx  \ve.
\end{equation}
Consider any point $q$ in a small neighborhood of $\bo$. In view of (\ref{split}), $q=(\rho, \omega)$ with $\omega\in {\mathbb S}^2$.
Local expansion around $\bo$ at $t=T$ takes the form
\begin{equation}\label{4.15.6.16}
\bb^{-1} t(\rho)=\bb^{-1}t(\bo)+\p_\rho (\bb^{-1}t )(\bo) (\rho-\rho(\bo))+O((\rho-\rho(\bo))^2)\cdots.
\end{equation}
Due to (\ref{4.1.2.16}) and (\ref{4.5.6}),
\begin{equation*} 
-\p_\rho (\bb^{-1}t )=\bb^{-1}{t}\ckk k_{\Nb\Nb}+\rho\pib_{\bN\bN}.
\end{equation*}
As $\rho\rightarrow\rho(\bo)$, noting that $\bb^{-1}t(\bo)=\rho(\bo)$
\begin{equation*} 
-\p_\rho (\bb^{-1}t)(\bo)=\rho(\bo)(\ckk k_{\Nb\Nb}(\bo)+\pib_{\bN \bN}(\bo))
\end{equation*}
where  $\Nb(\bo)(\omega)=\bN(\bo)(\omega)$ is a unit vector  at the $\T_\bo\Sigma_T$.
By substituting this equation into (\ref{4.15.6.16}) and noting $\rho(\bo)-\rho= r_{\flat}^2/(\rho(\bo)+\rho)$, we obtain
\begin{equation*} 
\bb^{-1} t(\rho)=(\bb^{-1}t)(\bo)+\rho(\bo)(\ckk k_{\Nb\Nb}(\bo)+\pib_{\bN \bN}(\bo))\frac{r_\flat^2}{\rho(\bo)+\rho}+O(r_\flat^4).
\end{equation*}
Therefore
\begin{equation*} 
(\bb^{-1}t )^2(\rho)=(\bb^{-1}t)^2(\bo)+ 2(\bb^{-1}t)^2(\bo)(\ckk k_{\Nb\Nb}(\bo)+\pib_{\bN \bN}(\bo))\frac{r_\flat^2}{\rho(\bo)+\rho}+O(r_\flat^3)
\end{equation*}
which implies, in view of (\ref{lhk1}), (\ref{lpt1}) and $\bb^{-1}\approx 1$ in (\ref{3.21.16}),
\begin{equation*} 
|\tir^2-r_\flat^2|(\rho, \omega)\les \ve r_\flat^2.
\end{equation*}
This shows (\ref{4.15.5.16}).

We now show (\ref{4.5.17}) by a bootstrap argument. According to (\ref{4.15.5.16}) we can make a bootstrap assumption
\begin{equation}\label{4.7.12}
\left|\frac{\tir}{\rf}-1\right|(\rho, \omega)\le \dn,
\end{equation}
where $0<\dn<\f12$ is a small number to be chosen.  We will show that
\begin{equation}\label{4.7.13}
\left |\frac{\tir}{\rf}-1 \right|(\rho, \omega)\le C  \ve
\end{equation}
with a universal constant $C>0$.  We can choose $\dn=2 C\ve $ and thus (\ref{4.7.13}) improves (\ref{4.7.12}).

In order to derive (\ref{4.7.13}), we will use the fact
\begin{align} \label{4.5.8}
\left|\bb^{-1} t - t_\flat\right| \les \left\{\begin{array}{lll}
\ve r_\flat^2  & \mbox{ if } \rho>\frac{\bb^{-1} t}{3}, \\ [1.2ex]
\ve r_\flat & \mbox{ if } \rho\le \frac{\bb^{-1} t}{3}.
\end{array}\right.
\end{align}
which will be proved shortly. From (\ref{4.7.12}) it follows that $\rf \approx \tir$.
In view of the definition of $\tir$ and $\rf$ we have
\begin{equation}\label{4.5.4}
\tir-\rf=\frac{\bb^{-2}t^2-\tf^2}{\tir+\rf}.
\end{equation}
Therefore, when $\rho> \frac{\bb^{-1}t}{3}$, we may use (\ref{4.5.8}), $\rf\approx \tir$ and
$\bb^{-1}\approx 1$ in (\ref{3.21.16}) to obtain
\begin{equation*} 
\left|\frac{\tir-\rf}{\rf}\right|=\left|\frac{\bb^{-1}t-\tf}{(\tir+\rf)\rf}\right| (\bb^{-1}t+\tf)\les \ve (\bb^{-1}t+\tf)\les \ve
\end{equation*}
and, when $\rho\le \frac{\bb^{-1}t}{3}$, we may use (\ref{4.5.8}), $\rf \approx \tir$ and $\tir \approx t$ to obtain
\begin{equation*} 
\left|\frac{\tir-\rf}{\rf}\right|\les \ve \frac{\bb^{-1}t+\tf}{\rf}\les \ve
\end{equation*}

Finally we prove (\ref{4.5.8}). We may use (\ref{4.5.6}) to obtain
\begin{align}
\bN \left(\bb^{-1} t-\tf\right) =\tir \left(\frac{\bb^{-1}t}{\rho}\ckk k_{\Nb\Nb}+\pib_{\bN \bN}\right)\label{4.5.2}
\end{align}
If $\rho\ge \frac{\bb^{-1} t}{3}$, then by using (\ref{4.5.2}), (\ref{lhk1}), (\ref{lpt1}) and (\ref{4.14.3.16}) we have
\begin{align*}
\left|\bb^{-1}t -\tf\right|(\rf, \omega)
&=\left|\int_0^\rf r_\flat'\left(\frac{\bb^{-1}t}{\rho}\ckk k_{\Nb\Nb}+\pib_{\Nb \Nb}\right) d r_\flat'\right|\nn\\
&\les \int_0^\rf r_\flat' \left(|\ckk k_{\Nb\Nb}|+|\pib_{\bN\bN}|\right)d r_\flat'\les \ve \rf^2. 
\end{align*}
If $\rho\le \frac{\bb^{-1} t}{3}$, by using (\ref{b1}) in Proposition \ref{prl_1},
the last estimate in (\ref{lpt1}), and (\ref{init}) we can obtain
\begin{align*} 
\left|\frac{\bb^{-1}t-\tf}{t}\right|&= |\bb^{-1}-n+n-n(\Ga(t))-(\bb^{-1}(\Ga(t))-n(\Ga(t)))|\les \ve.
\end{align*}
This together with $\rf\approx \tir \approx t$ shows that $|\bb^{-1} t - \tf|\les \ve \rf$.
\end{proof}

Now we employ Lemma \ref{angome} to obtain the control of $\sn r$ and $\varpi$ on $\{t=T\}$.

\begin{proposition}\label{angur}
Assume that on $\I^{+}(\O) \cap \{t=T\}$ there hold
\begin{align}\label{4.1.2}
\left|\Ga_{ij}^k\right|+\left|g^{ij}-\delta^{ij}\right|\les \ve,
\end{align}
and the bootstrap assumption
\begin{equation}\label{4.14.24.16}
\left|\frac{\rf}{r}-1\right|<\f12,
\end{equation}
then on $\I^+(\O)\cap \{t=T\}$ we have
\begin{equation}\label{3.31.17}
|\sn r|^2_g+|1-\varpi^2|\les \ve
\end{equation}
and
\begin{equation}\label{3.31.18}
|\Sigma N|_e+|\Sigma N|_g\les \ve^\f12.
\end{equation}
Moreover, on $Z^s\cap \{t=T\}$ there holds
\begin{equation}\label{3.28.3.16}
0\le n-\varpi\les \ve.
\end{equation}
\end{proposition}

The above result will be proved simultaneously with the result which  improves (\ref{4.14.24.16}).
The latter is presented in Proposition \ref{4.9.11}.

We first give some basics for $\sn r$ and $\omega$ on $\{t = T\}$. Let $h^{ij} = g^{ij} - \delta^{ij}$.
In view of
\begin{equation*} 
g^{ij}\p_i r \p_j r=\delta^{ij} \p_i r \p_j r+h^{ij} \p_i r \p_j r
\end{equation*}
we have
\begin{equation}\label{3.31.15}
\varpi^2 +|\sn r|_g^2=1+h^{ij}\p_i r \p_j r.
\end{equation}
Hence with $|g^{ij}-\delta^{ij}|\le \ve$ sufficiently small there holds
\begin{equation}\label{3.31.16}
\varpi^2+|\sn r|^2_g\le 1.5.
\end{equation}

\begin{proposition}\label{ini}
At the center $\bo=(T, 0)$ there holds
\begin{equation}\label{inii3.31}
\limsup_{q\rightarrow(T,0)} \left(|\sn r|^2_g +|\varpi-g(\p_r, \p_r)^{-\f12}|\right)(q)\les \ve
\end{equation}
where the metric around $\bo\in \Sigma_T$ is written by Euclidian  polar coordinates,
\begin{equation*} 
g_{ij}dx^i dx^j=g(\p_r, \p_r) dr^2+\cir{\ga}_{AB}d\omega^A d\omega^B
\end{equation*}
with $\cir{\ga}$ being the induced metric on $S_{T,r}$.
\end{proposition}

We postpone the proof of Proposition \ref{ini} to the end of this subsection.

\begin{proof}[Proof of Proposition \ref{angur}]
Due to  Proposition \ref{ini}, we can  make an auxiliary bootstrap assumption
\begin{equation}\label{4.1.3}
|\sn r|_g(q)<\Delta_0^\f12,\quad  \forall\, q\in \{t=T\}\cap \I^+(\O)
\end{equation}
where $0<\dn<\frac{1}{4}$ is a small constant to be chosen. We will improve it to be
\begin{equation}\label{4.1.4}
|\sn r|_g \le C \left(\ve+\Delta_0^\frac{3}{2}\right)
\end{equation}
for some universal constant $C>0$. Then as long as $C\ve^\f12\le\Delta_0^{\f12}\le 1/(2\sqrt{C})$
and $0<\ep <1/4$, we can obtain
$$
|\sn r|_g \le \frac{1}{4}\Delta_0^\f12+ \ep^\f12 \Delta_0^\f12 \le \frac{3}{4} \Delta_0^\f12.
$$

We now prove (\ref{4.1.4}).
By (\ref{4.1.3}), (\ref{3.31.15}) and (\ref{3.31.16}), we have
\begin{equation}\label{4.1.5}
|1-\varpi^2|\les \dn+\ve, \quad \varpi^2\les 1,\quad |1-\varpi|\les \dn+\ve,
\end{equation}
where we employed $|g^{ij}-\delta^{ij}|\les \ve$ in (\ref{4.1.2}). Similarly, by using (\ref{3.31.4}) we have
\begin{equation}\label{4.1.6}
|\Sigma N|_e\les (\dn+\ve)^\f12.
\end{equation}
With $G$ denoting the right hand side of (\ref{4.4.9}), we now employ (\ref{3.25.2.16}), (\ref{4.4.9}),
(\ref{inii3.31}), and a similar argument as Lemma \ref{tsp1} to obtain
\begin{align*}
|\sn r|&\les \frac{1}{r}\int_0^{ \rf}\left( r |G|+\sn_A r\varpi (1-\varpi)\right) \cir{a} d r_\flat'\\
&\les \frac{1}{r}\int_0^\rf r' \left(|\Ga|+{r'}^{-1}(|h|+|1-\varpi|(|\sn_A r|+|\Sigma N|_e))\right) d r_\flat'
\end{align*}
where we employed (\ref{4.14.3.16}) and (\ref{4.14.25.16}).
By using (\ref{4.14.24.16}), (\ref{4.1.2}), (\ref{4.1.5}), (\ref{4.1.3}) and (\ref{4.1.6}), we have
\begin{equation*} 
|\sn r|\les \ve+(\dn+\ve)\left((\dn+\ve)^\f12+\Delta_0^\f12\right)\les \Delta_0^\frac{3}{2}+\ve,
\end{equation*}
which gives (\ref{4.1.4}) as desired. Thus (\ref{3.31.17}) is proved.
(\ref{3.31.18}) follows immediately as a consequence of (\ref{3.31.4}) and (\ref{3.31.17}).  (\ref{3.28.3.16}) follows as consequence of (\ref{3.31.4}) and (\ref{3.31.17}).
\end{proof}

Next, we will prove (\ref{4.14.24.16}), which is  a  comparison estimate between the Euclidean radius
function $r$ on the initial slice and the intrinsic radius function $\rf$.  To begin with, we derive
the following  transport equation.

\begin{lemma}\label{4.5.5}
On $\I^+(\O)\cap \{t=T\}$ there holds
\begin{align}
\varpi \p_r(\rf-r)&=\frac{\tir}{\rf}-\varpi-\frac{\tir}{\rf} \left(|\Sigma N|^2_e+\Sigma N^i \bN^j h_{ij}\right).\label{3.28.1.16}
\end{align}
\end{lemma}

\begin{proof}
We first employ (\ref{4.5.1}) and Lemma \ref{frames} to obtain
\begin{equation*} 
\bN(\rf)=-\frac{\rho}{\rf}\bN(\rho)=\frac{\tir}{\rf}.
\end{equation*}
This together with (\ref{neuc}) and (\ref{4.5.1}) gives
\begin{align*}
\varpi \p_r \rf&=\bN\rf-\Sigma N^\mu \p_\mu \rf=\frac{\tir}{\rf}+\Sigma N^\mu\frac{\rho}{\rf}\p_\mu \rho\\
&=\frac{\tir}{\rf}-\frac{\rho}{\rf}\Sigma N^i \fB_i=\frac{\tir}{\rf}-\frac{\rho}{\rf}\frac{\tir}{\rho} \Sigma N^i \bN^j g_{ij}
\end{align*}
where we used (\ref{fb1}) and (\ref{fbtn}) to obtain the last identity. By using (\ref{neuc}) again,
\begin{align*}
\varpi \p_r(\rf-r)&=\frac{\tir}{\rf}-\varpi-\frac{\tir}{\rf}\Sigma N^i\bN^j(h_{ij}+\delta_{ij})\\
&=\frac{\tir}{\rf}-\varpi-\frac{\tir}{\rf} \left(\Sigma N^i (\varpi\ne^j+\Sigma N^j)\delta_{ij}+\Sigma N^i\bN^jh_{ij} \right)
\end{align*}
which gives (\ref{3.28.1.16}) due to $\Sigma N^i \ne^j\delta_{ij}=0$.
\end{proof}

\begin{proposition}\label{4.9.11}
On the region $\Sigma_T\cap \I^+(\O)$ there hold
\begin{align}
\left|1-\frac{r}{\tir}\right|+\left|\frac{\rf}{r}-1\right|\les \ve, \qquad
\left|\frac{\tir}{r}-n^{-1}\right|\les \ve. \label{5.3.1.16}
 \end{align}
\end{proposition}

\begin{proof}
By using  (\ref{3.28.1.16}),  (\ref{3.31.18}), (\ref{3.31.17}) and (\ref{4.5.17}), we have
\begin{equation*} 
\left|\frac{\rf}{r}-1\right|
\les r^{-1}\int_0^r \left(\left|\frac{\tir}{\rf}-1\right|+|\varpi-1|+|\Sigma N|_e^2+|\Sigma N|_e \c| h|\right) dr\les \ve
\end{equation*}
Due to $1-\frac{r}{\tir}=\frac{\tir-\rf}{\tir}+\frac{\rf-r}{r}+(\frac{r}{\tir}-1)\frac{\rf-r}{r}$,
the above estimate and (\ref{4.5.17}) then shows that $|1-\frac{r}{\tir}|\les \ve$.
Hence the first estimate in (\ref{5.3.1.16}) is proved. This together with the last estimate in (\ref{lpt1})  implies
(\ref{5.3.1.16}) immediately.
 \end{proof}

\begin{corollary}\label{4.10.cmp}
On $\I^+(\O)\cap \Sigma_T$ there hold
\begin{equation}\label{iniop}
\begin{split}
|u-(T-r)|\les \ve, \\
\quad |u-\hat u|\les \ve \mbox{ in } Z^s.
\end{split}
\end{equation}
\end{corollary}

\begin{proof}
Using the definition of $u$ we can write $u -(t-r)=r-\tir+(\bb^{-1} t-t)$.
Since Proposition \ref{prl_1} (1) implies $\tir\le \bb^{-1}t\les t$, we have from Proposition \ref{4.9.11} that
$|r-\tir | \les \ve \tir\les \ve t$. By using Proposition \ref{prl_1} with
$t = T$ and  $|n(\bo)-1|\le \ve$ we also have
\begin{equation*} 
|\bb^{-1}t-t|\les |\bb^{-1}t-\tf|+|\bb^{-1}(0,t)-n((0,t))|t+|n(0,t)-1|t\les \ve.
\end{equation*}
Therefore
$$
|u-(T-r)| \le |r -\tir| + |\bb^{-1} t -t| \les \ve
$$
which shows the first inequality in (\ref{iniop}). The second inequality in (\ref{iniop}) follows as a direct consequence.
\end{proof}

Finally we prove Proposition \ref{ini}. For this purpose, we will resort to geodesic foliation in
a neighborhood of $\bo=(T,0)$. Let us first introduce the geometric set-up.

On $(\Sigma_T, g)$, we denote  the geodesic distance function by $s$, relative to which the geodesic from $\bo$ has unit velocity. Hence,
\begin{equation}\label{4.14.2.16}
g^{ij}\nab_i s \nab_j s=1.
\end{equation}
Let $i_0>0$ be the radius of injectivity of $\bo$ on $\Sigma_T$ and let $B(\bo, \ep)$ be the open geodesic ball
with radius $\ep$, where $0<\ep<i_0$. Then $B(\bo, \ep)=\cup_{0\le s<\ep}S_s$, where $S_s$ denotes the level set of $s$.
The metric $g$ on $B(\bo, \ep)$ can be written as
\begin{equation}\label{4.14.1.16}
ds^2+\ga'_{AB} d\omega^A d \omega^B
\end{equation}
where $\ga'$ is the induced metric on $S_s$ and $(\omega^1, \omega^2) $ are local coordinates on $\mathbb{S}^2$.
Clearly, $\nab s$ is the outward  unit normal of the foliation of $S_s$, which is denoted by $N'$.
Due to (\ref{4.14.2.16}) we have
\begin{equation}\label{4.14.5.16}
\nab_{N'} {N'}=0
\end{equation}

For any $q\in B(\bo, \ep)$, there exists a unique distance minimizing geodesic $\ti{\Ga}(s)$
connecting $q$ to $\bo$. Noting that with
\begin{equation*}
N_0:=\frac{d}{ds}\ti{\Ga}(0) \quad \mbox{ with  }  |N_0|_{g(\bo)}=1
\end{equation*}
we can write $q= \exp_\bo(s N_0)$. A point $q\in B(\bo, \ep)$ can  be regarded as a point on $S_{T,\rho}$
with unit normal $N$ as well as a point on $S_s$ with the unit normal $N'$, verifying  $N'|_{s=0}=N_0$.

At any $q\in B(\bo, \ep)$,  we introduce  the following decomposition
\begin{equation}\label{4.14.4.16}
\bN=\cos \varphi N'+Y
\end{equation}
where $\varphi\in [0,\frac{\pi}{2}]$ and $Y$ is a vectorfield tangent to $S_s$ at $q$.
By direct checking, $|Y|_g=\sin \varphi$. Here $\cos \varphi$ and $Y$ at the center $\bo$ are
understood as the limits when the point $q$ approaches $\bo$ along the geodesic $\exp_{\bo}(s N_0)$
with $s\rightarrow 0$.


\begin{lemma}\label{4.14.7.16}
Let $N_0 \in \T_{\bo} \Sigma_T$ be any unit vector. Then there hold
\begin{equation}\label{4.14.6.16}
\limsup_{s\rightarrow 0} \left(1-\cos \varphi(\exp_\bo (s N_0)) \right)
=\limsup_{s\rightarrow 0} \left(1-\frac{s\rho}{\tir}(\frac{1}{\rho(\bo)}+\ckk k_{N_0 N_0}+\pib_{N_0 N_0})\right),
\end{equation}
\begin{align}
&\limsup_{s\rightarrow 0}\left(|Y|^2+(1-\cos\varphi)\right)(\exp_\bo (s N_0))\les \ve, \label{4.14.20.16}\\
&\left|\frac{s}{\tir}-1\right|\les \ve \quad \mbox{ on } B(\bo, \ep). \label{4.14.14.16}
\end{align}
where $s$ is the normalized geodesic distance  to $\bo$ verifying (\ref{4.14.2.16}).
\end{lemma}

\begin{proof}
Let us first consider (\ref{4.14.6.16}).
For  $q=\exp_{\bo} (sN_0)$ consider
\begin{equation*} 
f(q):=\frac{\tir}{\rho}\cos \varphi=-\nab_i s\nab^i \rho.
\end{equation*}
We proceed by locally expanding round $\bo$  the above function as follows,
\begin{equation}\label{4.14.9.16}
f(q)=f(\bo)+(\nab_{N'} f)(\bo) s+O(s^2).
\end{equation}
Note that $f(\bo)=0$ due to $\tir=0$  at $\bo$. Now we calculate the term $(\nab_{N'} f)(0)$.
By using (\ref{4.14.5.16}) we have
\begin{align*} 
\nab_{N'} f=-\nab_{N'}\nab_i s\c \nab^i \rho-\nab_i s\nab_{N'}\nab^i \rho=- N'_i  \nab_{N'} \nab^i \rho.
\end{align*}
Note that
\begin{align*}
\bd_\mu \nab^i \rho &=\bd_\mu (\Pi_{\a'}^\a \bd^{\a'}\rho)e_\a^i
=e_\a^i \left(\bd_\mu (\bT^\a \bT_{\a'}+\delta^\a_{\a'})\bd^{\a'} \rho + \Pi_{\a'}^\a \bd_\mu \bd^{\a'}\rho\right)\nn\\
&=-e_\a^i \bd_\mu \bT^\a \l \bT, \fB\r+e_\a^i \bT^\a \bd_\mu \bT_{\a'} \bd^{\a'}\rho +\bd_\mu \bd^{\a'}\rho \Pi_{\a'}^\a e_\a^i \nn\\
&=-e_\a^i \bd_\mu \bT^\a \l \bT, \fB\r+\bd_\mu \bd^{\a'}\rho\Pi_{\a'}^\a e_\a^i, 
\end{align*}
where for the last equality we used $e_\a^i \bT^\a=0$. Therefore, from the above two equations we obtain
\begin{align*} 
\nab_{N'} f&=N'_i {N'}^\mu  (e_\a^i \bd_\mu \bT^\a \l \bT, \fB\r -\bd_\mu \bd^{\a'}\rho\Pi_{\a'}^\a e_\a^i) \\
&=\frac{\bb^{-1}t}{\rho}\pib_{N'N'}+\l \bd_{N'}\fB, N'\r.
\end{align*}
Hence
\begin{equation*} 
\lim_{s\rightarrow 0}\nab_{N'} f(\exp_{\bo}(s N_0))=\pib_{N_0N_0}+\l \bd_{N_0}\fB, N_0\r.
\end{equation*}
Combining this with (\ref{4.14.9.16}) gives (\ref{4.14.6.16}).

Now consider (\ref{4.14.14.16}).  We first prove  there holds for a fixed constant $C>0$ that
\begin{equation}\label{4.14.15.16}
(1-C\ve)\rf\le s\le (1+C\ve) \rf
\end{equation}
 with the help of the argument in \cite[(14.0.7.a)]{CK}.  For $q=(\rf, \omega)\in B(\bo, \ep)$, in view of (\ref{4.14.3.16})
 and (\ref{4.14.4.16}) we have
\begin{equation}\label{4.14.16.16}
s(q)=\int_0^\rf \p_\rf s d\rf'=\int_0^{\rf} \nab_\bN s \cir{a} d\rf' =\int_0^\rf \cos \varphi \cir{a} d\rf' \le \int_0^\rf \cir{a}|_{\Ga} d\rf'
\end{equation}
where $\Ga$ is the integral path of $\p_\rf$ with the angular variable $\omega\in {\mathbb S}^2$ fixed.
On the other hand, let $\Ga'$ be the distance minimizing geodesic connecting $\bo$ to $q$. Then
\begin{equation}\label{4.14.17.16}
s(q)=\int_{\Ga'} \nab_{N'}s ds' =\int_0^\rf (\cir{a}^2+\ga_{AB} \frac{d\omega^A}{d \rf}\frac{d\omega^B}{d\rf})^\f12 |_{\Ga'} d\rf'\ge \int_0^\rf \cir{a}|_{\Ga'} d\rf'.
\end{equation}
By using (\ref{4.5.17}), (\ref{4.14.16.16}) and (\ref{4.14.17.16}) we thus obtain (\ref{4.14.15.16}).
(\ref{4.14.14.16}) can be obtained by using (\ref{4.14.15.16}), (\ref{4.5.17}) and $\frac{s}{\tir}-1=\frac{s}{\rf} \frac{\rf}{\tir}-1$.

Finally (\ref{4.14.20.16}) can be obtained using (\ref{4.14.6.16}), (\ref{lhk1}), (\ref{lpt1}), $|Y|_g=\sin \varphi$ and $T\approx 1$.
\end{proof}

\begin{proof}[Proof of Proposition \ref{ini}]
 We first introduce the decomposition of $N'$ with the Euclidian radial normal $\ne$, in the same way as (\ref{neuc}).
\begin{equation*} 
N'=\varpi' \ne+\Sigma N'
\end{equation*}
and
\begin{equation}\label{4.14.22.16}
\varpi'(\bo)=g(\p_r,\p_r)^{-\f12}.
\end{equation}
With the help  of the above decomposition,
\begin{equation}\label{4.14.19.16}
\varpi=\bN(r)=\cos\varphi N'(r)+Y(r)=\cos \varphi \varpi' +Y(r).
\end{equation}
Note that,  similar  to (\ref{3.31.15}) we have
\begin{equation*}
{\varpi'}^2+|\sn' r|_{g}^2 =1+h^{ij}\p_i r \p_j r
\end{equation*}
where $\sn'$ is the Levi-Civita connection of the induced metric $\ga'$ on $S_s$.  In view of (\ref{4.14.22.16}) and (\ref{4.1.2}),
this implies
\begin{equation*} 
\limsup_{s\rightarrow 0}|\sn' r|_{g}^2(\exp_{\bo} (s N_0))\les \ve.
\end{equation*}
In view of $Y(r)=\sn' r(Y)$, this together with (\ref{4.14.20.16}) implies that
 \begin{equation*} 
\limsup_{s\rightarrow 0}|Y(r)|(\exp_\bo (s N_0))\les \ve.
 \end{equation*}
By combining this estimates with (\ref{4.14.20.16}), (\ref{4.14.22.16}) with (\ref{4.14.19.16}),
we can obtain the second part in (\ref{inii3.31}), due to
\begin{equation*}
\varpi-g(\p_r, \p_r)^{-\f12}=\cos \varphi (\varpi'-g(\p_r, \p_r)^{-\f12})+(\cos \varphi-1) g(\p_r, \p_r)^{-\f12}+Y(r).
\end{equation*}
With  the help of (\ref{3.31.15}), the other part follows as an immediate consequence.
\end{proof}

\subsection{The intrinsic geometry in $Z^s$}

Next we give the main result of this section, which lies in the core of analysis in this paper in $Z^s$.

\begin{theorem}[Main estimates]\label{4.10.6}
In $Z^s\subset\I^+(\O)$, there hold $t\approx\tir\approx r$ and
\begin{align}
&0\le n-\varpi\les \ve \bt^{-4}, \label{4.9.8}\\
&|\sn r|+|\Sigma N|_\bg\les\ve^\f12 \bt^{-2}, \label{4.9.16}\\
& \bt^2 |k_{\Nb A}|+\bt\rho \left|(\hat k, \tr k-\frac{3}{\rho})\right|\les \ve, \label{2.20.1.16}\\
&|u-\hat u|\les\ve, \label{3.28.4.16}\\
&\bt\left|\frac{\tir}{r}-n^{-1}\right|\les \ve, \label{1.4.1.16}
\end{align}
where $\bt=t+1$.
\end{theorem}

\begin{proof}
Let us make in $Z^s$ the bootstrap assumptions
\begin{align}
&|u-\hat u|\le 10\dn  t^\delta, \qquad \forall\, t> T \mbox{ and }\hat u\le {\hat u}_1, \label{4.10.7}\\
&\left|\frac{\tir}{r}-n^{-1}\right|\le \dn t^{-\delta},\qquad 0\le n-\varpi\le\dn t^{-\delta}, \label{4.9.9}\\
&t^{1-\d}\rho \left|(\hk, \tr k-\frac{3}{\rho})\right|\le \dn, \label{12.28.19.15}
\end{align}
where $0<\d<\frac{1}{10}$ and $C_0 \ve \le \dn\le \frac{1}{4T}$ are small numbers to be chosen later, here
$C_0$ is the universal constant in (\ref{3.28.3.16}), (\ref{iniop}), (\ref{lhk1})  and (\ref{5.3.1.16}).
$n-\varpi\ge 0$ follows directly from the second identity in (\ref{3.31.4}).
 To complete the proof, we need to show
 \begin{align}
 &|u-\hat u|\le C\ve^\f12(\dn+\ve^\f12)+6\dn T^\d, \label{4.25.1.16}\\
&0\le n-\varpi\le C \ve t^{-4}, \label{4.9.15}\\
&t\left|\frac{\tir}{r}-n^{-1}\right|\le C \ve, \label{5.3.2.16}\\
&t\rho \left|(\hk, \tr k-\frac{3}{\rho})\right|\le  C(\ve+\dn^2) \label{5.3.3.16}
\end{align}
which are improvement over (\ref{4.10.7})- (\ref{12.28.19.15}) whence choosing $\dn=\min\{2C\ve, \frac{1}{4} T^{-1}\}$.
This choice can be achieved since we can choose $\ve<\frac{1}{8T C}$. Here the universal
constant $C>C_0$ and $\ve$ is sufficiently small such that $C\ve^\f12<\f12$. With the completion
of (\ref{4.25.1.16})-(\ref{5.3.3.16}), we can obtain (\ref{4.9.8})-(\ref{1.4.1.16}) except the first
estimate in (\ref{2.20.1.16}). In the sequel, we prove (\ref{4.25.1.16})-(\ref{5.3.2.16}).

For $q\in Z^s$, $\sf{q}$ intersects $\{t=T\}$ at $q'$, which is in $\{r\ge 2\}\cap \Sigma_T$.
We will employ transport equations along the segment of $\sf{q}=\{\exp_\O(\rho V)\}$ from $q'$
to $q$ with $V$ determined by $q$, and initial data given in Proposition \ref{angur} and Proposition \ref{4.27.1.26}.
We will frequently employ in $Z^s$ the relations
\begin{equation}\label{4.9.12}
r\approx \tir\approx t, \quad \bb^{-1}\approx 1, \quad t\approx \tau, \quad   \varpi\approx n\approx 1
\end{equation}
which follow from  Proposition \ref{prl_1}, (\ref{4.9.9}) and  (\ref{4.9.6}).

Let us  employ  (\ref{bvarpi})  with the initial data given in (\ref{3.28.3.16}) in Proposition \ref{angur}.
We can write (\ref{bvarpi}) as
\begin{align}\label{4.9.13}
\fB (n -\varpi) + \frac{4}{\rho} (n-\varpi)  = (H_1 + H_2) (n-\varpi),
\end{align}
where
\begin{align*}
H_1 & = \frac{4}{\rho}-2 \frac{\tir}{\rho} \frac{r-M}{(r+2M)^2} n^{-2} (n + \varpi),\\
H_2 & = \frac{2M}{n^2(r+2M)^2} \left(\frac{\tir}{\rho} \varpi + \frac{\bb^{-2} t^2}{\rho \tir}(n+\varpi)\right).
\end{align*}
On $\Sigma_T\cap Z^s$ we can infer from (\ref{3.28.3.16}) and $\tau\approx T$ that
\begin{equation*}
\tau^4(n-\varpi)\les \ve, \quad \mbox{if } t=T.
\end{equation*}
Thus, with $q'=\exp_\O(\rho_0 V)$, by assuming
\begin{equation}\label{4.9.14}
\int_{\rho_0}^\rho\left(|H_1|+|H_2|\right) d\rho'\les \dn+\ve,
\end{equation}
we can apply Lemma \ref{tsp1} to $H=H_1, H_2$ and $\upsilon=(\frac{\tau}{\rho})^4$  for the equation (\ref{4.9.13})
to obtain
\begin{equation*}
(n-\varpi) \les \tau^{-4}\ve\approx t^{-4}\ve.
\end{equation*}
Here, to obtain the last inequality, we employed (\ref{4.9.12}).

It remains to prove (\ref{4.9.14}) with the help of (\ref{4.9.9}),  (\ref{4.9.12}) and $0<M\le\ve$.
Using $n^2 = (r-2M)/(r+2M)$ we can write
\begin{equation*} 
H_1 = \frac{2\tir}{\rho}\frac{r-M}{r^2-4M^2}(n-\varpi)- \frac{4nr}{\rho} \left(\frac{\tir}{r}-n^{-1}\right) \frac{r-M}{r^2-4M^2}
+\frac{4}{\rho}\frac{(r-4M)M}{r^2-4M^2}.
\end{equation*}
By using $r\approx r\pm 2M$ and $n\approx 1$, we have
\begin{equation}\label{h1sb}
|H_1|\les \frac{1}{\rho}\left(\left|\frac{\tir}{r}-n^{-1}\right|+\frac{M }{r}\right)
+\frac{\tir}{\rho}\frac{(n-\varpi)}{r}.
\end{equation}
Similarly, we can bound $H_2$ by
\begin{equation*} 
|H_2|\les \frac{M\tir}{\rho(r+2M)^2}\left(1+\frac{t^2}{\tir^2}\right) \les  \frac{M\tir}{\rho(r+2M)^2}
\approx \frac{M\tir}{\rho r^2}\approx \frac{M}{\rho r}.
\end{equation*}
Symbolically, this term has already appeared  in $H_1$. Thus, it suffices to consider only the
types of terms in (\ref{h1sb}). By using (\ref{3.20.1}), (\ref{4.9.12}) and the first assumption
in (\ref{4.9.9}), we have
\begin{align}\label{4.10.2}
\int_{\rho_0}^\rho \frac{1}{\rho'} \left|\frac{\tir}{r}-n^{-1}\right| d\rho'
\les \dn \int_{T}^t \frac{1}{\bb^{-1}n^{-1} t'}\l t'\r^{-\d} dt'\les \dn.
\end{align}
Similarly,
\begin{align*}
\int_{\rho_0}^\rho \frac{1}{\rho'}\left(\frac{M}{r}+(n-\varpi)\right) d\rho'
& \les \int_T^ t \frac{1}{\bb^{-1}n^{-1}t'}\left(\frac{M}{r}+(n-\varpi)\right) d t'\\
&\les (\ve+\dn) \int_T^t {t'}^{-(1+\d)} d t'\les  \ve+\dn,
\end{align*}
where we employed the second assumption in (\ref{4.9.9}) as well. This ends the proof of
(\ref{4.9.14}) and thus  (\ref{4.9.15}) is proved.
(\ref{4.9.16}) follows as a direct consequence of (\ref{4.9.15}), (\ref{12.10.26})  and
the second identity in (\ref{3.31.4}).



Next we prove (\ref{4.25.1.16}). Due to (\ref{4.10.7}) we have
\begin{equation}\label{4.17.1.16}
u\le {\hat u}_1+10\dn t^\d.
\end{equation}
We consider $Z^s$ foliated by $\cup_{\hat u\le \hat u_1} \C^s_{\hat u}$. For any point $q\in Z^s$,
we regard $q$ as a point in $\C^s_{\hat u}$  with $\hat u$ uniquely determined by $q$. There is a
unique null geodesic $\Ga'(t)$ on $\C^s_{\hat u}$, such that $\frac{d}{dt} \Ga'=\hat L$, which
intersects $\{t=T\}$ at $q'$. Note that $\hat u$ is invariant on $\C^s_{\hat u}$, with the help of
Corollary \ref{4.10.cmp} we can obtain
\begin{align}
|(u-\hat u)(q)| & =\left|(u-\hat u)(q')+\int_T^t \hat L (u-\hat u)dt'\right|\le C_0\ve+ \left|\int_T^t \hat L(u) dt'\right|\nn\\
&\le \int_T^t u\left|\sn_A r\c k_{A\Nb}+n\l \bd_\bT \bT, \bN\r+\ckk k_{\Nb\Nb}\frac{(n-\varpi)\tir-u \varpi}{\rho}\right|dt' \nn\\
& \quad \, +  C_0\ve+\int_T^t |n-\varpi| dt. \label{4.10.8}
\end{align}
From (\ref{4.9.15}) it follows that
\begin{equation*}
\int_T^t |n-\varpi| dt'\les \ve \int_T^t t'^{-4}dt'\les \ve.
\end{equation*}
In order to treat the first term on the right hand side of (\ref{4.10.8}), we will rely on (\ref{12.28.19.15})
to treat the terms on $\ckk k_{\Nb\Nb}, k_{\Nb A}$. By (\ref{4.9.16}), $\rho^2=u\ub$ and $\rho\les t$, we have
\begin{align*}
&\int_T^t u|\sn_A r\c k_{A\Nb}| dt\les \ve^\f12\int_T^t|\rho k_{A\Nb}|\rho  t'^{-3}d t'
\les \ve^\frac{1}{2} \dn\int_T^t t'^{-4+\delta}\rho dt'\les \ve^\frac{1}{2}\dn.
\end{align*}
Noting that $\bd n\approx \frac{\ve}{r^2}$ which can be obtained from (\ref{3.31.13}), we may use
(\ref{4.17.1.16}) and (\ref{4.9.12}) to derive that
\begin{align*}
&\int_T^t u n |\l \bd_\bT \bT, \bN\r|d t'\les\ve\int_T^t r^{-2}t'^\delta dt'\les \ve.
\end{align*}
Due to $\rho^2=u \ub$ and (\ref{4.9.12}), we have $\frac{u t}{\rho}\approx \rho$. Thus by using (\ref{4.9.15})
and (\ref{12.28.19.15}) we can obtain
\begin{align*}
\int_T^t \left| \ckk k_{\Nb\Nb} \frac{u(\tir+u)}{\rho}(n-\varpi)\right|d t'
&\le \int_T^t |\ckk k_{\Nb\Nb}\rho'| (n-\varpi) dt'\les \ve\dn\int_T^t t'^{-4} dt'\les \ve\dn.
\end{align*}
Since $\frac{u^2}{\rho}=\frac{\rho u}{\ub}$, we may use (\ref{12.28.19.15}), (\ref{4.17.1.16}) and (\ref{3.20.lap})
to derive that
\begin{equation*}
\int_T^t \left|\frac{u^2 n}{\rho} \ckk k_{\Nb \Nb} \right|d t'
\le 2\dn \int_T^t \frac{t'u}{\ub}{t'}^{\d-2} dt'< 4\dn T^{\d-1}(\hat u_1+10\dn T^{\d}),
\end{equation*}
where we used the property $\ub>\frac{5t}{8}$ which can be seen as follows. Indeed, in view of the first assumption
in  (\ref{4.9.9}), (\ref{3.20.lap}) and (\ref{bb2}),  we can derive, with $M$ sufficiently small so that $n(2)>\frac{4}{5}$, that
\begin{align*}
\ub&=\bb^{-1}t+\tir=(\bb^{-1}-n)t+nt+n^{-1}r+\left(\frac{\tir}{r}-n^{-1}\right)r\\
&\ge t\left(n+n^{-1}\frac{r}{t}\right)-c\ve \ln t-\frac{n(2)}{10}  t^{1-\d}\\
&\ge t\left(\frac{9}{10}n+\frac{1}{10}n^{-1}\right)-c\ve \ln t>\frac{5t}{8}
\end{align*}
where we used $\dn<\frac{1}{4T}\le \frac{n(2)}{10}$, the fact that $n(r)$ is increasing and that $c\ve$ can be
sufficiently small, we also employed (\ref{4.9.6}) to obtain
$$
\frac{r}{t}\ge  \frac{t-\hat u_1}{2t}>\f12-\frac{t-2}{2t}>\frac{1}{10}.
$$
Hence (\ref{4.25.1.16}) is proved because $\hat u_1=T-\ga(2)\le T-2$ and $10\dn T^\d<\f12 T$ due to $T>5$ and $4\dn T<1$.

Next, we prove (\ref{5.3.2.16}).
Since $\pib_{\bN\bN}$ in (\ref{cmr_1}) vanishes since we consider $Z^s$ only, we can rewrite (\ref{cmr_1}) as
\begin{equation}\label{1.4.6.16}
\fB \left(\frac{\tir}{r}-n^{-1}\right)+\frac{1}{\rho}\left(\frac{\tir}{r}-n^{-1}\right)
=-H \left(\frac{\tir}{r}-n^{-1}\right)+G,
 \end{equation}
where
\begin{align*}
G = \frac{\tir^2}{r^2 \rho} (n-\varpi) - \frac{\rho}{r} \bN (\log n), \qquad
H = \frac{n}{\rho} \left(\frac{\tir}{r}-n^{-1}\right) + \frac{\tir}{\rho} \bN(\log n).
\end{align*}
By using (\ref{3.20.1}), (\ref{3.31.13}), (\ref{4.9.12}) and (\ref{4.9.9}), we can derive that
\begin{align}\label{1.4.9.16}
\int^\rho_{\rho_0}|H |d\rho'&\les \int_{T}^t \frac{\rho'}{\bb^{-1}t'}\frac{\tir}{\rho'}\bN \log n d t'
+\int_{\rho_0}^\rho \frac{n}{\rho'} \left(\frac{\tir}{r}-n^{-1}\right) d\rho' \nn \\
&\les\int_{T}^t\ve\bt^{-2} dt'+\int_{\rho_0}^\rho\dn \l \rho'\r^{-1} \bt^{-\d} d\rho'\les \ve+\dn.
\end{align}
Next we consider the term $G$. Note that in $Z^s$ where $\hat u\le \hat u_1\le  T$, we can
use (\ref{4.25.1.16}) to obtain
\begin{equation}\label{1.4.7.16}
0\le u\les T+\dn+\ve\les 1.
\end{equation}
Thus we may use $\ub \approx t$, $r \approx t\approx \tau$  and (\ref{1.4.7.16}) to infer that
\begin{equation}\label{1.4.11.16}
\tau^{-1}\int_{\rho_0}^\rho \tau' \frac{\rho'}{r} |\bN\log n |d\rho'\les \tau^{-1}\int_T^t \frac{{\rho'}^2 \tau'}{r t'}|\bN \log n| dt'\les\ve\frac{\max{u}}{\tau} \les \ve/\tau
\end{equation}
By using (\ref{4.9.12}) and (\ref{4.9.8}) we also have
\begin{equation}\label{1.4.12.16}
\tau^{-1}\int_{\rho_0}^\rho \tau'\frac{\tir^2}{r^2\rho}(n-\varpi) d\rho'
\les \tau^{-1}\int_{T}^t \frac{\tau'}{\bb^{-1}t'} (n-\varpi) dt'\les \ve/\tau.
\end{equation}
Note that due to (\ref{5.3.1.16}) and (\ref{4.9.12}), we have $\tau  |\frac{\tir}{r}-n^{-1}|\les \ve$
on $Z^s\cap \Sigma_T$.  In view of (\ref{1.4.9.16})-(\ref{1.4.12.16}), we may use Lemma \ref{tsp1} to obtain
\begin{equation*} 
\left|\frac{\tir}{r}-n^{-1}\right|\les \ve/\tau.
\end{equation*}
By using $\tau \approx t$, we can obtain (\ref{5.3.2.16}).

It remains to prove (\ref{5.3.3.16}). This will rely on (\ref{4.9.15}),  (\ref{4.9.16}) as the
consequence of (\ref{4.9.15}),  as well as (\ref{1.4.7.16}) in $Z^s$. We will divide the proof
into two steps: the first step is to control curvature components, which is presented in
Proposition \ref{1.2.1.16}; the second step is to use the obtained estimates on curvature  to
control the second fundamental form $k$.
\end{proof}

We will need the estimates on Weyl components in $Z^s$ relative to the intrinsic frame
$\{\fB, \Nb, e_A, A=1,2\}$. We will first prove Proposition  \ref{3.18.1.16}. The following result
can follow as a consequence.

\begin{proposition}\label{1.2.1.16}
For all  $t>T$   in $Z^s$
\begin{align}
\left|W(S,e_A,S, e_C)\right| &\les \ve t^{-2}, \label{2.14.1.16}\\
\rho\left|W(\fB,e_A, \fB, \Nb)\right| & \les \ve^\frac{3}{2} t^{-4},  \label{2.14.2.16}
\end{align}
\begin{equation}\label{12.31.7}
\varrho=\frac{1}{4}W(L, \Lb, L,\Lb)=n^{-4}\hat \varrho \left(1+O\left(\frac{\ve}{t^4}\right)\right).
\end{equation}
More precisely
\begin{equation}\label{2.13.1.16}
\varrho=n^{-4} \hat\varrho \left(1+\frac{3}{2}\left(n^{-2} \varpi^2-1\right)\right).
\end{equation}
\end{proposition}

The above result is crucial to prove (\ref{2.20.1.16}) in Theorem \ref{4.10.6}, which is to
control the geometry of the hyperboloidal foliation in $Z^s$, where the density of the level
set is approaching $\infty$. We remark that (\ref{2.20.1.16}), together with the control of the
second fundamental form $k$ on wave zone will imply that the radius of conjugacy is $+\infty$.
The pointwise bound on curvature components, combined with the result of radius of conjugacy,
implies that the radius of injectivity is $+\infty$. The estimate (\ref{2.20.1.16}) is also crucial
to justify the limit of Hawking mass exists, which is the main result in Section \ref{mass}.

Recalling $\hat L$ from (\ref{4.3.2}), we define a pair of null frame
\begin{equation}\label{4.25.2.16}
\hat L=\p_t+n^2 \p_r, \quad \hat \Lb=\p_t-n^2\p_r.
\end{equation}
By using (\ref{scm}), we have $\l \hat L, \hat\Lb\r=-2n^2$. This implies $\{n^{-1}\hat L, n^{-1}\hat \Lb, {\hat e}_A, A=1,2\}$
forms a canonical null tetrad in $Z^s$, where $\{\hat e_A, A=1,2\}$ is an orthonormal frame on $S_{t, \hat u}$.

Now using (\ref{4.25.2.16}), we have
\begin{equation}\label{12.10.23}
\bT=\frac{n^{-1}}{2}(\hat L+\hat \Lb), \quad n\p_r=\frac{n^{-1}}{2}(\hat L-\hat\Lb).
\end{equation}
Let $\mu=1-n^{-1}\varpi$. It then follows from (\ref{neuc}) that
\begin{align}\label{3.18.3.16}
\begin{split}
L  & =\frac{n^{-1}}{2}\left[(2-\mu)\hat L+\mu\hat \Lb\right]+\Sigma N, \\
\Lb& = \frac{n^{-1}}{2} \left[\mu \hat L+(2-\mu)\hat \Lb\right]-\Sigma N.
\end{split}
\end{align}
In $Z^s$, by using (\ref{4.16.1}) and (\ref{12.10.23}), we have
\begin{equation}\label{2.14.3.16}
e_A =\Sigma e_A + \f12 n^{-2}\sn_A r (\hat L-\hat \Lb).
\end{equation}

For future reference, let us also set
\begin{align}
\cir{\ub}=n^{-1}(\bb^{-1}t+ n^{-1} \varpi \tir), \quad \cir{u}=n^{-1}(\bb^{-1}t- n^{-1} \varpi \tir).\label{12.10.19}
\end{align}
By using  (\ref{dcp_2}) and (\ref{3.18.3.16}),
\begin{equation}\label{12.31.8}
S=\rho \fB=\f12 \left(\cir{\ub} \hat L+\cir{u}\hat \Lb\right)+\tir \Sigma N.
\end{equation}

To prove Proposition \ref{3.18.1.16},  we will employ the following properties of
curvature in $Z^s$ under the canonical null tetrad $\{n^{-1}\hat L, n^{-1} \hat \Lb, \hat e_A, A=1,2\}$.

\begin{lemma}\label{4.28.17.16}
\begin{enumerate}
\item Under the null decomposition in terms of $n^{-1}\hat L=e_4, n^{-1}\hat\Lb=e_3$, in $Z^s$
the only nonvarnishing Weyl components in the list of Definition \ref{3.18.19} is
$\hat\varrho:=\frac{1}{4}W(\hat L, \hat \Lb, \hat L, \hat\Lb)$ which is given by
\begin{equation}\label{12.10.10}
n^{-4}\hat \varrho=-\frac{4 M}{(r+2M)^3}.
\end{equation}

\item As direct consequences, by using \cite[Page 149, (7.3.3c)]{CK} we have
\begin{equation}\label{4.28.16.16}
\begin{split}
& W( \hat e_A, n^{-1} \hat\Lb, \hat e_B, n^{-1} \hat L)=-n^{-4}\hat \varrho\,\delta_{A B},\\
& W(\hat e_A, \hat e_B, \hat e_C, \hat e_D)=-n^{-4}\ep_{A B}\ep_{C D}\hat \varrho,\\
& W(\hat e_A, \hat e_B, \hat e_C, n^{-1}\hat \Lb)=0; \,  W(\hat e_A, \hat e_B, \hat e_C, n^{-1} \hat L)=0.
\end{split}
\end{equation}
\end{enumerate}
\end{lemma}
We will postpone the proof of  (1) in the above lemma to Lemma \ref{12.10.21} in the next section. Since it is a fact of the Schwartzchild metric itself, the proof is independent of the intrinsic hyperboloidal frame, which also means it is independent of any result in this section. In the sequel, we will constantly use Lemma \ref{4.28.17.16} without mentioning.

\begin{proof}[Proof of Proposition \ref{3.18.1.16} and Proposition \ref{1.2.1.16} ]
We first show (\ref{12.31.7}). We decompose
$
\varrho = \frac{1}{4} W(L,\Lb, L,\Lb)
$
by using (\ref{3.18.3.16}), the properties of Weyl tensor $W$ and Lemma \ref{4.28.17.16},
\begin{align}\label{1.1.7.16}
4\varrho&= W(L, \Lb, L,\Lb) & \nn \\
&=W\left(\frac{n^{-1}}{2}\left[(2-\mu)\hat L+\mu\hat \Lb\right]+ \Sigma N, \frac{n^{-1}}{2}\left[\mu\hat L+(2-\mu)\hat \Lb\right]-\Sigma N, \right.\nn\\
& \qquad \quad  \left. \frac{n^{-1}}{2}\left[(2-\mu)\hat L+\mu\hat \Lb\right] + \Sigma N,
  \frac{n^{-1}}{2}\left[ \mu\hat L+(2-\mu) \Lb\right]-\Sigma N \right)\nn\\
& = \emph{I} + 2\emph{II}+\emph{III}+\emph{IV},
\end{align}
where
\begin{align*}
\emph{I} &=\frac{n^{-4}}{16} W\left((2-\mu)\hat L+\mu\hat \Lb, \mu \hat L+(2-\mu)\hat \Lb,
  (2-\mu)\hat L+\mu\hat \Lb, \mu\hat L+(2-\mu) \hat \Lb\right), \displaybreak[0]\\
\emph{II} &= W\left(\frac{n^{-1}}{2}\left[\mu\hat L+(2-\mu)\hat \Lb\right],\Sigma N, \frac{n^{-1}}{2}\left[(2-\mu)\hat L+\mu\hat \Lb\right],
  \Sigma N\right), \displaybreak[0]\\
\emph{III} &= W\left( \Sigma N, \frac{n^{-1}}{2}\left[\mu \hat L+(2-\mu)\hat \Lb\right], \Sigma N,
  \frac{n^{-1}}{2}\left[(\mu \hat L+(2-\mu)\hat \Lb\right]\right), \displaybreak[0]\\
\emph{IV} &= W\left(\Sigma N, \frac{n^{-1}}{2}\left[(2-\mu)\hat L+\mu\hat \Lb\right], \Sigma N,
\frac{n^{-1}}{2}\left[(2-\mu)\hat L+\mu\hat \Lb\right]\right).
\end{align*}
By direct calculation, it is easy to see that
\begin{align}\label{4.23.3.16}
\emph{I}=\frac{n^{-4}}{16}\left((2-\mu)^4+\mu^4-2(2-\mu)^2\mu^2\right) W(\hat L, \hat \Lb, \hat L, \hat \Lb)
=4n^{-6}\varpi^2  \hat\varrho.
\end{align}
Noting that
\begin{align*}
\emph{II}+\emph{IV} &=W\left(\frac{n^{-1}}{2}\left[(2-\mu)\hat L+\mu\hat\Lb\right], \Sigma N, n^{-1}(\hat L+\hat \Lb), \Sigma N\right)\\
&=n^{-2}W(\hat L, \Sigma N, \hat \Lb,\Sigma N)=-n^{-4}\hat\varrho|\Sigma N|_\bg^2,\\
\emph{II}+\emph{III} &=\emph{II}+\emph{IV},
\end{align*}
we obtain, in view of the above identities and  (\ref{4.23.3.16}) that
\begin{equation}\label{1.1.8.16}
\varrho= n^{-4} \hat \varrho \left(n^{-2}\varpi^2-\f12 \left(1-n^{-2}\varpi^2\right)\right)
\end{equation}
which together with (\ref{4.9.15}) implies (\ref{12.31.7}).

We now consider  $\ab_{AC} := W(\Lb, e_A, \Lb, e_C)$. We may use (\ref{3.18.3.16}) and (\ref{2.14.3.16}) to
replace $\Lb$, $e_A$ and $e_C$ and expand it. Due to the various vanishing terms implied by Lemma \ref{4.28.17.16} we have
\begin{align*}
\ab_{AC} = I_1 + I_2 + I_3 + I_4 + I_5 +I_6,
\end{align*}
where
\begin{align*}
I_1 & = W\left(\frac{n^{-1}}{2}\left[\mu\hat L+(2-\mu)\hat \Lb\right],
   \f12 n^{-2} \sn_A r  (\hat L-\hat\Lb), \right. \displaybreak[0]\\
& \qquad \qquad \qquad  \left. \frac{n^{-1}}{2}\left[\mu\hat L+(2-\mu)\hat \Lb\right],
   \f12 n^{-2}\sn_C r\c (\hat L-\hat\Lb)\right), \displaybreak[0]\\
I_2 &=  W\left(\frac{n^{-1}}{2}\left[\mu\hat L+(2-\mu)\hat \Lb\right], \Sigma e_A,
   -\Sigma N, \f12 n^{-2}\sn_C r (\hat L- \hat \Lb)\right), \displaybreak[0]\\
I_3 &= W\left(\frac{n^{-1}}{2}\left[\mu\hat L+(2-\mu)\hat \Lb\right], \Sigma e_A,
   \frac{n^{-1}}{2}\left[\mu\hat L+(2-\mu)\hat \Lb\right],\Sigma e_C\right), \displaybreak[0]\\
I_4 &= W\left(-\Sigma N, \f12 n^{-2}\sn_A r (\hat L-\hat \Lb),
   \frac{n^{-1}}{2}\left[\mu\hat L+(2-\mu)\hat \Lb\right], \Sigma e_C\right), \displaybreak[0]\\
I_5 &= W\left(-\Sigma N, \Sigma e_A, -\Sigma N, \Sigma e_C\right), \displaybreak[0]\\
I_6 &= W\left(\Sigma N, \frac{1}{2}n^{-2} \sn_A r(\hat L-\hat \Lb), \Sigma N, \f12 n^{-2} \sn_C r(\hat L-\hat \Lb)\right).
\end{align*}
By straightforward manipulation we have
\begin{align*}
I_1 & = \frac{n^{-6}}{16 } \left(\mu^2 + 2 \mu(2-\mu) + (2-\mu)^2 \right)
  \sn_A r \sn_C r W(\hat L, \hat \Lb, \hat L, \hat \Lb)
  = n^{-6} \hat \varrho \sn_A r \sn_C r.
\end{align*}
By using  Lemma \ref{4.28.17.16}, we can derive that
\begin{align*}
I_2 &= \frac{1}{4 } n^{-3} \sn_C r \left[(2-\mu) W(\hat \Lb, \Sigma e_A, \hat L, \Sigma N)
  - \mu W(\hat L, \Sigma e_A, \hat \Lb, \Sigma N) \right]\\
& = \frac{1}{2 }(\mu-1 ) n^{-5} \hat \varrho \sn_C r
  \l \Sigma e_A, \Sigma N \r= \f12 n^{-8}\hat \varrho\varpi^2\sn_C r\sn_A r.
\end{align*}
By the similar argument we can obtain
\begin{align*}
I_3
&= -\frac{1}{2}\mu(2-\mu)  n^{-4} \hat \varrho \l\Sigma e_A, \Sigma e_C\r
= -\f12 (1-n^{-2}\varpi^2) n^{-4}\hat \varrho\l \Sigma e_A, \Sigma e_C\r, \displaybreak[0]\\
I_4
& = -\frac{1}{4}(2-\mu-\mu) n^{-5} \hat \varrho \sn_A r \l \Sigma N, \Sigma e_C\r
 = \f12 n^{-8}\hat \varrho\varpi^2\sn_C r\sn_A r, \displaybreak[0]\\
I_5 &= -n^{-4} \hat\varrho \left(\Sigma N\wedge \Sigma e_A\right)\c \left(\Sigma N\wedge \Sigma e_C\right), \displaybreak[0]\\
I_6 &= -\frac{1}{2} n^{-4} \sn_A r \sn_C r W(\hat L, \Sigma N, \hat \Lb, \Sigma N)
= \frac{1}{2}n^{-6} \hat \varrho \sn_A r\sn_C r |\Sigma N|^2.
\end{align*}
By using the above expressions for $I_i, i=1, \ldots, 6$, we can obtain from Theorem \ref{4.10.6}
the desired estimate on $\ab$ given in (\ref{3.18.0}).

For $\a_{AC} := W(L, e_A, L, e_C)$ we may use the same argument for treating $\ab_{AC}$,
 which implies $\a_{AC}=\sum_{i=1}^6 \ti I_i$, where these $\ti I_i$ can be easily obtained from $I_i$ by swapping $\mu$ with $2-\mu$, and changing $\Sigma N$ to $-\Sigma N$.
Then it is clear that $\ti I_i=I_i$ for $i=1, \cdots 6$.
This shows that $\a=\ab$ and thus we obtain the estimate on $\a$ in (\ref{3.18.0}).

Next we consider $\udb_A := \f12 W(\Lb, e_A, L, \Lb)$. By using  (\ref{3.18.3.16}) and  (\ref{2.14.3.16}), we decompose $\udb$ as follows
\begin{align*}
2 \udb_A & = J_1 + J_2 + J_3 + J_4 +J_5,
\end{align*}
where
\begin{align*}
J_1 & = W\left(\frac{n^{-1}}{2}\left[\mu\hat L+(2-\mu)\hat \Lb\right],
\frac{1}{2} n^{-2}\sn_A r(\hat L-\hat \Lb), \right.\\
& \qquad \qquad \qquad \left. \frac{n^{-1}}{2} \left[(2-\mu)\hat L+\mu\hat \Lb\right],
 \frac{n^{-1}}{2} \left[\mu\hat L+(2-\mu)\hat \Lb\right]\right), \displaybreak[0]\\
J_2 &=  W\left(\frac{n^{-1}}{2}\left[\mu\hat L+(2-\mu)\hat \Lb\right],\Sigma e_A,
\frac{n^{-1}}{2}\left[\mu\hat L+(2-\mu)\hat\Lb\right], -\Sigma N\right), \displaybreak[0]\\
J_3 &= W\left(\frac{n^{-1}}{2}\left[\mu \hat L+(2-\mu)\hat \Lb\right], \Sigma e_A,
 \frac{n^{-1}}{2} \left[(2-\mu)\hat L+\mu \hat \Lb\right], -\Sigma N\right), \displaybreak[0]\\
J_4 &= W\left(-\Sigma N, \f12 n^{-2}\sn_A r(\hat L-\hat \Lb), \frac{n^{-1}}{2} \left[(2-\mu)\hat L+\mu\hat \Lb\right], -\Sigma N\right), \displaybreak[0]\\
J_5 &= W\left(-\Sigma N, \f12 n^{-2} \sn_A r(\hat L-\hat \Lb),  \frac{n^{-1}}{2} \left[\mu \hat L+(2-\mu) \hat\Lb\right], -\Sigma N\right).
\end{align*}
Direct calculation shows that
\begin{equation}\label{3.19.3.16}
J_1=\frac{1}{4} n^{-5} (\mu-(2-\mu)) \sn_A r W(\hat L, \hat \Lb, \hat L, \hat \Lb)=-2n^{-6}\varpi\sn_A r \hat \varrho.
\end{equation}
For $J_2$ and $J_3$, we first note that by using Lemma \ref{4.28.17.16}
\begin{align}
J_2 + J_3 &=W\left(\frac{n^{-1}}{2}\left[\mu\hat L+(2-\mu)\hat \Lb\right], \Sigma e_A, n^{-1}(\hat L+\hat \Lb), -\Sigma N\right)\nn\\
&=-\frac{n^{-2}\mu}{2}W(\hat L, \Sigma e_A, \hat \Lb, \Sigma N)+\frac{n^{-2}}{2}(\mu-2)W(\hat \Lb, \Sigma e_A, \hat L, \Sigma N)\nn\\
 &= \frac{\mu+2-\mu}{2} n^{-4}\hat\varrho \l\Sigma e_A, \Sigma N\r=- n^{-6}\varpi \hat \varrho \sn_A r.\label{3.20.2.16}
\end{align}
Similarly we can derive by using Lemma \ref{4.28.17.16} that
\begin{align*} 
J_4 + J_5=W\left(-\Sigma N, \f12 n^{-2}\sn_A r(\hat L-\hat \Lb), \frac{n^{-1}}{2}(\hat L+\hat \Lb), -\Sigma N\right)=0
\end{align*}
Combining this with (\ref{3.19.3.16}) and (\ref{3.20.2.16}) we therefore obtain
\begin{equation}\label{3.20.4.16}
\udb_A=-\frac{3}{2}n^{-6}\varpi\hat \varrho\sn_A r.
\end{equation}
which together with Theorem \ref{4.10.6} shows the estimate on $\udb$ in (\ref{3.18.0}).

For $\b_A := \frac{1}{2} W(L, e_A, L, \Lb)$, we can use the similar argument to derive that
\begin{align*}
 2\b_A &= W(L, e_A, L, \Lb)\\
& = W\left(\frac{n^{-1}}{2}\left[(2-\mu)\hat L+\mu\hat \Lb\right],
   \f12 n^{-2}\sn_A r(\hat L-\hat\Lb), \right.\\
   & \qquad \qquad \qquad \left. \frac{n^{-1}}{2}\left[(2-\mu)\hat L+\mu\hat \Lb\right],
   \frac{n^{-1}}{2} \left[\mu\hat L+(2-\mu)\hat \Lb\right]\right) \displaybreak[0]\\
& \quad \, + W\left(\frac{n^{-1}}{2}\left[(2-\mu)\hat L+\mu\hat \Lb\right], \Sigma e_A,
\frac{n^{-1}}{2}\left[\mu\hat L+(2-\mu)\hat\Lb\right], -\Sigma N\right)\displaybreak[0]\\
& \quad \, + W\left(\frac{n^{-1}}{2}\left[(2-\mu)\hat L+\mu\hat\Lb\right], \Sigma e_A,
 \frac{n^{-1}}{2}\left[(2-\mu)\hat L+\mu\hat\Lb\right], -\Sigma N\right) \displaybreak[0]\\
   &\qquad \qquad-J_4-J_5\\
&= -3n^{-6} \varpi  \sn_A r\hat\varrho.
\end{align*}
Indeed, by straightforward checking, the sum of the second and the third term is the same as $J_2+J_3$, and the first term equals $J_1$. We also employed $J_4+J_5=0$ to get the last identity. This shows that $\b = \udb$.

Finally we show that $\sigma=0$. Recall $2\sigma \ep_{AC}=W_{AC34}$, we only need to show $W_{AC34}=0$.
By using (\ref{3.18.3.16}),  (\ref{2.14.3.16}) and Lemma \ref{4.28.17.16} we have
\begin{align*}
W_{AC43} &= W(e_A, e_C, L, \Lb) = \emph{I} + \emph{II} +\emph{III},
\end{align*}
where
\begin{align*}
\emph{I} &= W\left(\f12 n^{-2}\sn_A r(\hat L-\hat \Lb), \f12 n^{-2}\sn_C r(\hat L-\hat \Lb), \right.\\
& \qquad \qquad  \left.\frac{n^{-1}}{2}\left[(2-\mu)\hat L+\mu\hat \Lb\right],
\frac{n^{-1}}{2}\left[\mu\hat L+(2-\mu)\hat\Lb\right]\right), \displaybreak[0]\\
\emph{II} &= W\left(\Sigma e_A, \f12 n^{-2}\sn_C r(\hat L-\hat\Lb),
  \frac{n^{-1}}{2}\left[(2-\mu)\hat L+\mu\hat \Lb\right]+\Sigma N, \right.\\
&\qquad \qquad \left.\frac{n^{-1}}{2}\left[\mu\hat L+(2-\mu)\hat \Lb\right]-\Sigma N \right), \displaybreak[0]\\
\emph{III} &= W\left(\f12 n^{-2}\sn_A r(\hat L-\hat \Lb), \Sigma e_C,
\frac{n^{-1}}{2}\left[(2-\mu)\hat L+\mu\hat\Lb\right]+\Sigma N, \right.\\
& \qquad \qquad \left. \frac{n^{-1}}{2} \left[\mu\hat L+(2-\mu)\hat \Lb\right]-\Sigma N\right).
\end{align*}
It is clear that $\emph{I}=0$.  By direct calculation we have
\begin{align*}
\emph{II} &=\f12 n^{-2}\sn_C r W\left(\Sigma e_A, \hat L-\hat\Lb, \frac{n^{-1}}{2}\left[(2-\mu)\hat L+\mu\hat \Lb\right], -\Sigma N\right)\\
&\quad\, +\f12 n^{-2}\sn_C r W\left(\Sigma e_A,\hat L-\hat\Lb, \frac{n^{-1}}{2}\left[\mu\hat L+(2-\mu)\hat \Lb\right], -\Sigma N\right)=0
\end{align*}
since $W(\Sigma e_A, \hat L-\hat \Lb, \hat L+\hat \Lb, \Sigma N)=0$.
By the same argument we can show that $\emph{III}=0$. Hence $\sigma=0$.

Next, we prove (\ref{2.14.1.16}) and (\ref{2.14.2.16}). We note that  by  using (\ref{dcp_2})
\begin{equation*}
W(S, e_A, S, e_C)=\frac{\ub^2}{4}\a_{AC}+\frac{u^2}{4} \ab_{AC}-\frac{\rho^2}{2}\varrho \delta_{AC}, \quad \rho W(e_A, \fB, \Nb, \fB)=-\frac{1}{4} (u\udb_A+\ub\b_A).
\end{equation*}
By using (\ref{1.4.7.16}) in $Z^s$,
\begin{equation}\label{5.15.2.16}
\rho^2\les t,
\end{equation}
we therefore conclude in $Z^s\cap \{t\ge T\}$, by using the estimates  in  (\ref{3.18.0})
\begin{align*}
&|W(S, e_A,S,e_C)|\les  \ve t^{-3}(t+\ve^2 t^{-2} ) \les \ve t^{-2},\\
&|\rho W(e_A, \fB, \Nb, \fB)|\les \ve t^{-4}
\end{align*}
as desired.  Thus the proof is complete.
\end{proof}

We will employ (\ref{2.20.3.16})-(\ref{2.20.5.16}) to prove (\ref{5.3.3.16}), which will be proved
simultaneously together with the stronger estimate in (\ref{2.20.1.16}) for $k_{\Nb A}$.

\begin{proof}[Proof of (\ref{2.20.1.16}) and (\ref{5.3.3.16})]
We first note that for any point $p$ in $Z^s$,  $\sf{ p}$ is fully contained in $Z^s$ when $t\ge T$, due to Remark \ref{4.9.2.16}.

In view of (\ref{zba}) and $t\approx \tir$ in (\ref{4.9.12}), we have
\begin{equation}\label{2.20.14.16}
|\zb|\les\ve t^{-4} \mbox { if } t\ge T \mbox{ in } Z^s
\end{equation}
since $|\sn \log n|\les \ve \bt^{-4}$ due to (\ref{4.9.16}), (\ref{4.16.1}) and (\ref{3.20.lap}), as well as $ \pt_{\bi\bj}=0 $, in $Z^s$.

In view of Proposition \ref{1.2.1.16} and (\ref{5.15.2.16}), there hold in $Z^s$
\begin{equation*} 
|\rho^2 \Eb_{AC},\, \rho^2 \Eb_{\Nb\Nb},\, \rho t^2 \Eb_{\Nb C}|\les \ve \bt^{-2}.
\end{equation*}
Combining this with Lemma \ref{dcp_sc} shows that
\begin{equation}\label{2.20.10.16}
|\rho^2\widehat{\bR}_{\fB\Nb \fB \Nb},\, \rho^2 \widehat{\bR}_{\fB A\fB C},\, \rho t^2 \widehat{\bR}_{\fB \Nb \fB C}|\les \ve \bt^{-2}  \mbox{ in } Z^s.
\end{equation}
Here we used the fact that $Z^s$ is a vacuum region.

In order to prove (\ref{2.20.1.16}) in $Z^s$, for $t\ge T$, in view of Proposition \ref{4.27.1.26},
we can  make the following bootstrap assumptions
\begin{equation}\label{2.20.11.16}
\left|t \rho(\hk_{\Nb \Nb}, \tr k-\frac{3}{\rho}, \hk_{AC}), t^2 \hk_{\Nb A}\right|\le \dn.
\end{equation}
As a direct consequence of (\ref{2.20.11.16}),  for any $p$ in $Z^s$, the integral along
$\sf{p}=\exp_\O(\rho V)$ from $q=\sf{p}\cap \{t=T\}$, where $V\in {\mathbb H}_1$, we have
\begin{equation}\label{2.20.16.16}
\int_{\rho(q)}^\rho \left|\tr k-\frac{3}{\rho'}\right|d\rho'\les \dn,
\end{equation}
and
\begin{equation}\label{2.20.12.16}
\left|\rho^2(\hk\hot \hk)_{\Nb\Nb}, \rho^2 (\hk\hot \hk)_{AC}, \rho t (\hk\hot \hk)_{\Nb A}, \rho^2 |\hk|^2\right|
\les \dn^2 \bt^{-2},
\end{equation}
where the definition of $\hot$ can be found in (\ref{hotd}), and (\ref{2.20.12.16}) can be obtained in view of the symbolic identities in (\ref{4.28.9.16}) and (\ref{2.20.11.16}). To obtain (\ref{2.20.16.16}), we also employed (\ref{3.20.1}) and $\bb, n\approx 1 $ in $Z^s$ due to Proposition \ref{prl_1}.

By using (\ref{2.20.14.16}), (\ref{2.20.12.16}) and (\ref{2.20.10.16}), the terms in (\ref{2.20.3.16})-(\ref{2.20.5.16}) verify
\begin{equation*} 
|\rho^2 G_{\Nb \Nb}, \rho^2 G_{AC}, \rho t G_{\Nb A} |\les (\ve+\dn^2 )\bt^{-2}
\end{equation*}
We consider the transport equations (\ref{2.20.3.16})-(\ref{2.20.5.16}), which symbolically are recast below for   $\H_\rho$ tangent tensor fields
\begin{equation}\label{2.20.13.16}
\sn_\fB F+\frac{2}{3} \tr k F =G.
\end{equation}
For any $p\in Z^s$, by using Lemma \ref{tsp1} (2), (\ref{3.20.1}) and (\ref{2.20.16.16}),  we integrate (\ref{2.20.3.16})-(\ref{2.20.5.16}) along $\sf{p}$. By virtue of $\bb^{-1}\approx 1, \, t\approx \tau$ in (\ref{3.21.16}), also using Proposition \ref{4.27.1.26}, we can obtain
\begin{align*}
&|\rho \tau \hk_{\Nb\Nb}|\les\ve+ \int_T^t |{\rho'}^2  G_{\Nb\Nb}| d {t'} \les \ve+\dn^2,     \\
&|\rho \tau \hk_{AC}|\les \ve+\int_T^t |{\rho'}^2 G_{AC}| d{t'} \les  \ve+\dn^2, \\
&|\tau^2 \hk_{\Nb A} |\les \ve+ \int_T^t |\rho' t' G_{\Nb A}| d{t'} \les \ve+\dn^2.
\end{align*}
Similarly integrating (\ref{3.14.1}) along $\sf{p}$, with the help of (\ref{lhk1}) and  (\ref{2.20.12.16}),  gives
$$
\left|\rho \tau (\tr k-\frac{3}{\rho})\right|
\les\ve+ \int_T^t {\rho'}^2\left(|\bR_{\fB\fB}|+|\hk|^2 \right) d{t'}\les \ve+\dn^2
$$
where $\bR_{\fB\fB}=0$ since $Z^s$ is a vacuum region.

Due to $t\approx \tau$, we can summarize the above four estimates as
\begin{equation*}
\left|t \rho(\hk_{\Nb \Nb}, \tr k-\frac{3}{\rho}, \hk_{AC}), t^2 \hk_{\Nb A}\right|\le C(\ve+\dn^2).
\end{equation*}
With  $\dn= 3C \ve$    and $\dn <\frac{1}{2C}$, (\ref{2.20.11.16}) can be improved to be bounded by $\frac{5}{6}\dn$, since $C(\ve+\dn^2)<\frac{5}{6}\dn$ holds in this situation.
\end{proof}

\section{\bf On the region of excision}\label{wza}

In this section, we will prove Proposition \ref{12.29.2}.
For this purpose, we consider the part on $\H_{\rho_*}$ contained in the schwarzschild zone,
foliated by the optical function $\hat u$, where $1\le \hat u\le \hat u_1$. We will obtain
Proposition \ref{12.29.2} by proving the following result.

\begin{proposition}\label{12.30.1}
Let $\bar t=\frac{1}{|S|}\int_S t d\mu_{\ga_S}$  with $S=S_{\rho, {\hat u}}$ and $\ga_S$ the induced metric on $S$. There holds
\begin{equation}\label{12.25.1}
\underset{{S_{\rho, \hat u}}}{\emph{osc}} (t):=\max_{S_{\rho, \hat u}}|t-\bar t|\les \ve^\f12.
\end{equation}
As a consequence, for $\rho>T$ sufficiently large,
\begin{align}\label{12.30.2}
&t_{\max}(S_{\rho, \hat u_1})<t_{\min}(S_{\rho, \hat u}), \mbox{ if }1\le \hat u<\hat u_1.
\end{align}
If $\rho\le \rho_*,\, t_*<t<t^*$ and $\hat u>1$, then
\begin{equation}\label{12.30.3}
  \hat u(S_{t, \rho})<\hat u_1.
 \end{equation}
\end{proposition}

Indeed, Proposition \ref{12.29.2} follows from  (\ref{12.30.3}) immediately.
In order to derive the estimate (\ref{12.25.1}), we use ${}^\dag\sn$ to denote the covariant derivative
on $S_{\rho, \hat u}$. Then
\begin{equation}\label{12.30.5.1}
\underset{S_{\rho, \hat u}}{\mbox{osc}} (t)
\le \|{}^\dag\sn t\|_{L^\infty(S_{\rho, \hat u})} \mbox{diam}(S_{\rho, \hat u}).
\end{equation}
Therefore, we need to estimate  $\|{}^\dag\sn t\|_{L^\infty(S_{\rho, \hat u})}$ and
$\mbox{diam}(S_{\rho, \hat u})$. For this purpose, we first give the geometric set-up for
the $\hat u$ foliation of the part of $\H_{\rho_*}$ in the schwarzschild zone. We define
\begin{equation}\label{12.07.1}
L^s =n^{-2} \p_t+\p_r \quad \mbox{and} \quad \Lb^s=n^{-2}\p_t-\p_r.
\end{equation}
Clearly, due to (\ref{4.25.2.16})
\begin{equation}\label{12.10.07}
\hat L=n^2 L^s, \quad \hat \Lb=n^2 \Lb^s,
\end{equation}
Thus $L^s, \Lb^s$ form a null pair and $\l L^s, \Lb^s\r = -2n^{-2}$.
 Moreover, one can use (\ref{4.3.2}) to show that $\bd_{L^s} L^s =0$.

 We define the null second fundamental forms in terms
of $ L^s $ and $\Lb^s$ in Schwarzschild zone by
\begin{equation}\label{12.10.09}
\uda{s}\chi(X, Y)=\l \bd_X L^s, Y\r; \quad\uda{s}\chib(X, Y)=\l\bd_X  \Lb^s, Y\r
\end{equation}
for any $S_{t, \hat u}$-tangent vector fields $X$ and $Y$. Similarly, we can introduce the null second
fundamental forms $\uda{{\widehat{}}}\chi$ and $\uda{{\widehat{}}}\chib$ relative to $\hat L$ and $\hat \Lb$.
We also introduce
\begin{equation}\label{12.07.2}
-{\dfa}^{-1}=\l \fB, L^s\r \quad \mbox{ and } \quad {}^\dag\!\Nb=\frac{-\nabb \hat u}{|\nabb \hat u|}
\end{equation}
which are the lapse function and the radial normal  of the $\hat u$-foliation on $\H_\rho$ respectively.
We first prove the following result which includes the estimate on $|{}^\dag\sn t|$.

\begin{lemma}\label{12.28.31}
Let $\rho$ be sufficiently large. There hold on $\H_\rho\cap \{1\le \hat u\le \hat u_1\}$ that
\begin{align}
& {\dfa}^{-1} = - a^{-1} n^{-2} (\varpi-n) + n^{-1}\frac{u}{\rho}, \label{12.07.3} \displaybreak[0]\\
& {}^\dag\!\Nb =\dfa L^s-\fB, \label{12.07.4} \displaybreak[0]\\
&\left|\dfa-\frac{n\rho}{u}\right|\les \frac{\ve}{\l\rho\r^3\bt}, \label{12.07.8} \displaybreak[0]\\
&\left|{}^\dag\!\Nb(t)-n^{-1}\frac{\tir}{\rho}\right| \les \frac{\ve}{\l\rho\r^3\bt}, \label{12.08.6} \displaybreak[0]\\
&\left|{}^\dag\sn t\right|_\bg \les \frac{\ve^\f12}{\l \rho\r^2}. \label{12.08.5}
\end{align}
\end{lemma}

\begin{proof}
In what follows the metric $\bg$ is actually the Schwarzchild metric since we only consider the part
of $\H_{\rho}$ that is fully contained in the schwarzschild zone. We will frequently  use the facts
\begin{equation}\label{4.22.1.16}
u\gtrsim 1 \quad \mbox{ and } \quad \rho^2=u\ub \gtrsim \ub \gtrsim 1 \quad  \mbox{ if  }1\le  \hat u\le \hat u_1,
\end{equation}
where we used (\ref{3.28.4.16}) with sufficiently small $\ve$. By using (\ref{12.07.1}),
Lemma \ref{frames} and (\ref{eun}) we have
\begin{align*}
-{\dfa}^{-1} & = \l \fB, L^s\r =\left\l a^{-1}\bN+\frac{\bb^{-1}t}{\rho}n^{-1}\p_t, n^{-2}\p_t +\p_r \right\r \displaybreak[0]\\
&=a^{-1}\l \bN, \p_r\r + \frac{\bb^{-1}t}{\rho} n^{-3}\l \p_t, \p_t\r
=a^{-1} n^{-2} \bN(r) -\frac{\bb^{-1}t}{\rho}n^{-1} \displaybreak[0] \\
& =a^{-1} n^{-2} (\varpi-n) +\left(a^{-1}-\frac{\bb^{-1}t}{\rho}\right)n^{-1} \displaybreak[0]\\
 &=a^{-1} n^{-2} (\varpi-n) - n^{-1}\frac{u}{\rho},
\end{align*}
where we used $a^{-1} = \tilde{r}/\rho$ and $u =b^{-1} t-\tilde{r} $ in the last step. This shows (\ref{12.07.3}).

Next, we prove (\ref{12.07.8}). In view of (\ref{12.07.3}) we have
\begin{align*}
\dfa &=\frac{n\rho}{u}\frac{1}{1-\frac{\tir}{u}n^{-1}(\varpi-n)}.
\end{align*}
According to (\ref{4.9.8}) and $n\approx 1$, the term $\frac{\tir}{u}n^{-1}(\varpi-n)$ is small for sufficiently small $\ve$ so that
we can derive that
\begin{align*}
\dfa =\frac{ n\rho}{u} \left(1+O\left(\frac{\tir}{u}n^{-1} (\varpi-n)\right) \right)
 = \frac{n\rho}{u}+ O\left(\frac{\rho\tir}{u^2} (\varpi-n)\right).
\end{align*}
Note that $\rho/u^2 = \underline{u}^2/\rho^3$, $\underline{u}\approx t$ and $\tilde{r}\approx t$,
we obtain (\ref{12.07.8}) by using (\ref{4.9.8}) and (\ref{4.22.1.16}).

Notice that
\begin{equation*} 
-\nabb \hat u=-\bd \hat u -\fB^\mu \bd_\mu \hat u \fB=L^s+\l L^s, \fB\r \fB=L^s-\dfa^{-1}\fB.
\end{equation*}
This together with the facts $\l \fB, \fB\r = -1$ and $\l L^s, L^s\r =0$ implies that $|\nabb \hat u|=\dfa^{-1}$.
Therefore ${}^\dag \!\Nb=\dfa L^s-\fB$ which is (\ref{12.07.4}).
Hence, by virtue of   (\ref{bt}), (\ref{12.07.1}), $\rho^2 = u \underline{u}$ and $\underline{u} = b^{-1} t+ \tilde{r}$,
we obtain
\begin{equation}\label{12.07.7}
\begin{split}
{}^\dag \!\Nb (t) = \dfa L^s(t)-\fB(t)=\dfa n^{-2}-\frac{\bb^{-1}n^{-1}t}{\rho}
=n^{-2}\left(\dfa-\frac{n\rho}{u}\right)+\frac{\tir}{n\rho}
\end{split}
\end{equation}
which together with (\ref{12.07.8}) shows (\ref{12.08.6}).

Finally, we consider the angular derivative ${}^\dag \sn t$ which can be written as
\begin{equation*}
{}^\dag \sn t= \bd t+\fB(t) \fB- {}^\dag \!\Nb(t)\, {}^\dag \!\Nb
= n^{-1}\left(-\bT+\frac{\bb^{-1}t}{\rho}\fB\right) - {}^\dag \!\Nb(t)\, {}^\dag\! \Nb,
\end{equation*}
where for the second equality we used (\ref{gt}) and (\ref{bt}). Therefore
\begin{align*}
\l {}^\dag \sn t, {}^\dag \sn t\r
& = n^{-2} \left\l -\bT + \frac{\bb^{-1} t}{\rho} \fB, -\bT + \frac{\bb^{-1} t}{\rho} \fB \right\r + {}^\dag \!\Nb(t)^2 \\
& \quad \, - 2 n^{-1}\, {}^\dag \!\Nb(t) \left\l  -\bT + \frac{\bb^{-1} t}{\rho} \fB, {}^\dag\!\Nb \right\r.
\end{align*}
By using (\ref{eq_1}) and noting that $\l \fB, {}^\dag\! \Nb\r = 0$ and $\l \bT, {}^\dag\!\Nb\r =-n \, {}^\dag\! \Nb(t)$, the latter of which follows from (\ref{gt}),
we have
$$
\l {}^\dag \sn t, {}^\dag \sn t\r
= n^{-2} \left(-1 + \frac{\bb^{-2} t^2}{\rho^2}\right)- {}^\dag \!\Nb(t)^2
= \frac{\tilde{r}^2}{n^2 \rho^2} - {}^\dag \!\Nb(t)^2.
$$
In view of (\ref{12.07.7}), (\ref{12.07.8}), $n\approx 1$ and $\tilde{r} \approx t$ we obtain
\begin{align*}
\l {}^\dag \sn t, {}^\dag \sn t\r = -n^{-1} \left(\dfa n^{-1}-\frac{\rho}{u}\right)
\left[n^{-1} \left(\dfa n^{-1}-\frac{\rho}{u}\right) + \frac{2\tilde{r}}{n\rho}\right]
= O\left(\frac{\ve}{\l\rho\r^4}\right)
\end{align*}
which shows (\ref{12.08.5}).
\end{proof}

Next we will derive the estimate on $\mbox{diam}(S_{\rho, \hat u})$ by estimating the
Gaussian curvature on $S_{\rho, \hat u}$.  We start from a preliminary result which includes
a proof of Lemma \ref{4.28.17.16}(1).

\begin{lemma}\label{12.10.21}
\emph{(i)} The traceless parts of $\uda{s}\chi$ and $\uda{s}\chib$ vanish with the traces given by
\begin{equation}\label{12.10.12}
\tr \, \uda{s} \chi= \frac{2}{r+2M} \quad \mbox{ and } \quad \tr \, \uda{s}\chib=-\frac{2}{r+2M}.
\end{equation}
The  Gaussian curvature $K$ on $S_{t, \hat u}$ verifies
\begin{equation}\label{12.10.13}
 K=(r+2M)^{-2}.
\end{equation}

\emph{(ii)} Relative to the null decomposition in terms of $\hat L, \hat\Lb$, in $Z^s$
the only nonvarnishing component of the Weyl tensor $W$ is $\hat\varrho=\frac{1}{4}W(\hat L, \hat \Lb, \hat L, \hat\Lb)$
with
\begin{equation*}
n^{-4}\hat \varrho=-\frac{4 M}{(r+2M)^3}.
\end{equation*}
\end{lemma}

\begin{proof}
(i) According to (\ref{scm}), the Gaussian curvature $K$ on $S_{t, \hat u}$ is a constant which,
by the Gauss-Bonnet theorem, is given by (\ref{12.10.13}).

Let $\ga$ be the induced metric on $S_{t, \hat u}$ and let $\mu_\ga$ be the associated area form,  we have
\begin{equation}\label{12.10.14}
\hat L \mu_\ga =n^2\tr \, \uda{s}\chi \, \mu_\ga \quad \mbox{and} \quad \hat \Lb \mu_\ga =n^2 \tr\, \uda{s}\chib \, \mu_\ga.
\end{equation}
Note that $\mu_\ga=(r+2M)^2 \mu_{\mathbb{S}^2}$, one may use (\ref{12.10.14}), (\ref{12.10.07}), (\ref{12.07.1})
and $n^2 = \frac{r-2M}{r+2M}$ to derive (\ref{12.10.12}) immediately. In particular, (\ref{12.10.12}) implies that
\begin{equation}\label{12.10.15}
\f12 \tr\,\uda{{\widehat{}}}\chib=-\frac{r-2M}{(r+2M)^2} \quad
\mbox{and} \quad \f12 \tr \, \uda{{\widehat{}}}\chi=\frac{r-2M}{(r+2M)^2}.
\end{equation}

(ii) Because $\bg$ in $Z^s$ is a Schwarzchild matric,  in terms of
the canonical null tetrad $\{n^{-1}\hat L, n^{-1}\hat \Lb, \hat e_A, A=1,2\}$, where $\{\hat e_A, A=1,2\}$ is the orthonormal frame on $S_{t, \hat u}$,  all components of the Weyl tensor $W$ in Definition \ref{3.18.19} vanish except $\hat \varrho$. We may use (\ref{gauss}) to determine $\hat \varrho$. In fact,
note that $\l L^s, \hat \Lb\r=-2$, $\{L^s, \hat \Lb, \hat e_A, A =1, 2\}$ also form a canonical null tetrad on $S_{t, \hat u}$,
where $\{\hat e_A, \hat A=1,2\}$ is an orthonormal frame therein, we may use (\ref{gauss}) to obtain
\begin{equation*}
K=-\frac{1}{4}\tr\, \uda{{\widehat{ }}}\chib\tr\, \uda{s}\chi+\f12 \uda{s}\chih\c \uda{{\widehat{}}}\chibh-n^{-4}\hat \varrho;
\end{equation*}
the source term in (\ref{gauss}) disappears because $Z^s$ is a vacuum region. Consequently, by using $\uda{s}\chih=0$,
(\ref{12.10.12}), (\ref{12.10.13}) and (\ref{12.10.15}), we obtain
$n^{-4}\hat \varrho =-4 M/(r+2M)^3$.
\end{proof}

 \begin{proposition}\label{gcomp_1}
For $1\le \hat u\le \hat u_1$ and $T\le \rho\le \rho_*$, let $K$ and $\emph{diam}(S_{\rho, \hat u})$ denote the Gaussian
curvature and the diameter of $S_{\rho, \hat u}$ respectively. Then
\begin{equation}\label{12.28.3}
\left|K-\frac{n^2}{(r+2M)^2}\right|\les \ve \frac{n^2}{(r+2M)^2}
\end{equation}
and
\begin{equation}\label{12.28.2}
\emph{diam}(S_{\rho, \hat u})\les r_{\max}(S_{\rho, \hat u}),
\end{equation}
where, in Schwartzchild zone, $r(p)$ denotes the coordinate value of the  point $p=(t,r, \omega), \omega\in {\mathbb S}^2$  in the standard polar coordinates, and $r_{\max}(S_{\rho, \hat u})$ denotes the maximal value of $r(p)$ over $S_{\rho, \hat u}$.
\end{proposition}

\begin{proof}
(\ref{12.28.2}) follows from (\ref{12.28.3}) as an application of Bonnet-Myers theorem, see \cite{Docarmo} for instance.
Hence we only need to show (\ref{12.28.3}). We will use the Gauss equation (\ref{gauss}). For this purpose
$S_{\rho, \hat u}$ is regarded as a $2$-sphere embedded in $Z^s$ with the normal vector fields ${}^\dag L$
and ${}^\dag \Lb$ given by
\begin{equation}\label{12.28.4}
{}^\dag L=\fB+{}^\dag\Nb=\dfa L^s, \qquad {}^\dag\Lb =\fB-{}^\dag\Nb=2\fB-\dfa L^s.
\end{equation}
It is straightforward to see that
\begin{equation*} 
\l {}^\dag L, {}^\dag \Lb\r=\l \dfa L^s, 2\fB-\dfa L^s\r=-2.
\end{equation*}
Let $\{{}^\dag e_{A}, A=1,2\}$ be an orthonormal frame on $S_{\rho, \hat u}$.  Then
$\{{}^\dag L, {}^\dag\Lb, {}^\dag e_{A}, A=1,2\}$ form a canonical null tetrad.  We define
\begin{equation*} 
{}^\dag  \chi_{AC}:=\l \bd_{{}^\dag e_A} {}^\dag L, {}^\dag e_C\r, \qquad
{}^\dag \chib_{AC}:=\l \bd_{{}^\dag e_A} {}^\dag \Lb, {}^\dag e_C\r
\end{equation*}
We claim
\begin{equation}\label{12.28.7}
{}^\dag \chi_{AC}=\frac{1}{2} \dfa \, \tr\uda{s}\chi\delta_{AC}.
\end{equation}
To see (\ref{12.28.7}), we will decompose $\{{}^\dag e_A\}_{A=1}^2$ in terms of $ \hat L, \hat \Lb, \{\hat e_A\}_{A=1}^2$,
where $\{\hat e_A\}_{A=1}^2$ is an orthonormal frame on $S_{t, \hat u}$.

Since $\l {}^\dag e_A, {}^\dag L\r =0$, we have $\l {}^\dag e_A, \hat L \r =0$. Therefore we can
decompose ${}^\dag e_A$ uniquely as
\begin{equation*} 
{}^\dag e_A=\Sigma {}^\dag e_A + f \hat L,
\end{equation*}
where $f$ is a scalar function and $\Sigma {}^\dag e_A$ is an $S_{t, \hat u}$-tangent vector field.
Since $\{{}^\dag e_A\}_{A=1}^2$ is orthonormal, we have
\begin{align*}
\delta_{AC}=\l {}^\dag e_A, {}^\dag e_C\r
&=\left\l \Sigma {}^\dag e_A + f \hat L, \Sigma {}^\dag e_C + f\hat L \right\r
 = \left\l \Sigma {}^\dag e_A, \Sigma {}^\dag e_C \right\r.  
\end{align*}
Recall that $\bd_{L^s} L^s=0$. We may use $\hat L = n^2 L^s$ in (\ref{12.10.07}) and Lemma  \ref{12.10.21}(i) to derive that
\begin{align*}
{}^\dag\chi_{AC} &=\dfa\l \bd_{{}^\dag e_A} L^s, {}^\dag e_C\r
=\dfa \left\l \bd_{\Sigma{}^\dag e_A} L^s, \Sigma {}^\dag e_C + f\hat L \right\r\nn \displaybreak[0]\\
& =\dfa \left\l \bd_{\Sigma{}^\dag e_A} L^s, \Sigma {}^\dag e_C \right\r
= \dfa \, \uda{s}\chi (\Sigma {}^\dag e_A, \Sigma {}^\dag e_C) \nn \displaybreak[0]\\
&=\frac{1}{2} \dfa \, \tr\uda{s}\chi \l \Sigma{}^\dag e_A, \Sigma{}^\dag e_C\r
=\frac{1}{2} \dfa \, \tr\uda{s}\chi \delta_{AC} 
\end{align*}
which shows (\ref{12.28.7}).

From (\ref{12.28.7}) it follows that the traceless part of ${}^\dag \chi$ vanishes.  Combining this
with the vacuum property in $Z^s$, we can infer from (\ref{gauss}) that
\begin{equation}\label{12.28.14}
K=-\frac{1}{4} \tr{}^\dag \chi\tr{}^\dag \chib-\frac{1}{4}W({}^\dag L, {}^\dag\Lb, {}^\dag L, {}^\dag\Lb).
\end{equation}
By using (\ref{12.28.4}), the first identity in (\ref{12.31.8}), Lemma \ref{4.28.17.16} and (\ref{12.10.07}),
we can see that
\begin{align}\label{12.28.16}
&W({}^\dag L, {}^\dag \Lb, {}^\dag L, {}^\dag\Lb)=4 \left(\dfa\right)^2 W(L^s,\fB, L^s, \fB) \nn \displaybreak[0]\\
&=\frac{4\left(\dfa\right)^2}{\rho^2} W\left(L^s, \f12(\cir{u}\hat \Lb+\cir{\ub}\hat L)+\tir \Sigma N,
L^s, \f12(\cir{u}\hat \Lb+\cir{\ub}\hat L)+\tir \Sigma N \right) \nn \displaybreak[0]\\
&=\left(\frac{\cir{u}\dfa}{\rho}\right)^2 W(L^s,\hat \Lb, L^s, \hat \Lb)
=4\left(\frac{\cir{u}\dfa}{\rho} \right)^2 n^{-4}\hat \varrho.
\end{align}
By using (\ref{12.10.19}) for $\cir{u}$, (\ref{12.07.8}) for $\dfa$ and (\ref{4.9.8}) we deduce that
\begin{align*}
\frac{\cir{u}\dfa}{\rho}
& = n^{-1} \left(u + (1-n^{-1} \varpi) \tir\right) \left(\frac{n}{u}+O\left(\frac{\ve}{\l\rho\r^4 \bt}\right)\right) \displaybreak[0]\\
& = n^{-1}\left(u+O\left(\frac{\ve}{\bt^3}\right)\right) \left(\frac{n}{u}+O\left(\frac{\ve}{\l\rho\r^4\bt}\right)\right) \displaybreak[0]\\
& = 1 + O\left(\frac{\ve}{u t^3}\right)+O\left(\frac{u\ve}{\rho^4 t}\right).
\end{align*}
Since $u \ub =\rho^2$ and $\ub\approx t$, it follows that
$$
\frac{\cir{u}\dfa}{\rho} = 1 + O\left(\frac{\ve}{\l \rho\r^2 \bt^2}\right).
$$
Therefore, we may use (\ref{12.28.16}), (\ref{12.10.10}) and $M\les \ve$ to deduce that
\begin{equation}\label{12.28.24}
W({}^\dag L, {}^\dag \Lb, {}^\dag L, {}^\dag\Lb)/(\frac{n}{r+2M})^2=O \left(\frac{\ve}{r+2M}\right)
= O\left(\frac{\ve}{t}\right).
\end{equation}

Next we show that
\begin{equation}\label{12.28.17}
\tr \, {}^\dag \chib+\frac{2nu}{\rho(r+2M)}=O\left(\frac{\ve}{\l\rho\r\bt}\right).
\end{equation}
By using (\ref{12.28.7}) we have
\begin{align}
{}^\dag \chib_{AC}
& = \left\l \bd_{{}^\dag e_A} (2\fB-{}^\dag L), {}^\dag e_C \right\r
= 2 k({{}^\dag e_A, {}^\dag e_C})-{}^\dag\chi_{AC}\nn \displaybreak[0]\\
&=2 \hat k({{}^\dag e_A, {}^\dag e_C}) +\left(\frac{2}{3}\tr k - \f12\dfa \, \tr\uda{s}\chi \right) \delta_{AC}\nn \displaybreak[0]\\
&=2 \hk({{}^\dag e_A, {}^\dag e_C}) + \frac{2}{3}\left(\tr k-\frac{3}{\rho}\right)\delta_{AC}
+ \left(\frac{2}{\rho}-\f12\dfa \, \tr\uda{s}\chi\right)\delta_{AC}. \label{12.28.21}
\end{align}
In view of Lemma \ref{12.10.21}(i), (\ref{12.07.8}), $\rho^2 = u \ub$, $r \approx t$ and $M\les \ve$ we have
\begin{align*}
\frac{2}{\rho}-\f12 \dfa \tr\uda{s}\chi
& =\frac{2}{\rho}-\frac{1}{r + 2M} \left(\frac{n\rho}{u}+O\left(\frac{\ve}{\l\rho\r^3\bt}\right)\right) \displaybreak[0]\\
&=\frac{1}{\rho} \left(2-\frac{n\ub}{r+2M}\right) + O\left(\frac{\ve}{\l\rho\r^3\bt^2}\right)\\
&=\frac{2r - n\ub}{\rho(r+2M)} + O\left(\frac{\ve}{\l\rho\r \bt}\right).
\end{align*}
Recall that $\ub = u + 2 \tir$, we may use (\ref{1.4.1.16}) and $r\approx t$ to obtain
\begin{align*}
\frac{2}{\rho}-\f12 \dfa \tr\uda{s}\chi
&=\frac{-nu+ 2(r-n \tir)}{\rho(r+2M)}+ O\left(\frac{\ve}{\l\rho\r \bt} \right)
=-\frac{nu}{\rho (r+2M)}+O\left(\frac{\ve}{\l\rho\r\bt}\right).
\end{align*}
Combining this with (\ref{12.28.21}) and using the estimates
$
t\rho |(\hk, \tr k-3/\rho)|\les \ve
$
established in (\ref{2.20.1.16}), we can obtain (\ref{12.28.17}).

Note that (\ref{12.28.7}) implies $\tr \, {}^\dag \chi = \dfa \, \tr {}^\dag \uda{s}\chi$.
Therefore, by using (\ref{12.28.17}), (\ref{12.10.12}) and (\ref{12.07.8}), we derive
\begin{align}\label{12.28.22}
\tr{}^\dag \chi \c \tr{}^\dag \chib+\frac{4n^2}{(r+2M)^2}
& = \frac{4n^2}{(r+2M)^2} \left(1-\frac{\dfa u}{n\rho}\right) + \frac{\dfa}{r+2M} O\left(\frac{\ve}{\l\rho\r\bt}\right) \nn \displaybreak[0]\\
& = O\left(\frac{\ve}{\l\rho\r^2\bt^4}\right) + O\left(\frac{\ve}{\brho^2\bt}\right) \nn \displaybreak[0]\\
& = O\left(\frac{\ve}{\l\rho\r^2\bt}\right).
\end{align}
Combining (\ref{12.28.14}),  (\ref{12.28.24}), (\ref{12.28.22}), $t\approx r$, and noting that (\ref{4.22.1.16}) implies
$\rho^2\gtrsim \ub \approx t$,
  we finally obtain
\begin{equation*} 
K-\frac{n^2}{(r+2M)^2} = \frac{n^2}{(r+2M)^2}\left(O(\frac{\ve(r+2M)}{\rho^2})+O(\frac{\ve}{t})\right)
= \frac{n^2}{(r+2M)^2} O(\ve).
\end{equation*}
Thus (\ref{12.28.3}) is proved.
\end{proof}

\begin{proof}[Proof of Proposition \ref{12.30.1}]
 Due to (\ref{12.28.2}) and $t\approx r$ in (\ref{4.9.12}), we have
$\mbox{diam}(S_{\rho,\hat u})\les r_{\max}(S_{\rho, \hat u}) \les t_{\max}(S_{\rho, \hat u})$.
Thus it follows from (\ref{12.30.5.1}) and (\ref{12.08.5}) that
\begin{equation}\label{12.30.5}
\underset{S_{\rho, \hat u}}{\mbox{osc}} (t) \les t_{\max}(S_{\rho, \hat u}) \frac{\ve^\f12}{\l \rho\r^2}.
\end{equation}
This implies that ${\mbox{osc}}_{S_{\rho, \hat u}} (t) \les \ve^{1/2} t_{\max}(S_{\rho, \hat u})$
and hence $t_{\max}(S_{\rho, \hat u}) \approx t_{\min}(S_{\rho, \hat u})$ for sufficiently small $\ve$.
Thus $\ub\approx t \approx t_{\max}(S_{\rho, \hat u})$ on $S_{\rho, \hat u}$, combined with (\ref{4.22.1.16}),
 imply that $\rho^2 = u \ub \gtrsim \ub \approx t_{max}(S_{\rho, \hat u})$.
We can obtain from (\ref{12.30.5}) that $\mbox{osc}_{S_{\rho, \hat u}} (t) \les \ve^{1/2}$
which shows (\ref{12.25.1}).

Next we prove (\ref{12.30.2}). Let $p\in S_{\rho, \hat u_1}$ be the point that $t(p)$ achieves the
maximum on $S_{\rho, \hat u_1}$ and assume that $p$ has the standard polar coordinate
$(\rho, \hat u_1, \omega_0)$ with $\omega_0\in {\mathbb S}^2$. Then for the point on $S_{\rho, \hat u}$
with polar coordinate $(\rho, \hat u, \omega_0)$ we have
\begin{equation}\label{12.30.7}
t(\rho,\hat u, \omega_0)-t_{\max}(S_{\rho, \hat u_1})=\int_{\hat u}^{\hat u_1} {}^\dag\Nb(t) \dfa d\hat u'.
\end{equation}
By using (\ref{12.08.6}) and (\ref{12.07.8}),
\begin{align*}
\dfa {}^\dag\Nb(t)
= \left(\frac{n\rho}{u}+O(\frac{\ve}{\l\rho\r^3 \bt}) \right) \left(n^{-1}\frac{\tir}{\rho}+O(\frac{\ve}{\l\rho\r^3\bt}) \right)
= \frac{\tir}{u}+O(\frac{\ve}{\l\rho\r^4}).
\end{align*}
With $1\le \hat u < \hat u_1$ and (\ref{3.28.4.16}),  we have $0<u\les 1$, which gives
$\frac{\tir}{u}\approx \frac{\ub}{u}=(\frac{\rho}{u})^2\gtrsim \rho^2$ and hence $\dfa {}^\dag\Nb(t) \ge c \rho^2 -O(\ve)$
for some universal constant $c>0$. It then follows from (\ref{12.30.7}) that
\begin{equation}\label{12.30.9}
t(\rho, \hat u, \omega_0)-t_{\max}(S_{\rho, \hat u_1})\ge(  c\rho^2-O(\ve))(\hat u_1-\hat u).
\end{equation}
By using (\ref{12.25.1}), we thus obtain
\begin{equation*}
t_{\min}(S_{\rho, \hat u})+O(\ve^\f12)\ge t(\rho, \hat u, \omega_0) \ge t_{\max}(S_{\rho, \hat u_1})+(c\rho^2-O(\ve))(\hat u_1-\hat u).
\end{equation*}
Hence we conclude that, for $\rho\ge T$ and sufficient small $\ve$, there holds, with $\hat u<\hat u_1$,  that
$t_{\min}(S_{\rho, \hat u}) > t_{\max}(S_{\rho, \hat u_1})$ which gives (\ref{12.30.2}).

Finally we show (\ref{12.30.3}). Recall that $t_* = t_{\min}(S_{\rho_*, \hat u_0})$. With $\hat u=\hat u_0$ in  (\ref{12.30.2}), it follows that
\begin{equation*}
t^*\ge t\ge   t_*>t_{\max}(S_{\rho_*, \hat u_1}).
\end{equation*}
Due to (\ref{12.08.6}), we have ${}^\dag \Nb(t)>0$. This implies on $\H_\rho$, $t$ is a decreasing function of $\hat u$.
Therefore
\begin{equation*} 
\hat u(S_{t, \rho_*})<\hat u_1 \quad \mbox{ for } t_*\le t\le t^*.
\end{equation*}
Further, by using (\ref{3.25.2.16}) and (\ref{4.9.8}) we have
$$
\bN(\hat u)=-\bN(\ga(r))=-\ga'(r) \bN(r)=-n^{-2}\varpi<0
$$
which implies that on $\Sigma_t$, $\hat u$ is an increasing function of $\rho$. Consequently,
for $\rho\le \rho_*$ and $t_*<t<t^*$ we have $\hat u(S_{t, \rho}) \le \hat u(S_{t, \rho_*}) <\hat u_1$
which shows (\ref{12.30.3}).
\end{proof}

\section{\bf Hawking mass and Bondi mass}\label{mass}

In this section, we introduce the Hawking mass on $S_{t,\rho}=\Sigma_t\cap \H_\rho$ in $\I^+(\O)$
and investigate the asymptotic behavior as $t\rightarrow \infty$.

\begin{definition}\label{12.31.2}
Let $\ud r = \left(\frac{|S_{t,\rho}|}{4\pi} \right)^\f12$.
We define the Hawking mass enclosed by a $2$-surface $S_{t,\rho}$ to be
\begin{equation*}
m(t,\rho):=\frac{\ud r}{2}\left(1+\frac{1}{16\pi}\int_{S_{t,\rho}} \tr\chi \tr\chib \right).
\end{equation*}
If the Hawking mass $m(t,\rho)$ tends to a limit $M(\rho)$ as $t\rightarrow \infty$,
this limit is called the Bondi mass on $\H_\rho$.
\end{definition}

The main result of this section is the following.\begin{footnote}{The asymptotic behavior of Hawking mass along all hyperboloids does rely on the global existence of the foliation of proper time.}\end{footnote}

\begin{proposition}\label{2.9.1.16}
The Bondi mass $M(\rho)$ is well-defined  on each $\H_\rho$, and $$M(\rho)\equiv 2M.$$  More precisely there exists $t_s(\rho)$ sufficiently large so that for all $t>t_s(\rho)$,
\begin{equation}\label{12.31.4}
m(t,\rho)=2M+O(\ve^2 \bt^{-1}).
\end{equation}
\end{proposition}

We will rely on crucially the estimate of the Gaussian curvature $K$ on $S_{t,\rho}$.
Let $\{e_A, A=1, 2\}$ be an orthonormal frame on $S_{t, \rho}$. We may apply (\ref{gauss})
to the canonical null tetrad $\{L, \Lb, e_A, A=1,2\}$  to obtain
\begin{equation}\label{gauss1}
K+\frac{1}{4}\tr\chi\tr\chib-\f12 \chih_{AC}\chibh_{AC}=-\varrho+\f12\ga^{AC} S_{AC},
\end{equation}
where $\varrho = \frac{1}{4} W(L, \Lb, L, \Lb)$ and, according to (\ref{schouten}),
the last term on the right hand side, if non-vanishing, can be calculated as
\begin{align}\label{1.14.4.16}
\ga^{AC}S_{AC}
& = \ga^{AC}\bd_A \phi \bd_C \phi-\frac{1}{3} \left(\bd^\mu\phi \bd_\mu\phi-\fm\phi^2\right) \nn\\
& = \frac{1}{3}\fm\phi^2+\frac{2}{3}|\sn\phi|^2+\frac{1}{3}\bd_L \phi\bd_\Lb \phi.
\end{align}

\begin{lemma}\label{4.18.2.16}
Consider the region $\I^+(\O)$. On $S_{t, \rho}$ there hold \begin{footnote}{ We employ the result in the wave zone to indicate the difference on the rate of convergence for various geometric quantities in different region.}\end{footnote}
\begin{equation}\label{1.14.5.16}
|\tir^2 K-1|\les  \left\{\begin{array}{lll}
\ve \bt^{-\f12+3\d} & \mbox{ in } \I^+(\O),\\
\ve \bt^{-1}  & \mbox{  in }Z^s,
\end{array} \right.
\end{equation}
\item \begin{equation}\label{4.30.2.16}
\left|\frac{\tir}{\ud r}-1 \right| \les \left\{ \begin{array}{lll}
\ve \bt^{-\frac{1}{2}+3\d} & \mbox{ in  } \I^{+}(\O), \\
\ve \bt^{-1}  & \mbox{ in }Z^s
\end{array}\right.
\end{equation}
and
\begin{equation}\label{1.2.2.16}
\ud{r}\approx \tir, \quad\, \emph{diam}(S_{t,\rho}) \les \tir.
\end{equation}
\end{lemma}

\begin{proof}
We first prove (\ref{1.14.5.16}). Recall that $\chi_{AC} = \l \bd_A L, e_C\r$ and
$\chib_{AC}= \l \bd_A \Lb, e_C\r$. By using (\ref{dcp_2}) and $\Lb=2\bT-L$ we have
$L=\frac{1}{\tir}(\rho\fB- u\bT)$ and $\Lb =\frac{1}{\tir}(\ub\bT-\rho \fB)$. Therefore
\begin{align*} 
\chi_{AC} & 
= \frac{\rho}{\tir} k_{AC} +\frac{u}{\tir} \pib_{AC}
=\frac{\rho}{\tir} \left(\frac{1}{3} \tr k \d_{AC} + \hat{k}_{AC} \right) + \frac{u}{\tir} \pib_{AC},\\
\chib_{AC} & 
= -\frac{\rho}{\tir} k_{AC} - \frac{\ub}{\tir} \pib_{AC}
= -\frac{\rho}{\tir} \left(\frac{1}{3} \tr k \d_{AC} + \hat{k}_{AC} \right) - \frac{\ub}{\tir} \pib_{AC}.
\end{align*}
By taking the trace and the traceless part of $\chi$ and $\chib$ by the induced metric $\ga_{AC}$ on $S_{t,\rho}$,
we can obtain
\begin{align}\label{1.14.8.16}
\tr\chi = \frac{\rho}{\tir} \left(\frac{2}{3} \tr k - \ud\delta \right)+\frac{u}{\tir} \delta', \quad
\tr\chib = -\frac{\rho}{\tir} \left(\frac{2}{3}\tr k - \ud\delta\right)-\frac{\ub}{\tir} \delta'
\end{align}
and
\begin{align}\label{1.14.7.16}
\begin{split}
&\chih_{AC} = \frac{\rho}{\tir} \left(\hk_{AC}+\frac{1}{2}\ud \delta \ga_{AC}\right)
+\frac{u}{\tir} \left(\pib_{AC}-\f12\delta'\ga_{AC} \right), \\
&\chibh_{AC} = -\frac{\rho}{\tir} \left(\hk_{AC}+\frac{1}{2}\ud\delta \ga_{AC}\right)
+ \frac{\ub}{\tir} \left(-\pib_{AC}+\frac{1}{2}\delta'\ga_{AC} \right),
\end{split}
\end{align}
where $\ud\delta = \hat k_{\Nb\Nb}$ and $\delta'=-\pib_{\bN\bN}$ which were introduced in Lemma \ref{nbb}.

In order to proceed further, in Table \ref{table1} we list the decay estimates of certain geometric quantities in the regions
$\I^+(\O)$ and $Z^s$ respectively which will be proved shortly.

\begin{table}[h]
\begin{center}
\begin{tabular}{ c | c | c }
		$\dum$ & $\I^+(\O) $& $Z^s$ \\
 \hline
$ \tir^2 \tr\chi \tr\chib+4 $ & $\ve \bt^{-\f12+3\d}$ & $\ve \bt^{-1}$ \\[0.8ex]
$\tir^2 \chih_{AC}\c\chibh_{AC}$   & $\ve^2 \bt^{-1+6\d}$ & $\ve^2\bt^{-2}$ \\[0.8ex]
$\varrho$ &$\ve \bt^{-\frac{5}{2}}\rho^{\frac{3\d}{2}}$&$\ve\bt^{-3}$\\[0.8ex]
$\ga^{AC}S_{AC}$& $\ve^2 \bt^{-3+\d}$& 0\\
 \hline
\end{tabular}
\end{center}
\caption{}
\label{table1}
 \end{table}
By using (\ref{gauss1}), the decay estimates in Table \ref{table1}, $\tir \les t$ and $\rho\les t$, we can obtain
(\ref{1.14.5.16}) immediately.

It remains to prove the estimates in Table \ref{table1}.

By using (\ref{1.14.8.16}), $\rho^2 = u\ub$ and $u+\ub = 2 \bb^{-1} t$ it is straightforward to derive that
\begin{align*}
\tr\chi\tr\chib
& = \left[\frac{\rho}{\tir} \left(\frac{2}{3} \tr k - \ud\delta \right)+\frac{u}{\tir} \delta'\right]
\left[-\frac{\rho}{\tir} \left(\frac{2}{3}\tr k - \ud\delta\right)-\frac{\ub}{\tir} \delta' \right]\\
& = - \frac{\rho^2}{\tir^2} \left(\frac{2}{3} \tr k - \ud\delta \right)^2 +\frac{\rho^2}{\tir^2} \left(\delta'\right)^2
- \frac{2 \bb^{-1} t \rho}{\tir^2} \delta' \left(\frac{2}{3} \tr k - \ud\delta \right)
\end{align*}
Use the symbol $\Ab$  defined at the end of Section \ref{radiald}, we can write
$$
\frac{2}{3} \tr k - \ud\delta  = \frac{2}{3} \left(\tr k -\frac{3}{\rho}\right) + \frac{2}{\rho} - \ud\delta
= \Ab + \frac{2}{\rho}.
$$
Therefore, symbolically we have
\begin{align*}
\tir^2 \tr\chi\tr\chib
& = -\rho^2 \left(\Ab + \frac{2}{\rho}\right)^2 + \rho^2 \left(\delta'\right)^2 -2 \bb^{-1}t \rho \left(\Ab + \frac{2}{\rho}\right)\delta'\\
& = -4 -4\bb^{-1}t\delta' + \rho \Ab + \bb^{-1} t\rho {\delta'}\Ab +\rho^2 \Ab\c \Ab+\rho^2 \left(\delta'\right)^2.
\end{align*}
By (\ref{hk1}), (\ref{pt1}) and  (\ref{pt2}), in $(\I^+(\O)\cap \{t\ge T\})\backslash Z^s$, we have the estimates
\begin{equation}\label{1.17.1.16}
\rho |\Ab|\les \ve \bt^{-1}\rho^\d,\quad \bt|\delta'|\les \ve \bt^{-\frac{1}{2}+3\d},
\quad\rho |\pt_{\bi\bj}|_g\les\ve \bt^{-1/2+3\delta}.
\end{equation}
For $\{t\le T\}\cap \I^+(\O)$, by using (\ref{lpt1}) and (\ref{lhk1}),
\begin{equation}\label{1.17.3.16}
\rho |\Ab |+|\pt_{\bi\bj}|_g\les \ve
\end{equation}
 By (\ref{2.20.1.16}) we have the improved estimates in $Z^s$,
\begin{equation}\label{1.17.2.16}
 \rho |\Ab| \les \ve \bt^{-1}, \quad \pt_{\bi\bj}=0
\end{equation}
  These estimates together with the fact $\bb^{-1}\approx 1$ obtained in Proposition \ref{prl_1}
show that
\begin{equation*} 
|\tir^2 \tr\chi \tr\chib+4|\les\left\{\begin{array}{lll}
\ve \bt^{-\f12+3\d} &  \mbox{ in } \I^+(\O)\\
\ve \bt^{-1}  & \mbox{ in }Z^s
\end{array}\right.,
\end{equation*}
as recorded in Table \ref{table1}.

Next, by using (\ref{1.14.7.16}) we can calculate $\chih\c \chibh$ as
\begin{align*}
\chih_{AC}\c\chibh_{AC}
& = -\frac{\rho^2}{\tir^2} \left(\hk_{AC}+\frac{1}{2}\ud\delta\ga_{AC}\right)^2
-\frac{\rho^2}{\tir^2} \left(-\pib_{AC}+\frac{1}{2} \delta'\ga_{AC}\right)^2 \displaybreak[0]\\
&\quad \, + \frac{2\bb^{-1} t\rho }{\tir^2} \left(\hk_{AC}+\frac{1}{2}\ud\delta \ga_{AC}\right)
\left(-\pib_{AC}+\frac{1}{2}\delta'\ga_{AC}\right)
\end{align*}
which, in view of (\ref{6.6.3.16}), can be written symbolically  as
\begin{equation*}
\tir^2 \chih_{AC}\chibh_{AC} = \rho^2\left(\Ab^2+{\pt_{\bi\bj}}^2\right) + \rho t\Ab\c \pt_{\bi\bj}.
\end{equation*}
By using (\ref{1.17.1.16})-(\ref{1.17.2.16}), we can obtain
\begin{equation}\label{1.17.6.16}
|\tir^2 \chih_{AC}\chibh_{AC}| \les \left\{\begin{array}{lll}
\ve^2 \bt^{-1+6\d} & \mbox{ in }  \I^+(\O), \\
\ve^2\bt^{-2}  & \mbox{ in }  Z^s,
\end{array}\right.
\end{equation}

For the term of $\varrho$, by using (\ref{baloceng}) and Theorem \ref{2.15.2} (5) and Proposition \ref{1.2.1.16}, we have
\begin{equation*}
|\varrho|\les\left\{ \begin{array}{lll}
\ve \bt^{-3+\d}  & \mbox{ in } \I^+(\O) \cap \{\bb^{-1}t\le 3\rho\},\\
\ve \bt^{-\frac{5}{2}}\rho^{\frac{3\d}{2}}  & \mbox{ in } \I^+(\O) \cap \{\bb^{-1}t\ge 3\rho\}, \\
\ve \bt^{-3}  & \mbox{ in } Z^s.
\end{array}\right.
\end{equation*}

In $Z^s$ the Schouten tensor vanishes. Hence $\ga^{AC} S_{AC}=0$. By using (\ref{1.14.4.16}),
(\ref{4.30.3.16}) in Theorem \ref{2.15.2} and (\ref{baloceng}), we can also obtain
 \begin{equation*}
 \left|\ga^{AC}S_{AC}\right| \les  \ve^2 \bt^{-3+\d} \quad \mbox{ in } \I^+(\O).
 \end{equation*}
We thus obtain all the decay estimates in Table \ref{table1} and the proof of  (\ref{1.14.5.16}) is completed.

Next we prove (\ref{4.30.2.16}) and (\ref{1.2.2.16}). By (\ref{1.14.5.16}) and Bonnet-Myers theorem,  we have
$$
\mbox{ diam}(S_{t,\rho}) \les \tir_{\max}(S_{t,\rho}).
$$
Then we can obtain
\begin{equation}\label{4.17.2.16}
\underset{S_{t, \rho}}{\mbox{osc}} (\tir) \les \mbox{ diam}(S_{t,\rho})\sup_{S_{t,\rho}}|\sn \tir|
\les \tir_{\max}(S_{t, \rho}) \sup_{S_{t,\rho}}|\sn \tir|.
\end{equation}
We will use (\ref{3.6.4}) to estimate $|\sn \tir|$. To this end, we set
$\zeta_A :=k_{A\Nb}+\pib_{A\bN}$. By using the estimates (\ref{pt2}), (\ref{hk1}), (\ref{lkn}) and (\ref{lpt1}) in $\I^+(\O)$
and the estimates (\ref{2.20.1.16}) and (\ref{1.17.2.16}) in $Z^s$, we can obtain
\begin{equation*}  
|\zeta|\les\left\{ \begin{array}{lll}
\ve \bt^{-\frac{3}{2}+3\d}  & \mbox{ in  }\I^+(\O)\\
\ve \bt^{-2} & \mbox{ in } Z^s
\end{array}\right..
\end{equation*}
Thus, we may use (\ref{3.6.4}) and $\bb^{-1}\approx 1$ in Proposition \ref{prl_1} to derive that
\begin{equation*}
|\sn \tir| \les \left\{ \begin{array}{lll}
\ve \bt^{-\frac{1}{2}+3\d} & \mbox{ in  }\I^+(\O) \\
\ve \bt^{-1} & \mbox{ in }Z^s
\end{array}\right..
\end{equation*}
This, together with (\ref{4.17.2.16}), implies that
\begin{equation}\label{4.18.1.16}
\underset{S_{t, \rho}}{\mbox{osc}} \left(\frac{\tir}{\overline{\tir}}\right)
\les \left\{ \begin{array}{lll}
\ve \bt^{-\frac{1}{2}+3\d} & \mbox{ in  } \I^+(\O)\\
\ve \bt^{-1} & \mbox{ in }Z^s
\end{array}\right.
\end{equation}
where $\overline{\tir}$ denotes the average of $\tir$ over $S_{t, \rho}$. Note that, due to $\frac{\tir}{\overline{\tir}}\approx 1 $ which follows from (\ref{4.18.1.16}),
\begin{align*}
\left|\overline{\tir}^2 K-1\right|
& = \left|\tir^2 K -1\right| + \left|(\overline{\tir}^2-\tir^2) K\right|
\les \left|\tir^2 K-1\right|+\left|\frac{\tir}{\overline{\tir}}-1\right| \left|\tir^2 K\right|.
\end{align*}
In view of (\ref{1.14.5.16}) and (\ref{4.18.1.16}), we have
\begin{equation}\label{4.30.1.16}
\left|\overline{\tir}^2 K-1\right| \les \left\{ \begin{array}{lll}
\ve \bt^{-\frac{1}{2}+3\d} & \mbox{ in  } \I^+(\O), \\
\ve \bt^{-1}  & \mbox{ in }Z^s.
\end{array}\right.
\end{equation}
Note that the Gauss-Bonnet Theorem implies $\frac{1}{{\ud r}^2}
= \frac{1}{|S_{t, \rho}|} \int_{S_{t, \rho}} K d\mu_\ga$. Therefore
\begin{equation*}
\left|\frac{\overline{\tir}^2}{{\ud r}^2}- 1\right|
= \frac{1}{|S_{t,\rho}|} \left|\int_{S_{t, \rho}}\left(\overline{\tir}^2 K-1\right) d\mu_\ga\right|
\le \sup_{S_{t,\rho}} \left|\overline{\tir}^2 K- 1\right|
\end{equation*}
which, combined with  (\ref{4.30.1.16}), implies
\begin{equation*}
\left|\frac{\overline{\tir}^2}{\ud r^2}-1\right| \les \left\{ \begin{array}{lll}
\ve \bt^{-\frac{1}{2}+3\d} & \mbox{ in  } \I^+(\O), \\
\ve \bt^{-1} & \mbox{ in }Z^s.
\end{array}\right.
\end{equation*}
From this and (\ref{4.18.1.16}) we can obtain (\ref{4.30.2.16}). As an immediate consequence of (\ref{4.30.2.16}),
we can obtain (\ref{1.2.2.16}).  Hence the proof of Lemma \ref{4.18.2.16} is complete.
\end{proof}

We are ready to  prove Proposition \ref{2.9.1.16}.

\begin{proof}[Proof of Proposition \ref{2.9.1.16}]
By using (\ref{gauss1}), the Gauss-Bonnet theorem and the definition of $m(t, \rho)$ we have
\begin{equation*}
m(t, \rho) = \frac{\ud r}{8\pi} \left(4\pi+\frac{1}{4}\int_{S_{t,\rho}} \tr\chi \tr\chib d\mu_\ga\right)
= \frac{\ud r}{8 \pi} \int_{S_{t,\rho}} \left(\f12\chih\c\chibh- \varrho + \f12\ga^{AC}S_{AC}\right) d\mu_\ga.
\end{equation*}
For $t > t_{\max}(S_{\rho, \hat u_0})$, $\hat u$ decreases as $t$ increases, due to  (\ref{12.08.6}) which gives ${}^\dag \Nb(t)>0$.  This guarantees that $S_{t, \rho} \subset Z^s$ for $t > t_{\max}(S_{\rho, \hat u_0})$. Since $S_{AC}=0$ in $Z^s$, we have
\begin{equation*}
m(t, \rho) = \frac{\ud r}{8 \pi} \int_{S_{t,\rho}} \left(\f12\chih\c\chibh- \varrho \right) d\mu_\ga.
\end{equation*}
Noting that $|S_{t, \rho}|= 4\pi (\ud r)^2$, by using (\ref{1.2.2.16}) and $\tir \approx t$ we have $|S_{t, \rho}|\les t^2$.
Therefore, it follows from (\ref{1.17.6.16}) that
\begin{equation*}
\ud r \int_{S_{t,\rho}}\chih\c \chibh d\mu_\ga  = O(\ve^2 t^{-1})
\end{equation*}
for $t> t_{\max}(S_{\rho, \hat u_0})$. Consequently, we may use (\ref{12.31.7}) and (\ref{12.10.10}) to conclude
for $t> t_{\max}(S_{\rho, \hat u_0})$ that
\begin{align*}
m(t, \rho) &= O(\ve^2 t^{-1}) - \frac{\ud r}{8\pi} \int_{S_{t,\rho}} n^{-4} \hat \varrho(1+O(\ve t^{-4})) \nn\\
&=\frac{\ud r}{2\pi} M\int_{S_{t,\rho}} \frac{1}{(r+2M)^3} d\mu_\ga +  O(\ve^2 t^{-1}),
\end{align*}
where, for the second equality we used $M\les \ve$. Recall that $\tir \approx r \approx t$
and $n\approx 1$, we may use (\ref{1.4.1.16}) to obtain
\begin{align*}
\int_{S_{t, \rho}} \frac{1}{(r+ 2M)^3} d\mu_\ga
& = \int_{S_{t, \rho}} \tir^{-3} \left(n + \frac{r}{\tir} -n + \frac{2M}{\tir}\right)^{-3} d\mu_\ga\\
& = \int_{S_{t, \rho}} \tir^{-3} \left(n + O(\ve t^{-1})\right)^{-3} d\mu_\ga \\
& = \int_{S_{t, \rho}} (n \tir)^{-3} d\mu_\ga + O(\ve t^{-2}).
\end{align*}
Therefore
\begin{align*}
m(t, \rho) =\frac{\ud r}{2\pi} M\int_{S_{t,\rho}}  (n \tir)^{-3} d\mu_\ga +  O(\ve^2 t^{-1})
= \frac{M}{2\pi \ud r^2} \int_{S_{t,\rho}}  \left(\frac{\ud r}{n \tir}\right)^3 d\mu_\ga +  O(\ve^2 t^{-1}).
\end{align*}
By virtue of (\ref{4.30.2.16}), (\ref{3.20.lap}) and $M\les \ve$, we have
\begin{equation*}
\frac{\ud r}{n \tir}-1 = \left(\frac{1}{n}-1\right)+\frac{1}{n}\left(\frac{\ud r}{\tir}-1\right)
=O(\ve t^{-1}).
\end{equation*}
Consequently we can conclude that
\begin{align*}
m(t, \rho) &=\frac{M}{2\pi \ud r^2} \int_{S_{t,\rho}} \left(1 + O(\ve t^{-1})\right)^{-3} d\mu_\ga + O(\ve^2 t^{-1})\\
& = =\frac{M}{2\pi \ud r^2} \int_{S_{t,\rho}} d\mu_\ga + O(\ve^2 t^{-1})
=2 M +O(\ve^2 t^{-1})
\end{align*}
and the proof of Proposition \ref{2.9.1.16} is therefore complete.
\end{proof}

\noindent{\bf Acknowledgement.}  The results of this paper and \cite{Wang15} were reported in the workshop of  Mathematical Problems in General Relativity, in Simons Center for Geometry and Physics at Stony Brook in January  2015; in the  main conference of the program: ``General Relativity: A celebration of the 100th anniversary" in Institut Henri Poincar$\acute{e}$ in November 2015; and  were delivered as a minicourse in Oberwolfach Seminar: Recent Advances on the Global
Nonlinear Stability of Einstein Spacetimes in  May 2016. The author would like to thank the above institutes for the hospitality and the opportunities.

\end{document}